\documentclass[a4paper]{article}
\usepackage{makecell}
\usepackage{bbm}

\makeatletter
\let\c@figure\c@table
\let\ftype@figure\ftype@table

\usepackage{enumitem}
\usepackage{calc}

\usepackage{lmodern} 

\usepackage{graphicx} 

\usepackage{amsmath, accents}
\usepackage{amssymb,tikz}
\usepackage{pict2e}
\usepackage{amsthm}
\usepackage{amscd}
\usepackage{mathrsfs}
\usepackage{todonotes}
\usetikzlibrary{calc} 

\usepackage{yfonts}

\usepackage{marvosym}
\usepackage[percent]{overpic}

\newtheorem{theorem}{Theorem} [section]
\newtheorem{proposition}[theorem]{Proposition}	
\newtheorem{corollary}[theorem]{Corollary}
\newtheorem{lemma}[theorem]{Lemma}

\theoremstyle{definition}

\newtheorem{remark}[theorem]{Remark}

\newcommand{\C}{\mathbb{C}}

\newcommand{\R}{\mathbb{R}}
\newcommand{\N}{\mathbb{N}}
\newcommand{\PP}{\mathbb{P}}
\newcommand{\E}{\mathbb{E}}
\newcommand{\HH}{\mathbb{H}}
\newcommand{\re}{\text{\upshape Re\,}}
\newcommand{\im}{\text{\upshape Im\,}}

\newcommand{\erfc}{{\rm erfc}}

\let\oldbibliography\thebibliography
\renewcommand{\thebibliography}[1]{\oldbibliography{#1}
\setlength{\itemsep}{-0.5pt}}

\def\XXint#1#2#3{{\setbox0=\hbox{$#1{#2#3}{\int}$}
\vcenter{\hbox{$#2#3$}}\kern-.5\wd0}}

\usepackage{tikz}
\usetikzlibrary{arrows}
\usetikzlibrary{decorations.pathmorphing}
\usetikzlibrary{decorations.markings}
\usetikzlibrary{patterns}
\usetikzlibrary{automata}
\usetikzlibrary{positioning}
\usepackage{tikz-cd}
\tikzset{->-/.style={decoration={
				markings,
				mark=at position #1 with {\arrow{latex}}},postaction={decorate}}}
	
	\tikzset{-<-/.style={decoration={
				markings,
				mark=at position #1 with {\arrowreversed{latex}}},postaction={decorate}}}
				
\tikzset{
	master/.style={
		execute at end picture={
			\coordinate (lower right) at (current bounding box.south east);
			\coordinate (upper left) at (current bounding box.north west);
		}
	},
	slave/.style={
		execute at end picture={
			\pgfresetboundingbox
			\path (upper left) rectangle (lower right);
		}
	}
}

\usetikzlibrary{shapes.misc}\tikzset{cross/.style={cross out, draw, 
         minimum size=2*(#1-\pgflinewidth), 
         inner sep=0pt, outer sep=0pt}}

\allowdisplaybreaks	

\numberwithin{equation}{section}

\def\bigO{{\cal O}}

\makeatletter
\newcommand{\oset}[3][0ex]{%
  \mathrel{\mathop{#3}\limits^{
    \vbox to#1{\kern-2\ex@
    \hbox{$\scriptstyle#2$}\vss}}}}
\makeatother

\usepackage{tocloft}
\usepackage[colorlinks=true]{hyperref}
\hypersetup{urlcolor=blue, citecolor=red, linkcolor=blue}

\begin{document}
\title{Smallest gaps of the two-dimensional Coulomb gas}
\author{Christophe Charlier}

\maketitle

\begin{abstract}
We consider the two-dimensional Coulomb gas with a general potential at the determinantal temperature, or equivalently, the eigenvalues of random normal matrices. We prove that the smallest gaps between particles are typically of order $n^{-3/4}$, and that the associated joint point process of gap locations and gap sizes, after rescaling the gaps by $n^{3/4}$, converges to a Poisson point process. As a consequence, we show that the $k$-th smallest rescaled gap has a limiting density proportional to $x^{4k-1}e^{-\frac{\mathcal{J}}{4}x^{4}}$, where $\mathcal{J}=\pi^{2}\int \rho(z)^{3}d^{2}z$ and $\rho$ is the density of the equilibrium measure. This generalizes a result of Shi and Jiang beyond the quadratic potential. To establish this, we derive an explicit formula for the subleading term in the asymptotic expansion of the correlation kernel near the edge, which may be of independent interest.
\end{abstract}



{
\hypersetup{linkcolor=black}

}

\section{Introduction}

We consider the two-dimensional Coulomb gas defined by the probability measure
\begin{align}
& \frac{1}{Z_{n}}\prod_{1\leq j<k \leq n}|z_{j}-z_{k}|^{2} \prod_{j=1}^{n} e^{-n Q(z_{j})}d^{2}z_{j}, & & z_{1},\ldots,z_{n}\in \C, \label{general density intro}
\end{align}
where $Z_{n}$ is the normalization constant, $d^{2}z$ is the two-dimensional Lebesgue measure, and $Q:\C\to \R$ is the potential. The points $z_{1},\ldots,z_{n}$ can be interpreted as a system of electrons confined to a plane, repelling each other via Coulomb's law and under the influence of the external potential $Q$, at inverse temperature $\beta=2$. Another motivation for studying \eqref{general density intro} is that it also describes the joint distribution of the eigenvalues of random normal matrices $M$ (i.e. such that $MM^{*}=M^{*}M$) drawn with probability density proportional to $e^{-n\mathrm{Tr} \, Q(M)}dM$ (see e.g. \cite{EF2005,AHM2011} for more details). Hence, the results we present below can be interpreted in this random matrix context as well. 

\medskip A fundamental problem in the study of Coulomb gases is to understand how the system behaves as the number of particles $n$ becomes large. A classical result in this direction shows that the empirical measure $\frac{1}{n}\sum_{j=1}^{n}\delta_{z_{j}}$ converges almost surely, as $n\to + \infty$, to an equilibrium measure $\mu$, whose support is denoted $\mathcal{S}$ and called ``the droplet". The equilibrium measure provides a first-order description of the macroscopic distribution of the particles. Within the droplet, the repulsive interaction causes the particles to spread out in a highly regular manner, with typical spacing between neighbors of order $1/\sqrt{n}$. The purpose of this work is to obtain, under mild assumptions on $Q$, precise
results on the size and distribution of the smallest gaps between neighboring particles, in the regime
where $n\to + \infty$. This paper is inspired by the work of Shi and Jiang \cite{SJ2012}, who considered the case $Q(z) = |z|^{2}$.

\medskip In recent years, significant progress has been made in understanding the extremal gap statistics (for both the smallest and largest gaps) of various point processes. These include the CUE and GUE \cite{Vinson2001, BAB2013, FW2018}, the C$\beta$E \cite{FW2021}, the GOE \cite{FTW2019}, the GSE \cite{FLY2024}, the G$\beta$E \cite{AMR2022}, determinantal point processes with translation invariant kernels \cite{Soshnikov}, Wigner matrices \cite{B2022, LLM2020, Zhang}, as well as stationary Gaussian processes \cite{FGY2024}. In dimension two, beyond the aforementioned work \cite{SJ2012}, the only optimal results we are aware of are those by Feng and Yao \cite{FY2023} on the smallest gaps between zeros of Gaussian analytic functions, by Lopatto and Meeker \cite{LM2024} on the smallest gaps between bulk eigenvalues of real Ginibre matrices, and by Lopatto and Otto \cite{LO2025} on the largest gaps in the bulk of the complex Ginibre ensemble. The present paper can be viewed as a contribution to this line of research. We also mention that high-probability bounds for the smallest particle gap in the two-dimensional Coulomb gas, valid for general $\beta$ and $Q$, have been established in \cite{Ameur2018, MR2023, Thoma2024}. Closely related is the problem of estimating the probability of large spectral gaps in random normal matrices; this topic has also been widely studied, see e.g. \cite{ForresterHoleProba, C2025++, BP2024, AFLS2025} and the references therein. 

\medskip We assume throughout that $Q\in C^{2}(\C)$ and satisfies 
\begin{align}\label{growth of $Q$}
Q(z) \geq (2+\epsilon)\log |z|, 
\end{align}
for some $\epsilon >0$ and for all sufficiently large $|z|$. This condition ensures that \eqref{general density intro} is well-defined and that the equilibrium measure $\mu$, which is defined as the unique minimizer of 
\begin{align*}
\sigma \mapsto \int_{\C}\int_{\C} \log \frac{1}{|z-w|}d\sigma(z)d\sigma(w) + \int_{\C} Q(z)d\sigma(z)
\end{align*}
among all Borel probability measures $\sigma$ on $\C$, exists and is unique \cite[Theorem I.1.3]{SaTo}. Moreover, its support $\mathcal{S}$ is compact, and we have by \cite[Theorem II.1.3]{SaTo} that 
\begin{align}\label{def of mu and rho}
d\mu(z) = \rho(z)d^{2}z, \qquad \rho(z):=\frac{\Delta Q(z)}{4\pi} \mathbf{1}_{\mathcal{S}}(z),
\end{align}
where $\Delta = \partial_{x}^{2}+\partial_{y}^{2}$ is the standard Laplacian. Furthermore, there exists $\ell_{Q} \in \R$ such that 
\begin{align}
& Q(z)+2\int_{\C}\log \frac{1}{|z-w|}d\mu(w) = \ell_{Q}, & & z \in \mathcal{S}, \label{EL equality} \\
& Q(z)+2\int_{\C}\log \frac{1}{|z-w|}d\mu(w) \geq \ell_{Q}, & & z \in \mathbb{C}\setminus \mathcal{S}. \label{EL inequality}
\end{align}
We say that $Q$ is \textit{regular} if the inequality \eqref{EL inequality} is strict.

\medskip Following \cite{SJ2012}, we introduce a total order on $\C$ as follows: for $z_{1},z_{2}\in \C$, we write $z_{1} \prec z_{2}$ if either $\im z_{1} < \im z_{2}$ or if $\im z_{1} = \im z_{2}$ and $\re z_{1} < \re z_{2}$. We also write $z_{1} \preceq z_{2}$ if either $z_{1} \prec z_{2}$ or $z_{1}=z_{2}$.

\medskip Our main result concerns the asymptotic behavior of the smallest gaps between neighboring particles, both in terms of their size and their location in the plane. To make this precise, let $\HH_{+}:=\{u\in \C: \im u \geq 0\}$, and let $\{z_{i}\}_{i=1}^{n}$ be the points sampled from \eqref{general density intro}, indexed so that $z_{1}\preceq \ldots \preceq z_{n}$. 
We consider the following point process on $\HH_{+}\times\mathbb{C}$:
\begin{align}\label{1.1}
& \chi^{(n)} = \sum_{i=1}^{n-1}\delta_{(n^{3/4}(z_{i^*}-z_i), z_i)},
\end{align}
where $i^* = \arg\min_{j = i+1}^n\{|z_j-z_i|\}$.

 
\medskip We now state our first main result, which generalizes \cite[Theorem 1.1]{SJ2012} for a large class of potentials. 

\newpage 
\begin{theorem}\label{thm1}
Suppose $Q:\C\to \R$ is $C^{2}$-smooth on $\C$, satisfies the growth condition \eqref{growth of $Q$}, and is real analytic and strictly subharmonic (i.e. $\Delta Q>0$) in a neighborhood of $\mathcal{S}$. Suppose also that $Q$ is regular and that $\partial \mathcal{S}$ is a smooth Jordan curve.

As $n\to\infty$, the point process $\chi^{(n)}$ converges weakly to the Poisson point process $\chi$ on $\HH_{+}\times\mathbb{C}$ with intensity 
\begin{align}\label{intensity main thm}
& \mathbb{E}[\chi(B\times \Omega)] = \bigg(\pi^{2}\int_{\Omega\cap \mathcal{S}}\rho(v)^{3}d^{2}v\bigg)\bigg(\int_B|z|^2\frac{d^{2}z}{\pi}\bigg), 
\end{align}
for any bounded Borel sets $B\subset\HH_{+}$ and $\Omega\subset \mathbb{C}$.
\end{theorem}
\begin{remark}\label{rem:assumptions not needed}
We believe the assumptions that $Q$ is $C^{2}$-smooth on $\C$, regular, and that $\partial \mathcal{S}$ is a smooth Jordan curve are not essential and can be relaxed with some effort. 
\end{remark}

Let $t_1^{(n)}\leq \ldots \leq t_k^{(n)}$ be the $k$ smallest values of the rescaled gaps $\{n^{3/4}|z_i - z_j|: 1 \leq i < j \leq n\}$. For each $j=1,\ldots,k$, let $p_j^{(n)}<q_j^{(n)}$ be the (almost surely unique) indices such that
\begin{align*}
t_j^{(n)} = n^{3/4}|z_{q_j^{(n)}}-z_{p_j^{(n)}}|,
\end{align*}
and define $\theta_j^{(n)}\in[0,\pi)$ by
\begin{align*}
t_j^{(n)}e^{i\theta_j^{(n)}} = n^{3/4}(z_{q_j^{(n)}}-z_{p_j^{(n)}}).
\end{align*}
The following corollary extends \cite[Corollary 1.1]{SJ2012} by establishing the limiting joint distribution of $t_1^{(n)},\ldots,t_k^{(n)},\theta_{1}^{(n)},\ldots,\theta_{k}^{(n)}$ for general potential $Q$.

\begin{corollary}\label{col1}
Suppose $Q$ is as in Theorem \ref{thm1}. For any fixed $k\in \N_{>0}$, $0<x_1<y_1<\ldots<x_k<y_k$, and $0\leq\alpha_j<\beta_j\leq\pi$, $j=1,\ldots,k$, we have
\begin{multline*}
\lim_{n\to + \infty}\mathbb{P}(x_\ell<t_\ell^{(n)}<y_\ell, \; \alpha_\ell<\theta_\ell^{(n)}<\beta_\ell \mbox{ for all } 1\leq\ell\leq k) \\
= \big(e^{-\frac{\mathcal{J}}{4}x_k^4}-e^{-\frac{\mathcal{J}}{4}y_k^4}\big)\frac{\mathcal{J}^{k-1}}{4^{k-1}}\prod_{\ell=1}^{k-1}(y_\ell^4-x_\ell^4)\prod_{\ell=1}^{k}\frac{\beta_\ell-\alpha_\ell}{\pi},
\end{multline*}
where $\mathcal{J}:=\pi^{2}\int_{\mathcal{S}} \rho(z)^{3}d^{2}z$. 

\smallskip In particular, the $k$-th smallest rescaled gap satisfies
\begin{align*}
(t_k^{(n)},\theta_k^{(n)})
\oset{\mathrm{law}}{{\xrightarrow{\hspace*{0.55cm}}}}
(t_k,\theta_k) \qquad \mbox{as } n \to +\infty,
\end{align*}
where $t_{k}$ is a positive random variable with probability density
\begin{align}\label{density}
\PP\big(t_{k}\in [x,x+dx)\big) = \frac{\mathcal{J}^{k}}{(k-1)! \, 4^{k-1}} x^{4k-1}e^{-\frac{\mathcal{J}}{4}x^4}dx, \qquad x\in \R^{+}:=[0,+\infty),
\end{align}
and $\theta_k$ is uniformly distributed on $[0,\pi)$ and independent of $t_k$.
\end{corollary}

\begin{figure}[t]
\begin{center}
\begin{tikzpicture}[master]
\node at (0,0) {\includegraphics[width=4.3cm]{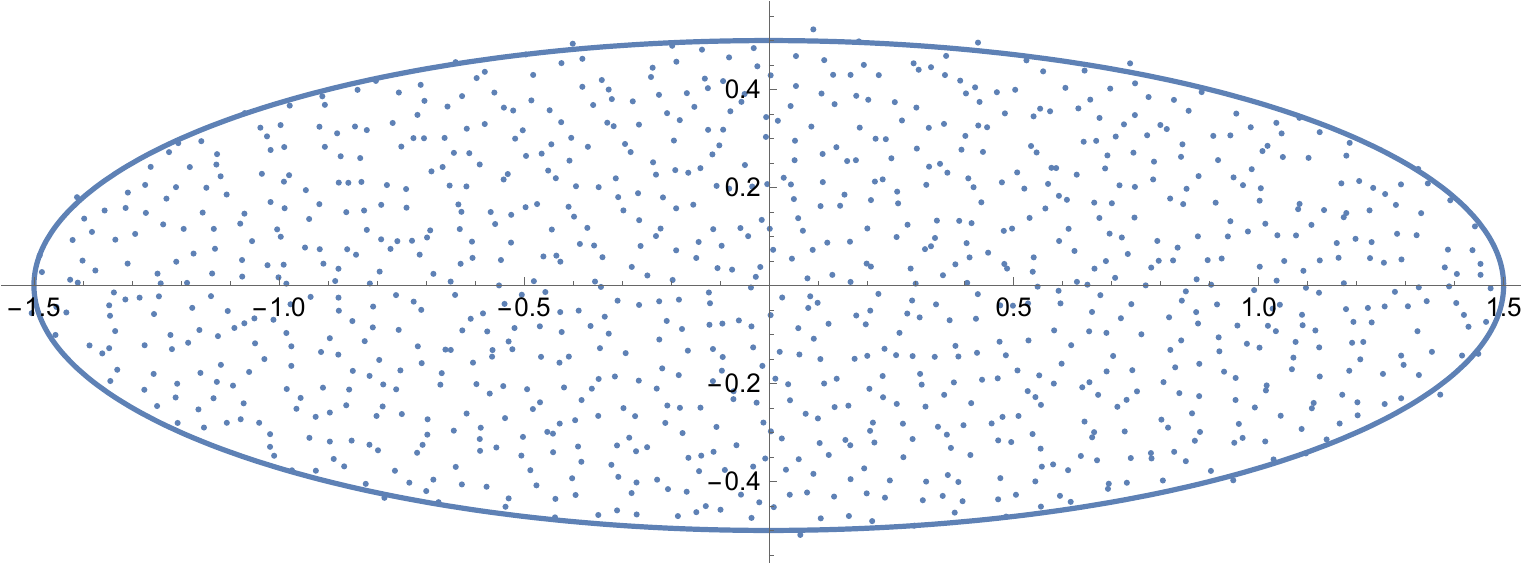}};
\node at (0,2.25) {\footnotesize Elliptic GinUE};
\end{tikzpicture}\hspace{-0.75cm}
\begin{tikzpicture}[slave]
\node at (0,0) {\includegraphics[width=3.4cm]{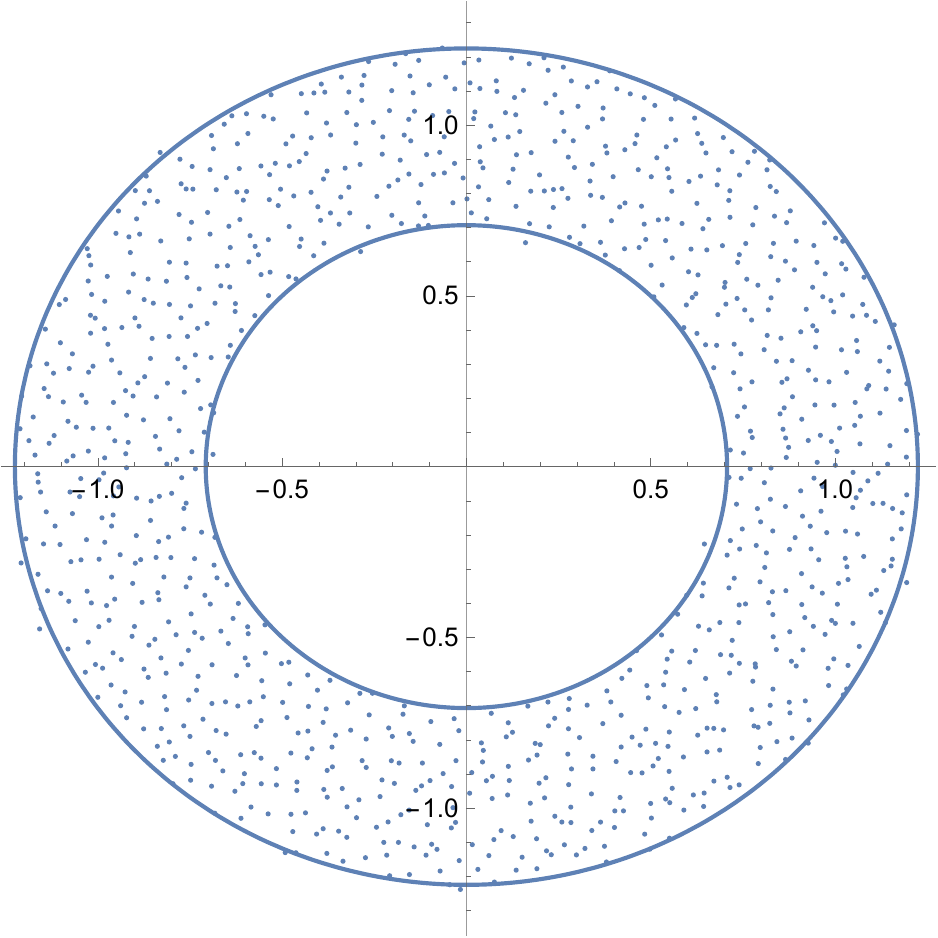}};
\node at (0,2.25) {\footnotesize Induced GinUE};
\end{tikzpicture}\hspace{-1.3cm}
\begin{tikzpicture}[slave]
\node at (0,0) {\includegraphics[width=3.4cm]{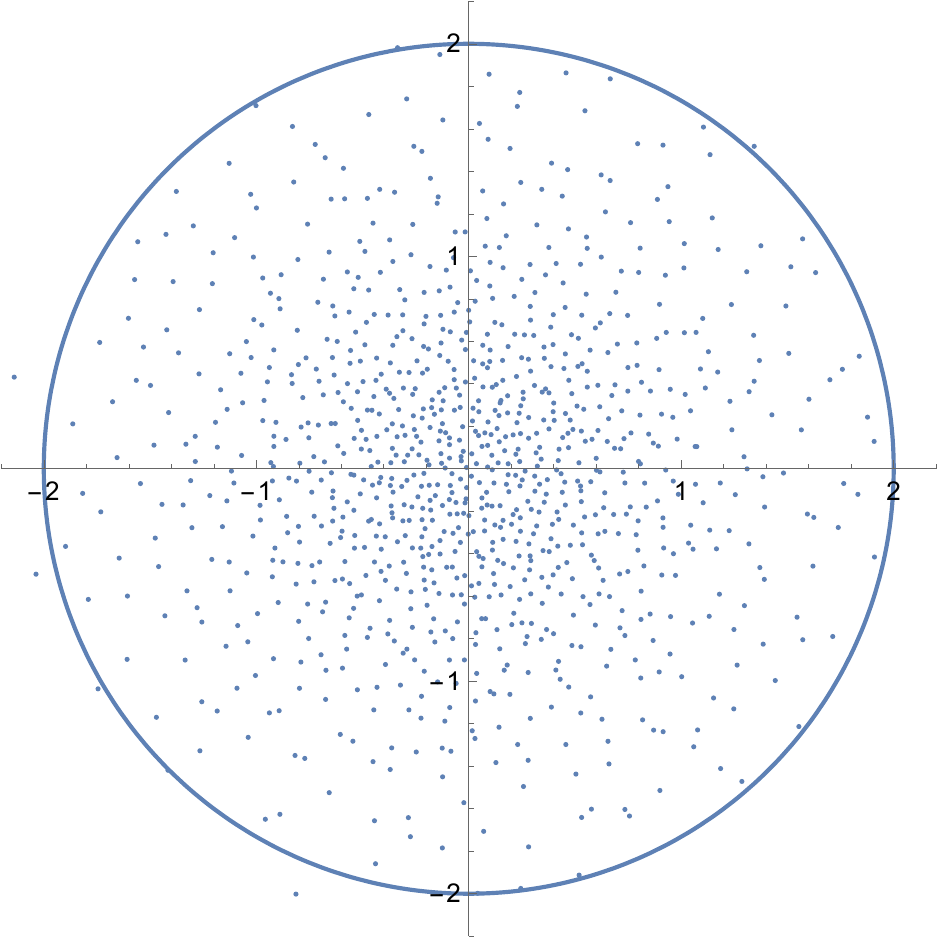}};
\node at (0,2.25) {\footnotesize Induced SrUE};
\end{tikzpicture}\hspace{-1.25cm}
\begin{tikzpicture}[slave]
\node at (0,0) {\includegraphics[width=3.4cm]{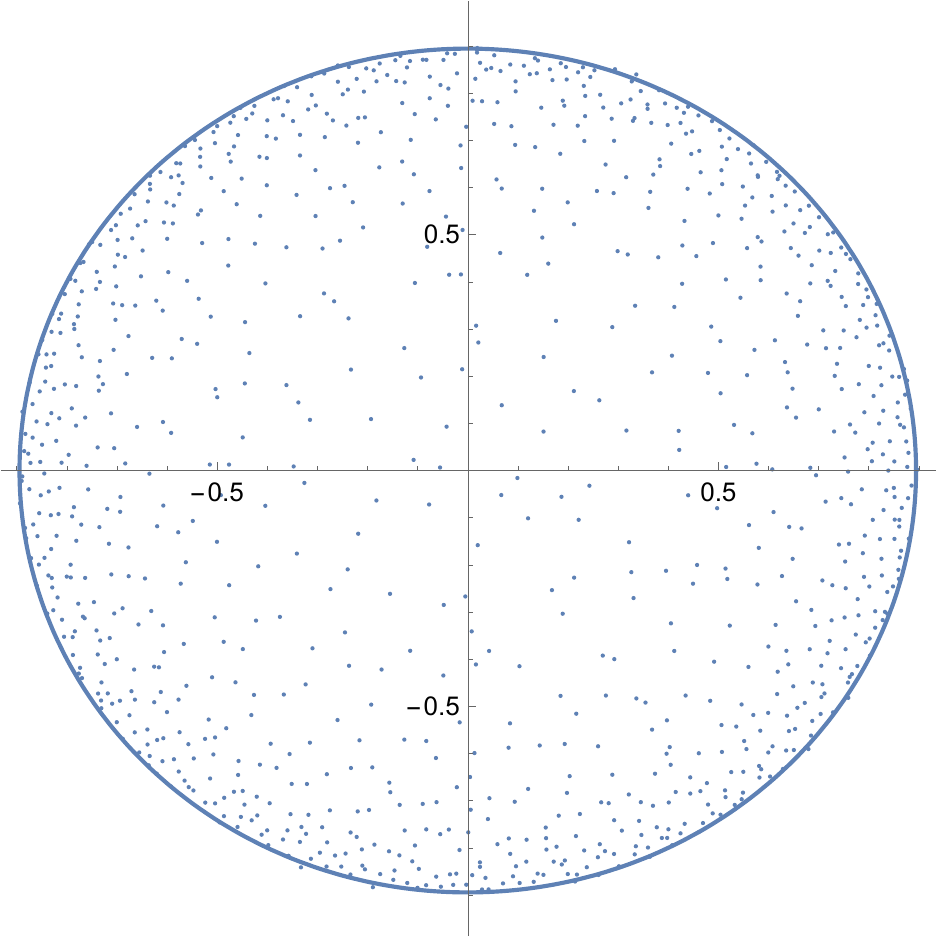}};
\node at (0,2.25) {\footnotesize TUE};
\end{tikzpicture}
\end{center}
\caption{Each plot shows a sample of \eqref{general density intro} with $n = 1000$ and a different potential $Q$: from left to right, $Q(z) = \frac{1}{1-\tau^{2}}(|z|^2-\tau \re z^{2})$ with $\tau=\frac{1}{2}$, $Q(z) = |z|^2 - 2a \log |z|$ with $a = \frac{1}{2}$, $Q(z) = a \log(1 + |z|^2)$ with $a = 1.25$, and $Q$ as in \eqref{Q TUE} with $a = 0.25$. The blue curves represent $\partial \mathcal{S}$. \label{fig:eigenvalues}}
\end{figure}
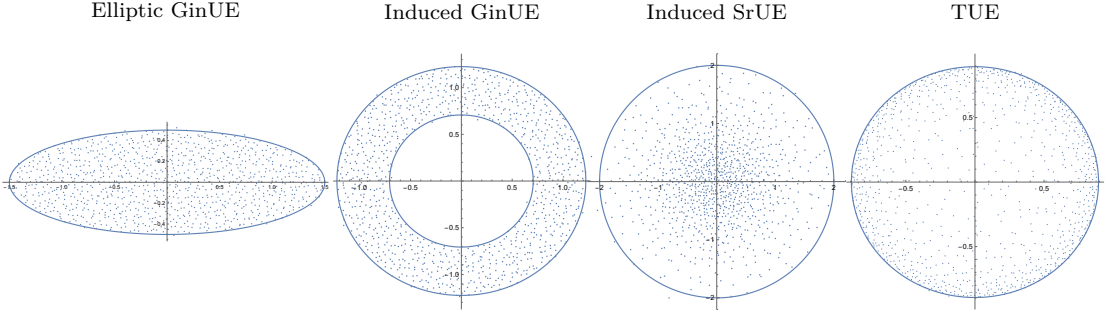

\subsubsection*{Numerical confirmations of Corollary \ref{col1}.}
In this subsection, we verify numerically Corollary \ref{col1} with $k=1,2$ for four random matrix models for which it is possible to make exact simulations (Figure \ref{fig:eigenvalues} shows typical eigenvalue configurations for these models). 
\vspace{-0.25cm}\paragraph{Smallest gaps of the elliptic Ginibre ensemble (elliptic GinUE).} A Ginibre matrix is a matrix with identically and independently distributed centered complex Gaussian entries. The elliptic Ginibre ensemble is a one-parameter deformation of the Ginibre ensemble defined as follows (see \cite[Section 2.3]{BFreview}): for $\tau \in [0,1)$, an $n\times n$ elliptic Ginibre matrix $J$ is of the form 
\begin{align*}
J = \sqrt{1+\tau}H_{1} + i \sqrt{1-\tau} H_{2}, \qquad \tau \in [0,1),
\end{align*}
where $H_{1},H_{2}$ are independent Hermitian matrices drawn from the Gaussian unitary ensemble (GUE) that can be constructed according to
\begin{align*}
H_{1} = \frac{1}{2}(G_{1}+G_{1}^{\dagger}), \qquad H_{2} = \frac{1}{2}(G_{2}+G_{2}^{\dagger}),
\end{align*}
with $G_{1}, G_{2}$ independent Ginibre matrices of size $n \times n$ whose entries have variance $\frac{1}{n}$, and $G_{1}^{\dagger},G_{2}^{\dagger}$ denote the conjugate transposes of $G_{1}$ and $G_{2}$, respectively. 
The eigenvalues of $J$ are distributed as \eqref{general density intro} with $Q(z)=\frac{1}{1-\tau^{2}}(|z|^2-\tau \re z^{2})$. In this case, the equilibrium measure has a constant density and is supported on an ellipse:
\begin{align*}
\mathcal{S}=\Big\{z \in \C: \Big( \frac{\re z}{1+\tau} \Big)^{2} + \Big( \frac{\im z}{1-\tau} \Big)^{2} \leq 1\Big\}, \qquad \rho(z) = \frac{1}{\pi(1-\tau^{2})}\mathbf{1}_{\mathcal{S}}(z), \qquad \mathcal{J} = \frac{1}{(1-\tau^{2})^{2}}.
\end{align*}
By Corollary \ref{col1}, $t_{k}$ has density proportional to $x^{4k-1}e^{\frac{-x^{4}}{4(1-\tau^{2})^{2}}}$. 
The histograms in Figure \ref{fig:histo} (left), obtained from the two smallest gaps $t_{1}^{(400)}, t_{2}^{(400)}$ computed across $10^4$ elliptic Ginibre matrices of size $400 \times 400$ with $\tau=\frac{1}{2}$, show convincing agreement with the theoretical predictions of Corollary \ref{col1}. 

\smallskip When $\tau = 1$, $J$ becomes Hermitian and Corollary \ref{col1} no longer applies. In this case, it is known \cite{Vinson2001,BAB2013} that the typical size of the smallest gaps is of order $n^{-4/3}$, in contrast to the $n^{-3/4}$ scaling obtained for $\tau\in [0,1)$. This transition is consistent with the fact that as $\tau \to 1$, the density of $t_{k}$ becomes increasingly concentrated near $0$. Exploring the \textit{almost-Hermitian regime}, where $n\to + \infty$ and $\tau \to 1$ simultaneously, would therefore be an interesting direction for future research.

\vspace{-0.25cm}\paragraph{Smallest gaps of the induced Ginibre ensemble (induced GinUE).} An induced Ginibre matrix is of the form $(G^{\dagger}G)^{1/2}U$, where $G$ is a Ginibre matrix of size $m\times n$ $(m\geq n)$ whose entries have variance $\frac{1}{n}$, and $U$ is an $n\times n$ Haar distributed unitary matrix that is independent of $G$. Eigenvalues of such matrices are distributed as \eqref{general density intro} with $Q(z) = |z|^{2}-2a\log |z|$ and $a=\frac{m-n}{n}$, see \cite[Proposition 2.8]{BFreview}. In this case, the eigenvalues accumulate uniformly on an annulus:
\begin{align*}
\mathcal{S}=\{z \in \C: \sqrt{a} \leq |z| \leq \sqrt{1+a}\}, \qquad \rho(z) = \frac{1}{\pi}\mathbf{1}_{\mathcal{S}}(z), \qquad \mathcal{J} = 1.
\end{align*}
Even though Corollary \ref{col1} does not apply directly since $\partial \mathcal{S}$ is the union of two smooth Jordan curves, our numerical simulations suggest that the conclusion of Corollary \ref{col1} still holds in this setting (this also supports Remark \ref{rem:assumptions not needed}). Indeed, the histograms in the second column of Figure \ref{fig:histo}, which were made by computing $t_{1}^{(400)},t_{2}^{(400)}$ for $10^4$ induced Ginibre matrices with $n=400$ and $m=600$ (i.e. $a=\frac{1}{2}$), show convincing agreement with the densities of $t_{1},t_{2}$ given by Corollary \ref{col1}.

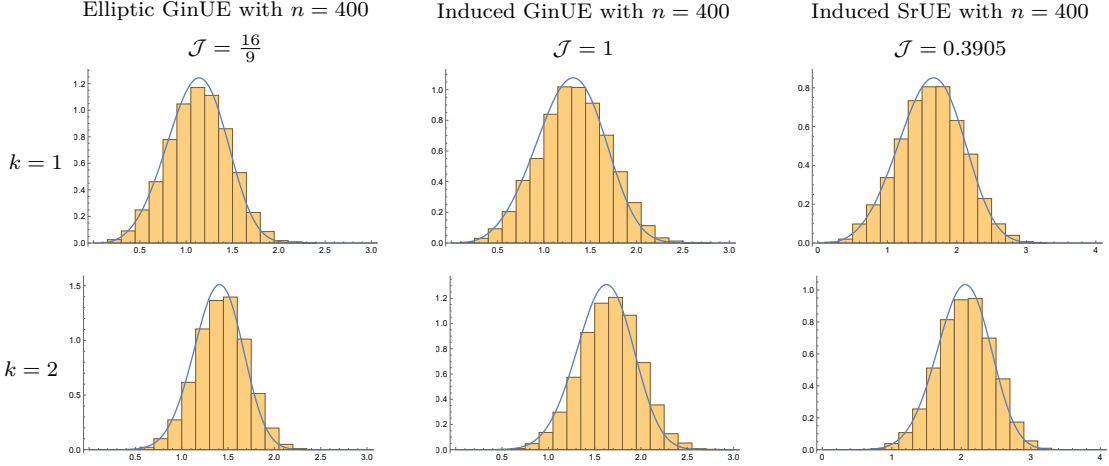
\begin{figure}
\begin{center}
\begin{tikzpicture}[master]
\node at (0,0) {\includegraphics[width=4cm]{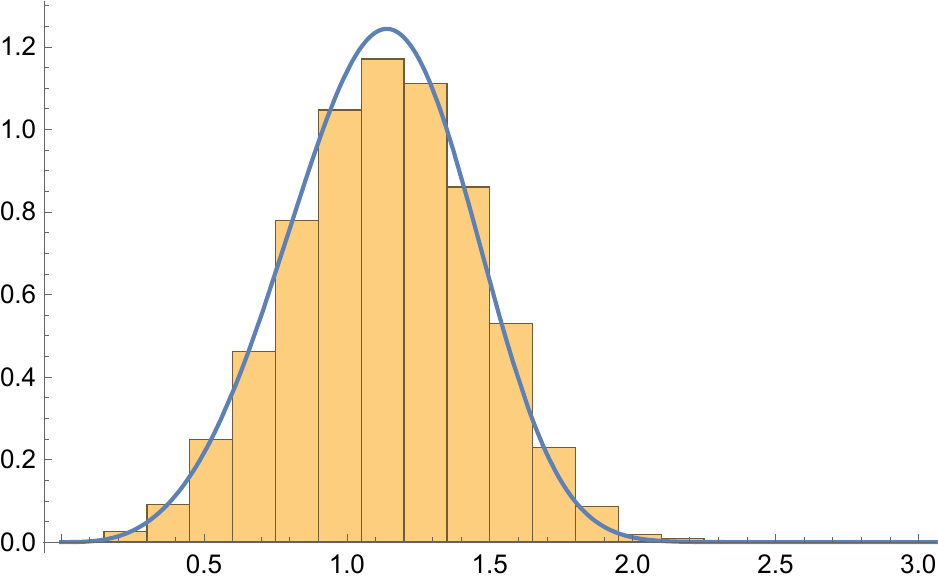}};
\node at (0,2) {\footnotesize Elliptic GinUE with $n=400$};
\node at (0,1.5) {\footnotesize $\mathcal{J}=\frac{16}{9}$};
\node at (-2.5,0) {\footnotesize $k=1$};
\end{tikzpicture}\hspace{-0.5cm}
\begin{tikzpicture}[slave]
\node at (0,0) {\includegraphics[width=4cm]{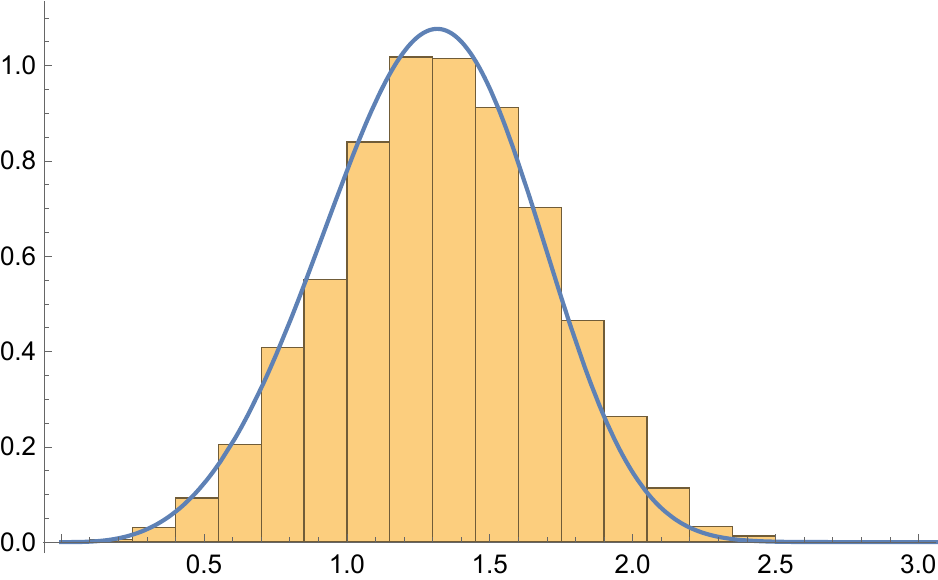}};
\node at (0,2) {\footnotesize Induced GinUE with $n=400$};
\node at (0,1.5) {\footnotesize $\mathcal{J}=1$};
\end{tikzpicture} \hspace{-0.5cm}
\begin{tikzpicture}[slave]
\node at (0,0) {\includegraphics[width=4cm]{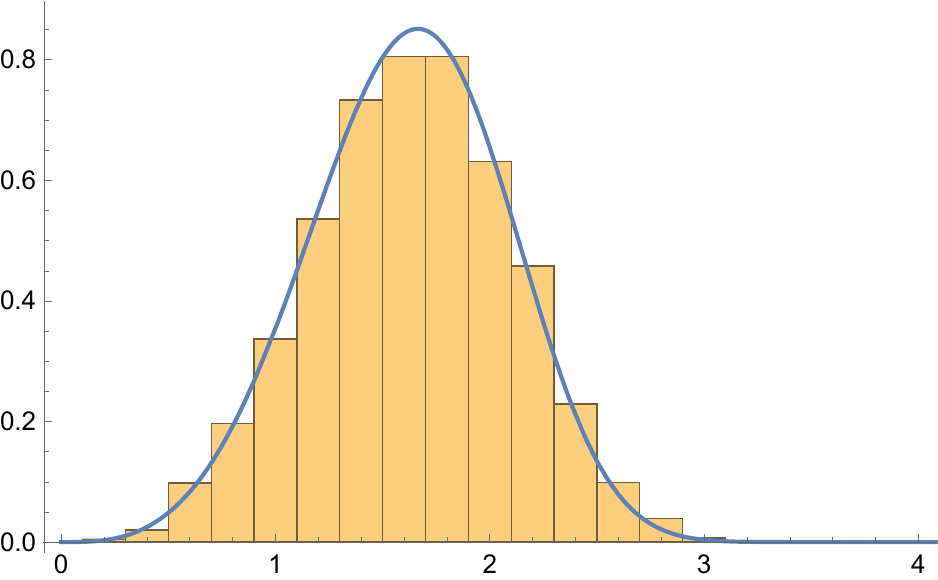}};
\node at (0,2) {\footnotesize Induced SrUE with $n=400$};
\node at (0,1.5) {\footnotesize $\mathcal{J}= 0.3905$};
\end{tikzpicture} 
\begin{tikzpicture}[master]
\node at (0,0) {\includegraphics[width=4cm]{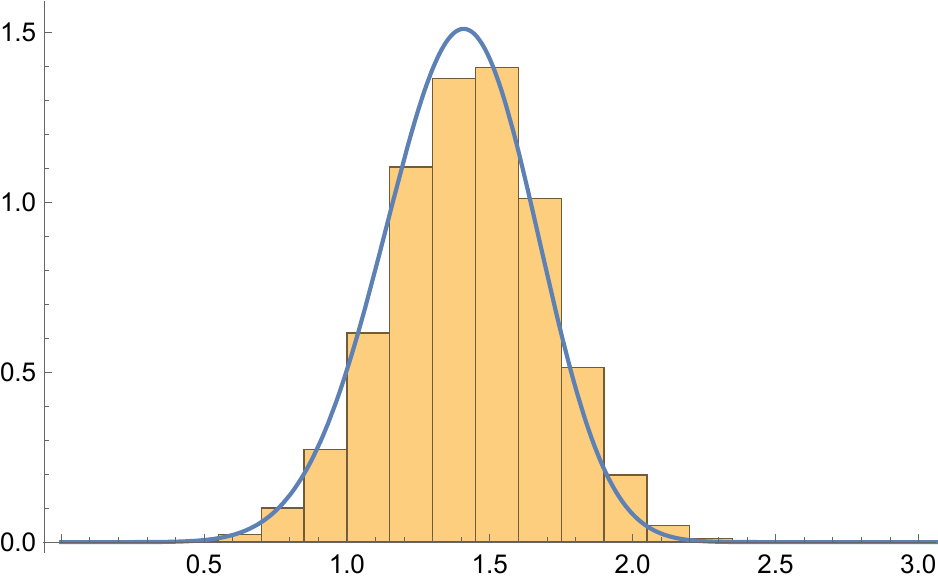}};
\node at (-2.5,0) {\footnotesize $k=2$};
\end{tikzpicture} \hspace{-0.5cm}
\begin{tikzpicture}[slave]
\node at (0,0) {\includegraphics[width=4cm]{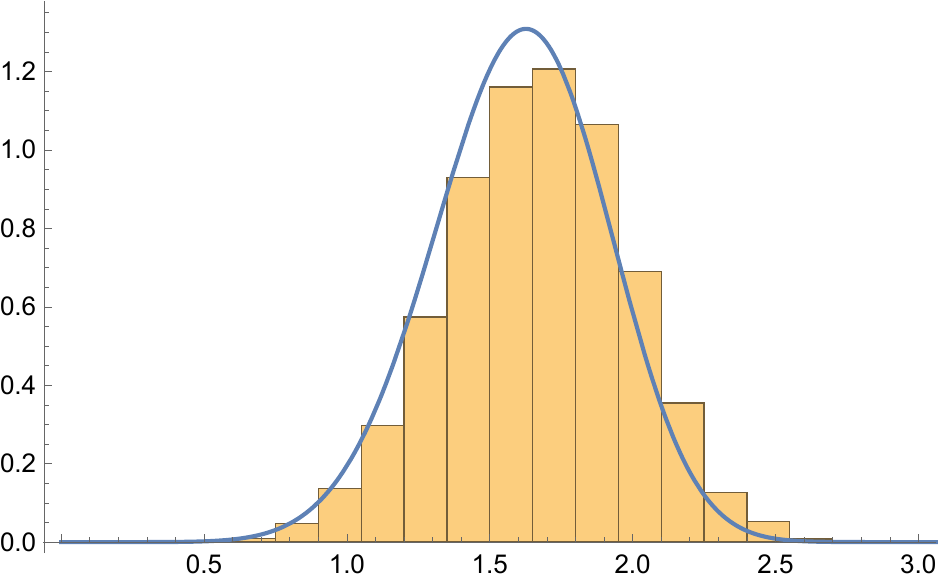}};
\end{tikzpicture} \hspace{-0.5cm}
\begin{tikzpicture}[slave]
\node at (0,0) {\includegraphics[width=4cm]{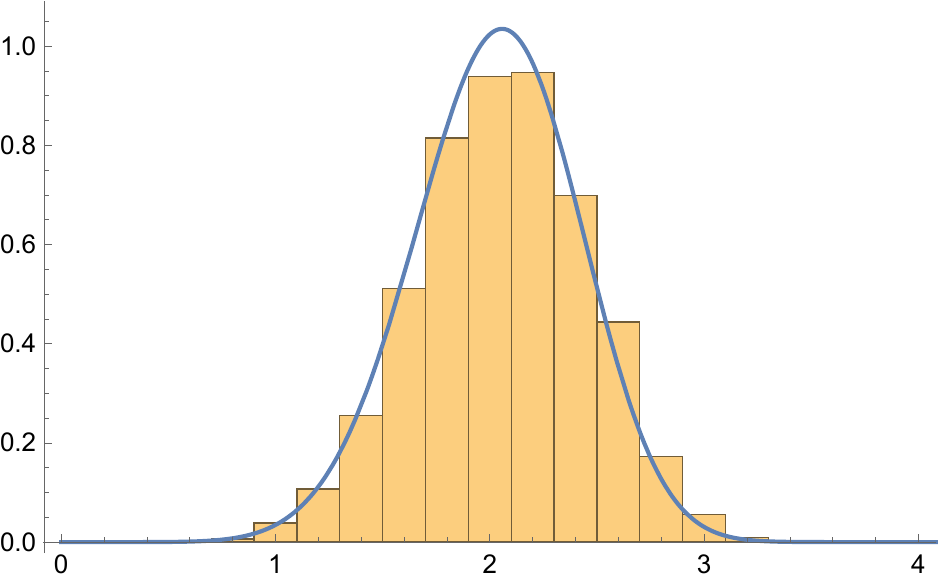}};
\end{tikzpicture} 
\end{center}
\caption{Numerical confirmations of Corollary \ref{col1} for $k=1$ (top) and $k=2$ (bottom), and for the potentials $Q(z) = \frac{1}{1-\tau^{2}}(|z|^2-\tau \re z^{2})$ with $\tau=\frac{1}{2}$ (left), $Q(z) = |z|^{2}-2a\log |z|$ with $a=\frac{1}{2}$ (middle), and $Q(z)=a\log(1+|z|^{2})$ with $a=1.25$ (right). In each diagram, the blue curve corresponds to the density \eqref{density} with either $k=1$ (top) or $k=2$ (bottom). The histograms were produced as described in the text. \label{fig:histo}}
\end{figure}

\vspace{-0.25cm}\paragraph{Smallest gaps of the induced spherical ensemble (induced SrUE).} An induced spherical matrix is of the form $(G_{2}^\dagger G_{2})^{-1/2}G_{1}$, where $G_{1}$ and $G_{2}$ are independent Ginibre matrices of size $n\times n$ and $m\times n$ ($m\geq n$), respectively, and whose entries have variance $1$. It follows from \cite[Exercises 3.6 q.3 with $\beta=2$]{F book} that the eigenvalues of $(G_{2}^\dagger G_{2})^{-1/2}G_{1}$ follow \eqref{general density intro} with $Q(z) = a\log (1+|z|^{2})$ and $a=\frac{m}{n} \geq 1$. If $m/n=a>1$ is fixed, then $Q$ satisfies the growth condition \eqref{growth of $Q$} and we have
\begin{align*}
\mathcal{S} = \{z\in \C:|z|\leq \frac{1}{\sqrt{a-1}}\}, \qquad \rho(z) = \frac{a \; \mathbf{1}_{\mathcal{S}}(z)}{\pi(1+|z|^{2})^{2}}, \qquad \mathcal{J} = \frac{1-5a+10a^{2}-10a^{3}+5a^{4}}{5a^{2}}.
\end{align*}
The histograms in the third column of Figure \ref{fig:histo} were made from $10^4$ induced spherical matrices with $n=400$ and $m=500$, and  support the validity of Corollary \ref{col1}.

\vspace{-0.25cm}\paragraph{Smallest gap of truncated unitary matrices (TUE).} Let $T$ be the upper-left $n\times n$ submatrix of a Haar distributed unitary matrix of size $(n+\alpha)\times (n+\alpha)$ for some $\alpha \in \mathbb{N}_{>0}$. It is shown in \cite{ZS2000} that the eigenvalues of $T$ are distributed according to  \eqref{general density intro} with
\begin{align}\label{Q TUE}
Q(z) = \begin{cases}
-a\log(1-|z|^{2}), & \mbox{if } |z|< 1, \\
+\infty, & \mbox{if } |z| \geq 1,
\end{cases} \qquad \mbox{ and } \qquad a=\frac{\alpha-1}{n}.
\end{align}
If $a>0$ is independent of $n$, then
\begin{align*}
\mathcal{S} = \{z\in \C:|z|\leq \frac{1}{\sqrt{1+a}}\}, \qquad \rho(z) = \frac{a}{\pi(1-|z|^{2})^{2}}, \qquad \mathcal{J} = \frac{1+5a+10a^{2}+10a^{3}+5a^{4}}{5a^{2}}.
\end{align*}
Note that although $Q\notin C^{2}(\C)$, its discontinuity lies outside the droplet $\mathcal{S}$, so we still expect the conclusion of Corollary \ref{col1} to hold when $a>0$. This is confirmed numerically, though the convergence appears significantly slower in this case. For the elliptic GinUE, induced GinUE, and induced SrUE, the histograms are already in good agreement with Corollary \ref{col1} when $n=400$. In contrast, for the TUE, the convergence is slower: even at $n=2000$, the histogram still shows noticeable deviations, but the progression across increasing values of $n$ clearly supports convergence, as illustrated in Figure~\ref{fig:TUE}. 

\begin{figure}
\begin{center}
\begin{tikzpicture}[master]
\node at (0,0) {\includegraphics[width=4cm]{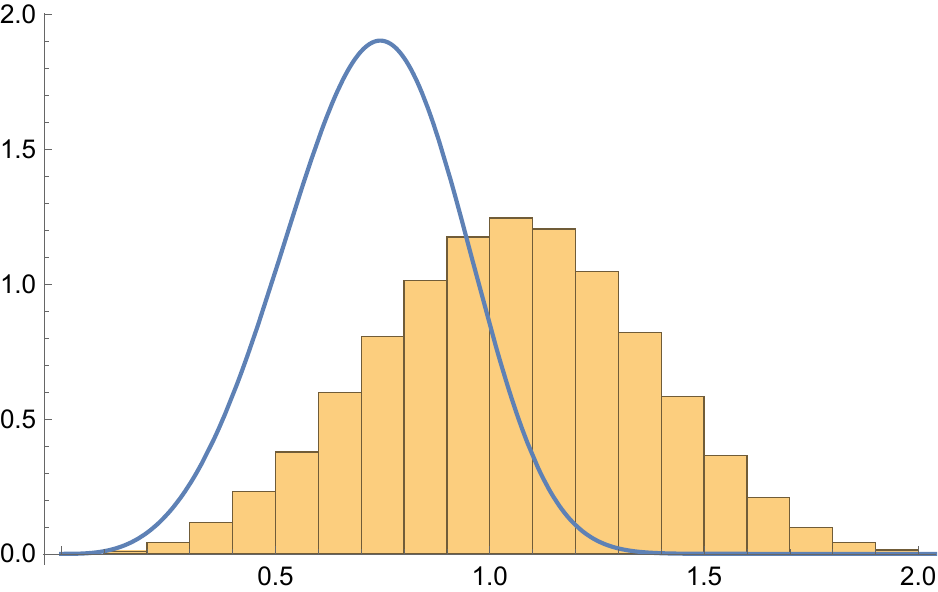}};
\node at (0,2) {\footnotesize TUE with $n=40$};
\node at (0,1.5) {\footnotesize $\mathcal{J}=9.7625$};
\node at (-2.5,0) {\footnotesize $k=1$};
\end{tikzpicture}\hspace{-0.5cm}
\begin{tikzpicture}[slave]
\node at (0,0) {\includegraphics[width=4cm]{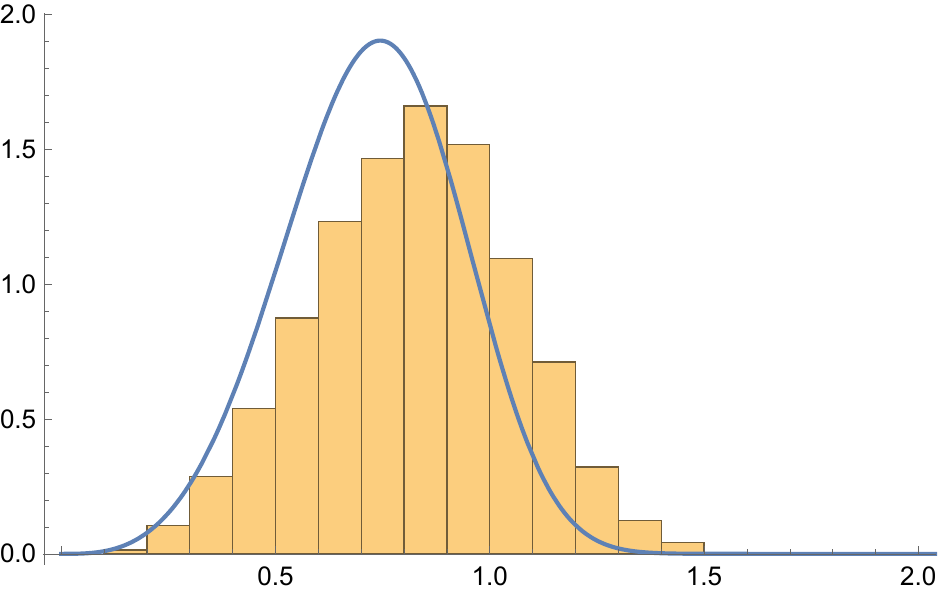}};
\node at (0,2) {\footnotesize TUE with $n=400$};
\node at (0,1.5) {\footnotesize $\mathcal{J}=9.7625$};
\end{tikzpicture} \hspace{-0.5cm}
\begin{tikzpicture}[slave]
\node at (0,0) {\includegraphics[width=4cm]{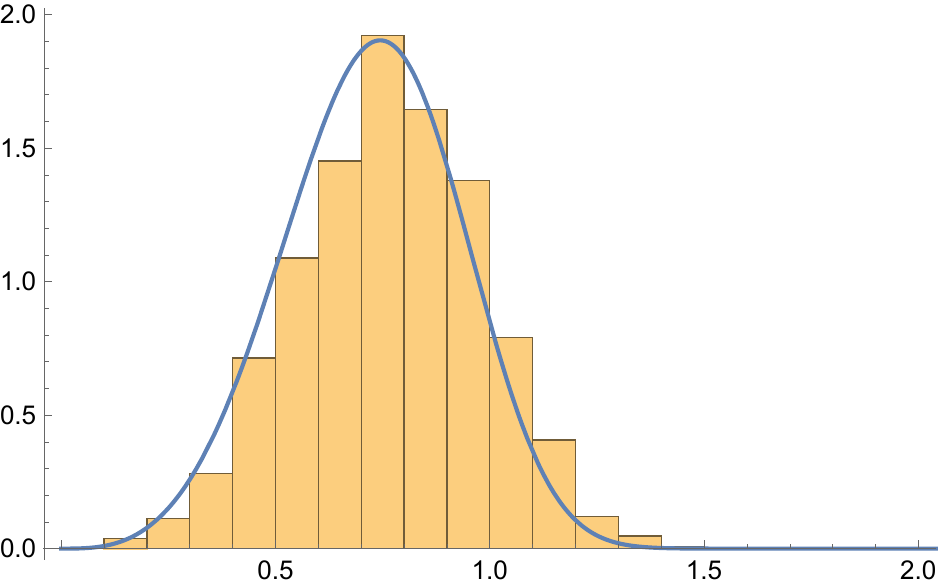}};
\node at (0,2) {\footnotesize TUE with $n=2000$};
\node at (0,1.5) {\footnotesize $\mathcal{J}=9.7625$};
\end{tikzpicture}
\begin{tikzpicture}[master]
\node at (0,0) {\includegraphics[width=4cm]{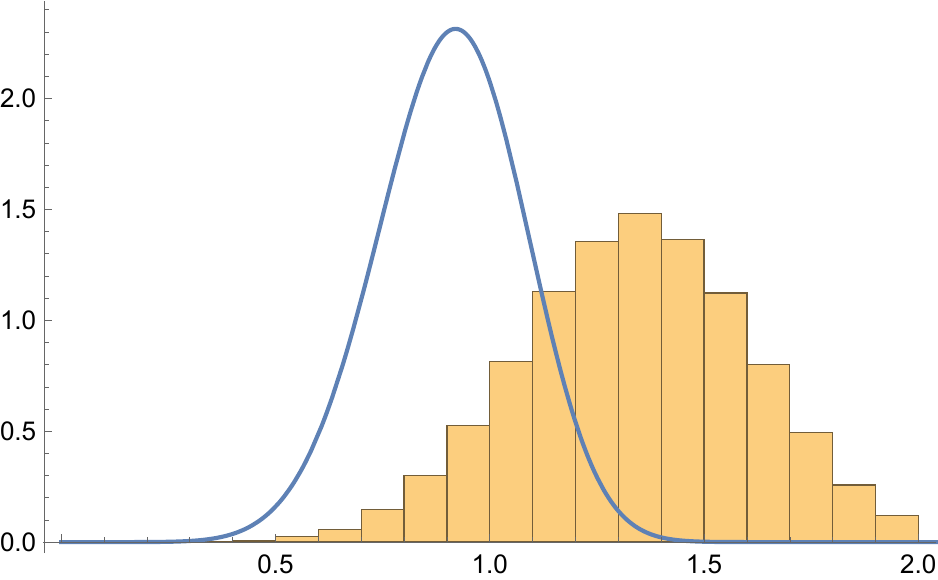}};
\node at (-2.5,0) {\footnotesize $k=2$};
\end{tikzpicture}\hspace{-0.5cm}
\begin{tikzpicture}[slave]
\node at (0,0) {\includegraphics[width=4cm]{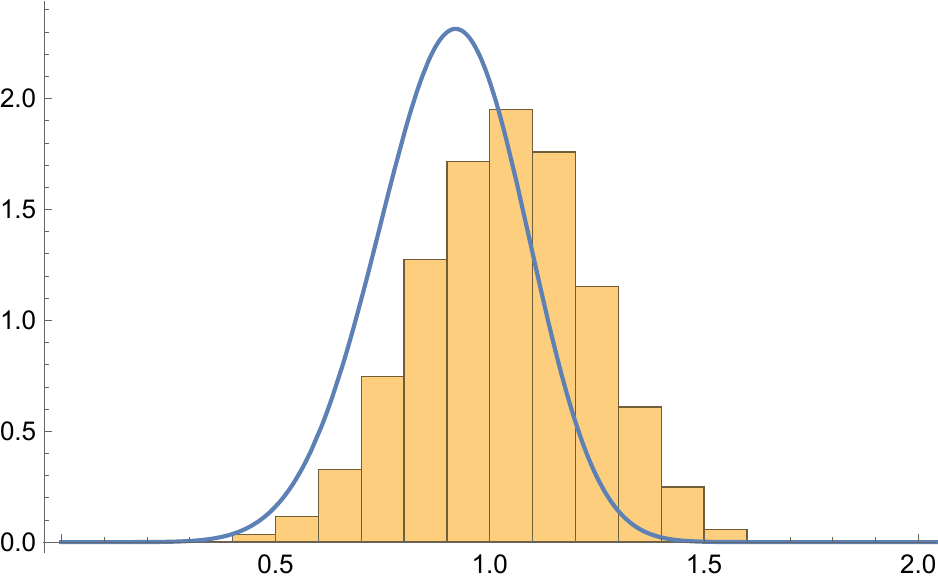}};
\end{tikzpicture} \hspace{-0.5cm}
\begin{tikzpicture}[slave]
\node at (0,0) {\includegraphics[width=4cm]{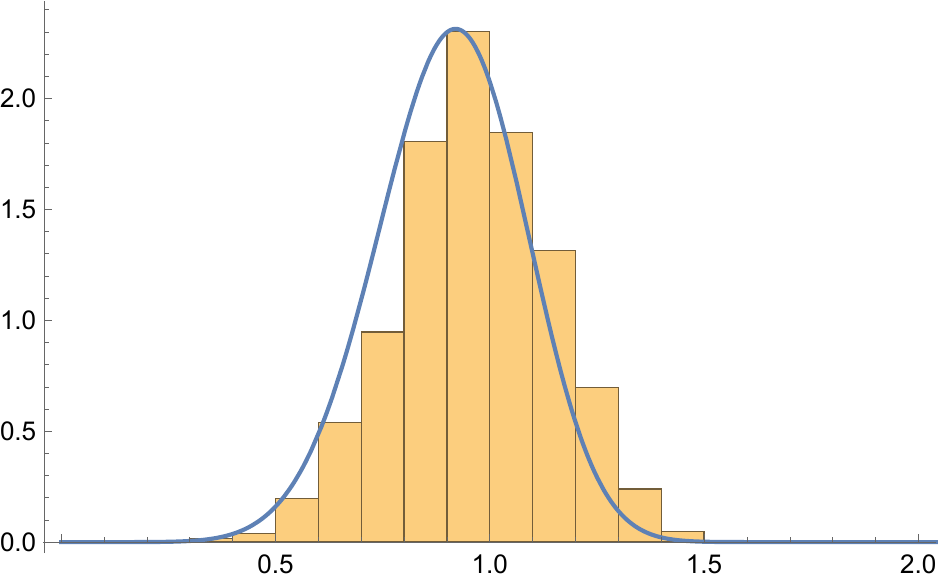}};
\end{tikzpicture}
\end{center}
\caption{\label{fig:TUE}Numerical confirmations of Corollary \ref{col1} for $k=1$ (top) and $k=2$ (bottom), for the potential \eqref{Q TUE} with $a=0.25$. The plots show histograms for increasing values of $n$; as $n$ grows, the empirical distributions seem to converge toward the theoretical prediction of Corollary \ref{col1}.}
\end{figure}
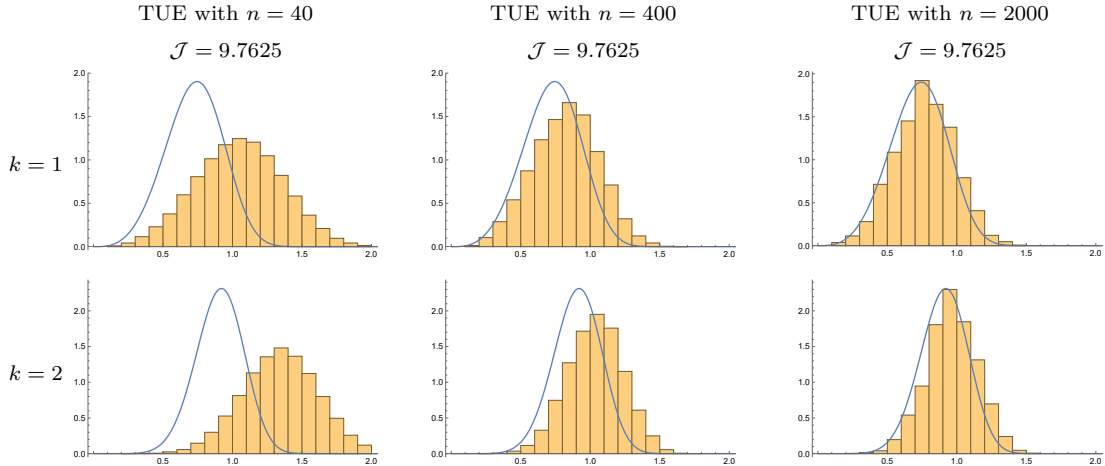

\medskip This slower convergence may be explained by \textit{where} the smallest gaps are most likely to occur. For the TUE, the limiting density $z\mapsto \rho(z)$ is maximized for $z \in \partial \mathcal{S}$, so the smallest gaps typically arise near the edge of the droplet. In contrast, for the elliptic GinUE, induced GinUE, and even more so for the induced SrUE, the smallest gaps are more likely to occur in the bulk. This distinction is significant because the convergence of the average empirical density $\frac{1}{n}K_{n}(z,z)$ (with $K_{n}$ the correlation kernel) to $\rho(z)$ is uniform on compact subsets of $\mathcal{S}^{\circ}$ (see \cite[Theorem 2.8]{AHM2010}), but breaks down near the edge (see \cite[Corollary 1.5.6]{HW2021}). When the smallest gaps are more likely to occur at the edge (as in the TUE), the convergence is therefore expected to be slower. Making this heuristic more precise would be an interesting direction for future work.

\paragraph{Outline.} Our proof follows the approach of \cite[Section 3]{SJ2012}, which in turn builds on ideas from \cite{BAB2013, Soshnikov}. The analysis relies on the fact that the point process \eqref{general density intro} is determinantal. In \cite{SJ2012}, many estimates on the $k$-point correlation functions were obtained using the explicit form of the Ginibre correlation kernel. To extend their analysis to general potentials $Q$, we rely on several estimates on the correlation kernel due to Ameur, Hedenmalm, and Makarov \cite{AmeurNearBoundary, AHM2010}, together with the new result Theorem \ref{thm:Charlier} that provides precise asymptotics for the correlation kernel near the soft edge. This new theorem extends a result of Hedenmalm and Wennman \cite{HW2021} by going beyond the leading term and by covering a slightly larger neighborhood of $\partial \mathcal{S}$. The relevant results from \cite{AmeurNearBoundary, AHM2010}, together with Theorem \ref{thm:Charlier}, are collected in Section \ref{sec:background}. Using these estimates, the arguments of \cite[Section 3]{SJ2012} are adapted in Section \ref{sec:proof}, which also contains the proof of our main results. Theorem \ref{thm:Charlier} is proved in Appendix \ref{appendix:Proof of edge}. 

\section{Estimates on the kernel}\label{sec:background}

Let $\xi^{(n)} := \sum_{i=1}^{n}\delta_{z_i}$, where $z_{1},\ldots,z_{n}$ are drawn from \eqref{general density intro} and indexed so that $z_{1}\preceq \ldots \preceq z_{n}$. The associated $k$-point correlation functions $\{\rho_{k}:\mathbb{C}^{k} \to [0,+\infty)\}_{k \geq 1}$ are defined such that
\begin{align}\label{def of DPP}
\mathbb{E}\bigg[\sum_{\substack{z_{1},\ldots,z_{k}\in \xi^{(n)} \\ z_{i} \neq z_{j} \; \mathrm{if} \; i \neq j}}f(z_{1},\ldots,z_{k})\bigg] = \int_{\C^{k}}f(w_1, \ldots, w_k)\rho_k(w_1, \ldots, w_k)d^{2}w_1 \ldots d^{2}w_k
\end{align}
holds for any bounded measurable function $f:\C^{k}\to \C$ with compact support. The sum at the left-hand side in \eqref{def of DPP} is taken over all (ordered) $k$-tuples of distinct points of $\xi^{(n)}$. With the function
\begin{align*}
f = \mathbbm{1}_{\Lambda_{1}^{k_{1}}\times \dots \times \Lambda_{m}^{k_{m}}},
\end{align*}
where $m,k_{1},\ldots,k_{m}$ are integers satisfying $\sum_{i=1}^{m}k_{i}=k$ and $\Lambda_{1},\ldots,\Lambda_{m}\subset \C$ are pairwise disjoint bounded Borel sets, \eqref{def of DPP} becomes
\begin{align}\label{def of DPP 2}
\mathbb{E}\bigg[\prod_{i=1}^{m}\frac{\xi^{(n)}(\Lambda_{i})!}{(\xi^{(n)}(\Lambda_{i})-k_{i})!}\bigg] = \int_{\Lambda_{1}^{k_{1}}\times \dots \times \Lambda_{m}^{k_{m}}}\rho_k(w_1, \ldots, w_k)d^{2}w_1 \ldots d^{2}w_k.
\end{align}
As mentioned above, the point process $\xi^{(n)}$ is determinantal, which means that all correlation functions exist and can be written in the form of a determinant:
\begin{align}\label{def of rhok}
\rho_k(w_1, \ldots, w_k) & = \det(K_n(w_i,w_j))_{1\leq i, j\leq k}. 
\end{align}
Here $K_{n}:\C^{2}\to \C$ is the so-called \textit{correlation kernel}, given by
\begin{align}\label{def of Kn}
K_n(z,w) =  e^{-\frac{1}{2}nQ(z)-\frac{1}{2}nQ(w)}\sum_{j = 0}^{n-1}p_{j,n}(z)\overline{p_{j,n}(w)},
\end{align}
where $p_{j,n}$ is the orthonormal polynomial of degree $j$ with positive leading coefficient satisfying 
\begin{align}\label{ortho condition}
\int_{\C}p_{j,n}(z)\overline{p_{k,n}(z)}e^{-nQ(z)}d^{2}z = \delta_{j,k}, \qquad j,k \geq 0.
\end{align}
For more background on determinantal point processes, see e.g. \cite{SoshnikovSurvey}.

\medskip Let $z_{0}\in \C$ and $\epsilon >0$. In what follows, we will use the notation
\begin{align*}
\partial_{z} = \frac{1}{2}\bigg(\frac{\partial}{\partial x} - i \frac{\partial}{\partial y}\bigg), \qquad \partial_{\overline{z}} = \frac{1}{2}\bigg(\frac{\partial}{\partial x} + i \frac{\partial}{\partial y}\bigg), \qquad D(z_{0},\epsilon) = \{z\in \C : |z-z_{0}| \leq \epsilon\}.
\end{align*}
If $Q$ is real analytic in a neighborhood of a point $z_{0}$, the function
\begin{align*}
\psi(z,w) = \sum_{j,k=0}^{+\infty} \frac{1}{j!k!}(z-z_{0})^{j}(w-\overline{z_{0}})^{k}\partial_{z}^{j}\partial_{\overline{z}}^{k}Q(z_{0})
\end{align*}
is holomorphic in $(z,w)$ when $z$ is close to $z_{0}$ and $w$ is close to $\overline{z_{0}}$, and we have
\begin{align*}
Q(z) = \psi(z,\overline{z})
\end{align*}
for all $z$ close to $z_{0}$. Let $b_{0}$ be the function
\begin{align*}
b_{0}(z,w) = \partial_{z}\partial_{w}\psi(z,w),
\end{align*}
so that $b_{0}(z,\overline{z}) = \partial_{z}\partial_{\overline{z}}Q(z) = \frac{\Delta Q(z)}{4}$ for all $z$ close to $z_{0}$. Since we assume that $Q$ is real-analytic in a neighborhood $\mathcal{N}$ of $\mathcal{S}$, by varying $z_{0}$ we find that $\psi(z,w)$ extends to a holomorphic function in a neighborhood of $\{(z,w)\in \mathcal{N}^{2}:w=\overline{z}\}$.

\smallskip For all $z,w$ close enough to each other, we have by \cite[(5.7) and above]{AHM2010} that
\begin{subequations}\label{asymp psi z near w}
\begin{align}
& \psi(z,\overline{w}) - \frac{Q(z) + Q(w)}{2} = - i \, \im L_{z}(w) - |w-z|^{2} \frac{\Delta Q(z)}{8} + \bigO(|w-z|^{3}), \\
& \re \psi(z,\overline{w}) - \frac{Q(z) + Q(w)}{2} = - \frac{\Delta Q(z)}{8} |z-w|^{2} + \bigO(|w-z|^{3}),
\end{align}
\end{subequations}
where $L_{z}(w) = (w-z)\partial_{z}Q(z) + \frac{1}{2}(w-z)^{2}\partial_{z}^{2}Q(z)$.

\smallskip The following result follows from \cite[Theorem 2.1]{AmeurNearBoundary}  (a weaker version of this theorem was previously proved in \cite[Theorem 2.8]{AHM2010}).

\begin{theorem}[Bulk correlation](\cite[Theorem 2.1]{AmeurNearBoundary})\label{thm:Ameur bulk-bulk correlation improved} Let $Q$ be as in Theorem \ref{thm1}. Fix a constant $M$ satisfying $M \geq 100/\sqrt{\min_{\mathcal{S}}\Delta Q}>0$, and set $\epsilon_{n} = M \sqrt{\log n}/\sqrt{n}$. Let $B_{n}$ be the $2\epsilon_{n}$-neighborhood of $\partial \mathcal{S}$. Let $(z_{n})$ be any convergent sequence such that $z_{n}\in \mathcal{S}\setminus B_{n}$ for all $n$. Then 
\begin{align*}
K_{n}(z,w) = \frac{n b_{0}(z,\overline{{w}})}{\pi}e^{n \psi(z,\overline{w}) -n \frac{Q(z)+Q(w)}{2}} + \bigO(1), 
\end{align*}
as $n\to + \infty$ uniformly for $z,w\in  D(z_{n},\epsilon_{n})$.
\end{theorem}
Let $\mathcal{S}^{\circ}$ denote the interior of $\mathcal{S}$. We will also use the following theorem, which shows that $K_{n}(z_{1},z_{0})$ decays rapidly when $z_{0}$ is in the bulk and $z_{1}$ is sufficiently far from $z_{0}$. 
\begin{theorem}[Fast decay]\cite[Theorem 8.1]{AHM2010}\label{thm:bulk damping}
Let $Q$ be as in Theorem \ref{thm1}, let $z_{1}\in \mathcal{N}$, and let $z_{0}\in \mathcal{S}^{\circ}$. Then there exist $c,C>0$ such that
\begin{align*}
|K_{n}(z_{1},z_{0})| \leq C n \exp \bigg( -c\sqrt{n} \min\bigg\{\frac{\mathrm{dist}(z_{0},\partial \mathcal{S})}{2},|z_{1}-z_{0}|\bigg\} \bigg)
\end{align*}
holds for all $n\in \N$. Moreover, the constant $c$ can be chosen independently of $z_{0}$ and $z_{1}$. 
\end{theorem}
The so-called \textit{obstacle function} $\widehat{Q}$ is defined for $z \in \C$ by
\begin{align}\label{def of Qhat}
\widehat{Q}(z) & = \ell_{Q}- 2 \int_{\C}\log \frac{1}{|z-w|}d\mu(w) 
\end{align}
where $\ell_{Q}$ is the constant in \eqref{EL inequality}. If $Q$ is real analytic in a neighborhood of $\mathcal{S}$, by \eqref{EL equality} $\widehat{Q}$ is real analytic in $\mathcal{S}^{\circ}$ and $\Delta Q(z) = \Delta \widehat{Q}(z)$ for all $z\in \mathcal{S}^{\circ}$. Moreover, $\widehat{Q}$ is subharmonic and $C^{1,1}$-smooth on $\C$ (see e.g. \cite{AHM2010}); therefore, since $Q$ is assumed to be regular, there exists $C_{0}>0$ such that
\begin{align}\label{Q minus Qhat lower bounded}
Q(z)-\widehat{Q}(z) \geq C_{0} \min\{1,\mathrm{dist}(z,\mathcal{S})^{2}\}, \qquad \mbox{for all } z \in \C.
\end{align} 
 
\begin{proposition}[Pointwise estimate]\cite[Proposition 3.6]{AHM2010}\label{prop:pointwise estimate}
Let $Q$ be as in Theorem \ref{thm1}. There exists $C_{1}>0$ such that
\begin{align*}
|K_{n}(z,w)| \leq C_{1} n e^{-\frac{n}{2}(Q(z)-\widehat{Q}(z))}e^{-\frac{n}{2}(Q(w)-\widehat{Q}(w))}, \qquad z,w\in \C.
\end{align*}
\end{proposition}

Finally, although the smallest gaps are unlikely to occur near the edge $\partial \mathcal{S}$, we will still need the following extension of a result from \cite{HW2021}, which provides subleading asymptotics for the kernel in a neighborhood of $\partial \mathcal{S}$, and which we believe is of independent interest.
\begin{theorem}[Edge correlation]\label{thm:Charlier}
Suppose $Q$ is as in Theorem \ref{thm1}, and let $z_{0}\in \partial \mathcal{S}$ and $M>0$ be fixed. Then, there exist unimodular continuous functions $c_{n}:\C\to \{z\in \C:|z|=1\}$ such that 
\begin{multline}
c_{n}(\xi;z_{0})\overline{c_{n}(\eta;z_{0})}K_{n}\bigg(z_{0}+\mathrm{n} \frac{\sqrt{2}\xi}{\sqrt{\Delta Q(z_{0})n/4}},  z_{0}+\mathrm{n} \frac{\sqrt{2}\eta}{\sqrt{\Delta Q(z_{0})n/4}}\bigg)  = \frac{\Delta Q(z_{0})n}{4\pi}\mathrm{k}(\xi,\eta)  \\ 
 +\mathrm{k}_{2}(\xi,\eta;z_{0}) \sqrt{n}+\bigO\big( 1+|\xi|^{6}+|\eta|^{6} \big),  \label{Charlier asymp}
\end{multline}
as $n\to + \infty$ uniformly for $z_{0}\in \partial \mathcal{S}$ and for $|\xi|,|\eta| \leq M \sqrt{\log n}$, where $\mathrm{n}\in \{z\in \C:|z|=1\}$ denotes the outward pointing unit normal to $\partial \mathcal{S}$ at $z_{0}$, and
\begin{align*}
& \mathrm{k}(\xi,\eta) = e^{2\xi \overline{\eta}-(|\xi|^{2}+|\eta|^{2})}\frac{\mathrm{erfc}(\xi+\overline{\eta})}{2}, \qquad \mathrm{erfc}(z) = \frac{2}{\sqrt{\pi}}\int_{z}^{+\infty}e^{-t^{2}}dt, \\
& \mathrm{k}_{2}(\xi,\eta;z_{0}) = \Big(\gamma_{1} (\re \xi)^{3} + \gamma_{1} (\re \eta)^{3} + \gamma_{2} (\re \xi)^{2} \im \xi + \gamma_{2} (\re \eta)^{2} \im \eta + \gamma_{3} \re \xi (\im \xi)^{2}+ \gamma_{3} \re \eta (\im \eta)^{2} \\
& + \gamma_{4} \re \xi + \gamma_{4} \re \eta + \gamma_{5} \im \xi + \gamma_{5} \im \eta \Big) \mathrm{k}(\xi,\eta) + \Big( \beta_{1} (\re \xi)^{2} + \overline{\beta_{1}} (\re \eta)^{2} + \beta_{2} (\re \xi)(\im \xi) \\
&  + \overline{\beta_{2}} (\re \eta)(\im \eta) + \beta_{3} (\im \xi)^{2} + \overline{\beta_{3}} (\im \eta)^{2} \Big) \Big( (\xi + \overline{\eta})\mathrm{k}(\xi,\eta) - \frac{e^{-2 (\re \xi)^{2}-2 (\re \eta)^{2}}}{2\sqrt{\pi}e^{2i(\re \xi \, \im \xi - \re \eta \, \im \eta)}} \Big) \\
& + \gamma_{6} (\xi + \overline{\eta})\mathrm{k}(\xi,\eta) + \gamma_{7} \frac{e^{-2 (\re \xi)^{2}-2 (\re \eta)^{2}}}{2\sqrt{\pi}e^{2i(\re \xi \, \im \xi - \re \eta \, \im \eta)}} + \Big( \beta_{4} \re \xi + \overline{\beta_{4}} \re \eta + \beta_{5} \im \xi + \overline{\beta_{5}} \im \eta \Big)  \\
& \times \bigg\{ (\xi + \overline{\eta})^{2}\mathrm{k}(\xi,\eta)- \frac{(\xi + \overline{\eta})e^{-2 (\re \xi)^{2}-2 (\re \eta)^{2}}}{2\sqrt{\pi}e^{2i(\re \xi \, \im \xi - \re \eta \, \im \eta)}} \bigg\} + \gamma_{8} \bigg\{ (\xi + \overline{\eta})^{3}\mathrm{k}(\xi,\eta) - \frac{(\xi + \overline{\eta})^{2}e^{-2 (\re \xi)^{2}-2 (\re \eta)^{2}}}{2\sqrt{\pi}e^{2i(\re \xi \, \im \xi - \re \eta \, \im \eta)}} \bigg\},
\end{align*}
for some coefficients $\beta_{1},\ldots,\beta_{5} \in \C$ and $\gamma_{1},\ldots,\gamma_{8}\in \R$ that depend smoothly on $z_{0}$. 
\end{theorem}
\begin{remark}\label{remark:weaker version of the main result}
We prove in Appendix \ref{appendix:k bounded} that $\mathrm{k}(\xi,\eta)$ is bounded for $\xi,\eta \in \C$ (see also \cite[Lemma 8.5]{AmeurOC2012}). It then directly follows from the definition of $\mathrm{k}_{2}$ that there exists $C>0$ such that
\begin{align*}
|\mathrm{k}_{2}(\xi,\eta;z_{0})| \leq C(1+|\xi|^{3}+|\eta|^{3}), \qquad \mbox{for all } \xi,\eta \in \C \mbox{ and } z_{0}\in \partial \mathcal{S}.
\end{align*}
Therefore, \eqref{Charlier asymp} implies 
\begin{multline}\label{lol44}
c_{n}(\xi;z_{0})\overline{c_{n}(\eta;z_{0})}K_{n}\bigg(z_{0}+\mathrm{n} \frac{\sqrt{2}\xi}{\sqrt{\Delta Q(z_{0})n/4}},  z_{0}+\mathrm{n} \frac{\sqrt{2}\eta}{\sqrt{\Delta Q(z_{0})n/4}}\bigg) \\  = \frac{\Delta Q(z_{0})n}{4\pi}\mathrm{k}(\xi,\eta) + \bigO\big((1+|\xi|^{3}+|\eta|^{3})\sqrt{n}\big), \qquad \mbox{as } n \to + \infty
\end{multline}
uniformly for $|\xi|,|\eta| \leq M\sqrt{\log n}$ and $z_{0}\in \partial \mathcal{S}$. We mention however that \eqref{lol44} is not enough for our needs; indeed, the factor $1+|\xi|^{3}+|\eta|^{3}$ in the error term, which can grow like $\bigO((\log n)^{3/2})$, would cause other error terms in our proof to blow up. Such issues do not arise in the bulk, where the subleading term in the kernel expansion is $\bigO(1)$ (see Theorem \ref{thm:Ameur bulk-bulk correlation improved}). 
\end{remark}

\cite[Corollary 1.7]{HW2021} is a weaker version of Theorem \ref{thm:Charlier}, with ``$\mathrm{k}_{2}(\xi,\eta;z_{0})\sqrt{n}+\bigO(1+|\xi|^{6}+|\eta|^{6})$" replaced by ``$o(n)$", and with ``uniformly for $|\xi|,|\eta| \leq M \sqrt{\log n}$" replaced by ``uniformly for $\xi,\eta$ in compact subsets of $\C$". The result \cite[Lemma 1.3 with $z_{0}=w_{0}$]{MMOC2025}, which was obtained simultaneously and independently, gives an asymptotic formula of the form \eqref{lol44}, with ``$\bigO\big((1+|\xi|^{3}+|\eta|^{3})\sqrt{n}\big)$" replaced by ``$\bigO\big((\log n)^{3}\sqrt{n}\big)$". Nevertheless, as mentioned in Remark \ref{remark:weaker version of the main result}, such estimates are not enough for our needs. We also mention that \cite[Corollary 1.7]{HW2021} was extended in \cite[Theorem B]{CW2025} to the case where $\partial \mathcal{S}$ consists of finitely many real-analytic Jordan curves, and that an all order expansion of the correlation kernel for arbitrary $z_0\in\C$ along the diagonal $\xi=\eta$ will appear in the forthcoming work \cite{HW2025}.

\medskip The proof of \cite[Corollary 1.7]{HW2021} proceeds in two steps. First, the case $\xi=\eta$ is treated using a Riemann sum approximation based on the asymptotics of the associated planar orthogonal polynomials obtained in \cite[Theorem 1.3]{HW2021}. Second, the result is extended to distinct $\xi,\eta$ via a polarization argument. This polarization argument, however, does not apply to the subleading term. In our setting we therefore proceed differently: Theorem \ref{thm:Charlier} is proved in Appendix \ref{appendix:Proof of edge} without polarization arguments, relying instead on Riemann sum approximations and the fact that the asymptotics in \cite[Theorem 1.3]{HW2021} are given to all orders and are valid in a $\sqrt{\log n}/\sqrt{n}$-neighborhood of $\partial \mathcal{S}$. In Appendix \ref{appendix:consistency checks}, we provide two consistency checks of Theorem \ref{thm:Charlier} by comparing with known results in the literature for special choices of $Q$.

\section{Proofs of main results}\label{sec:proof}

In this section we prove Theorem \ref{thm1} and Corollary \ref{col1}. 

\medskip We first recall the following result from \cite{Kallenberg}. 

\begin{proposition}\label{prop1}\cite[page 121]{Kallenberg}
Let $\{\eta^{(n)}\}_{n=1}^{+\infty}$ be a sequence of point processes on $\HH_{+}\times\mathbb{C}$, and let $\eta$ be a Poisson point process on $\HH_{+}\times\mathbb{C}$ with intensity measure $\mu$ having no atoms. If $\eta^{(n)}(J)$ converges in distribution to $\eta(J)$ for every set $J\subset \HH_{+}\times\mathbb{C}$ that is a finite union of bounded rectangles, then the sequence of point processes $\eta^{(n)}$ converges in distribution to $\eta$.
\end{proposition}

\medskip Consider the point process $\chi^{(n)} = \sum_{i=1}^{n-1}\delta_{(n^{3/4}(z_{i^*}-z_i), z_i)}$, where $z_{1},\ldots,z_{n}$ are drawn from \eqref{general density intro} and indexed so that $z_{1}\preceq \ldots \preceq z_{n}$, and $i^* = \arg\min_{j = i+1}^n\{|z_j-z_i|\}$. By Proposition \ref{prop1}, to prove Theorem \ref{thm1} we need to establish that for any $m\in \N_{>0}$ and fixed bounded Borel sets $B_{1},\ldots,B_{m}\subset \HH_{+}$ and $\Omega_{1},\ldots,\Omega_{m}\subset\mathbb{C}$ such that the rectangles $B_{1}\times\Omega_{1},\ldots,B_{m}\times\Omega_{m}$ are pairwise disjoint,
\begin{equation}\label{eqn2}
\chi^{(n)}(J)\oset{\mathrm{law}}{\underset{n\to \infty}{\xrightarrow{\hspace*{0.85cm}}}}\mathrm{Poisson}(\lambda_{J}),  
\end{equation}
where $J:=\cup_{\ell=1}^{m}(B_{\ell}\times \Omega_{\ell})$ and
\begin{align}\label{def of lambda A Omega}
\lambda_{J} =  \sum_{\ell=1}^{m}\biggl(\pi^{2}\int_{\Omega_{\ell}\cap \mathcal{S}}\rho(v)^{3}d^{2}v\biggl)\biggl(\int_{B_{\ell}}|z|^2\frac{d^{2}z}{\pi}\biggl).
\end{align}

\medskip Let $B_{\ell,n} = \{n^{-3/4}b: b\in B_{\ell}, b \succ 0\}$, and consider the point processes
\begin{align*}
\xi^{(n)} := \sum_{i=1}^{n}\delta_{z_i}, \qquad \xi^{(n)}_{J} := \sum_{i=1}^{n}\delta_{z_i} \mathbbm{1}_{\{\xi^{(n)}(z_i+\cup_{\ell:z_{i}\in \Omega_{\ell}}B_{\ell,n}) = 1\}}.
\end{align*}
We also let $\{\rho_{k}\}_{k \geq 1}$ and $\{\rho_{k}^{J}\}_{k \geq 1}$ denote the correlation functions of $\xi^{(n)}$ and $\xi^{(n)}_{J}$, respectively, where we drop the dependence in $n$ for conciseness. 

\medskip The thinned point process $\xi^{(n)}_{J}$ is not determinantal in general, but its correlation functions $\{\rho_{k}^{J}\}_{k\geq 1}$ can still be expressed in terms of $\{\rho_{k}\}_{k \geq 1}$, as shown by Soshnikov \cite[(4.5)]{Soshnikov} (see also \eqref{eqn5} below). There is therefore an advantage in considering $\xi^{(n)}_{J}$ instead of $\chi^{(n)}$. Set $\Omega:=\bigcup_{\ell=1}^{m}\Omega_\ell$. Since
\begin{align*}
& \chi^{(n)}(J) = \#\bigg\{z_i\in \Omega: i \leq n-1, \; z_{i^*} - z_i \in \bigcup_{\ell:z_i\in\Omega_\ell} B_{\ell,n}\bigg\}, \\
& \xi^{(n)}_{J}(\Omega) = \#\bigg\{z_i\in \Omega: \mbox{ there exists a unique }z_j \in z_i + \bigcup_{\ell:z_i\in\Omega_\ell} B_{\ell,n} \bigg\},
\end{align*}
the event $\{\chi^{(n)}(J)\neq \xi^{(n)}_{J}(\Omega)\}$ arises only if three or more eigenvalues cluster together, which is rare. Hence, we expect $\chi^{(n)}(J)$ and $\xi^{(n)}_{J}(\Omega)$ to be asymptotically equivalent in distribution. This fact is proved in the next lemma.


\begin{lemma}\label{lemma2}
For any $m\in\N_{>0}$, bounded Borel sets
$B_1,\ldots,B_m\subset\HH_+$ and
(possibly unbounded) Borel sets
$\Omega_1,\ldots,\Omega_m\subset\C$,
set $J=\bigcup_{\ell=1}^{m}(B_\ell\times\Omega_\ell)$ and $\Omega=\bigcup_{\ell=1}^{m}\Omega_\ell$. Then
\begin{align*}
\chi^{(n)}(J)-\xi^{(n)}_{J}(\Omega)
\oset{\mathrm{law}}{\underset{n\to\infty}
{\xrightarrow{\hspace*{0.85cm}}}}0.
\end{align*}
\end{lemma}

%
%

\begin{proof} 
The proof follows the ideas of \cite[Proof of Lemma 3.3]{SJ2012} and uses some estimates from Section \ref{sec:background}. Fix $\tau>0$ such that $\bigcup_{\ell=1}^{m}B_\ell\subset
\HH_{+,\tau}:=\{u\in\HH_+:|u|<\tau\}$, and define $\tau_n = \tau n^{-\frac{3}{4}}$. For $z_{0}\in \C$  and $\epsilon>0$, let
\begin{align*}
D^{+}(z_{0},\epsilon) = \{z\in \C : |z-z_{0}| \leq \epsilon, \; z_{0} \preceq z\}.
\end{align*}
We start by proving that if
\begin{align}\label{1 is not 1}
\mathbbm{1}_{\{z_{i^*}\in z_i+
\cup_{\ell:z_i\in\Omega_\ell}B_{\ell,n}\}}
\neq
\mathbbm{1}_{\{\xi^{(n)}(z_i+
\cup_{\ell:z_i\in\Omega_\ell}B_{\ell,n})=1\}},
\end{align}
then $\xi^{(n)}(D^+(z_i,\tau_n))\geq3$.
Indeed, suppose $z_{i^*}\in z_i+\cup_{\ell:z_i\in\Omega_\ell}B_{\ell,n}$ while $\xi^{(n)}(z_i+\cup_{\ell:z_i\in\Omega_\ell}B_{\ell,n}) \neq1$. Then $\xi^{(n)}(z_i+\cup_{\ell:z_i\in\Omega_\ell}B_{\ell,n})
\geq2$, which implies $\xi^{(n)}(D^+(z_i,\tau_n))\geq3$ since $z_i+\cup_{\ell:z_i\in\Omega_\ell}B_{\ell,n}
\subset D^+(z_i,\tau_n)$ and $z_i\in D^+(z_i,\tau_n)\setminus (z_i+\cup_{\ell:z_i\in\Omega_\ell}B_{\ell,n})$. Conversely, suppose $\xi^{(n)}(z_i+\cup_{\ell:z_i\in\Omega_\ell}B_{\ell,n}) =1$ while
$z_{i^*}\notin z_i+ \cup_{\ell:z_i\in\Omega_\ell}B_{\ell,n}$. Then we must have $z_j\in z_i+\cup_{\ell:z_i\in\Omega_\ell}B_{\ell,n}
\subset D^+(z_i,\tau_n)$
for some $j\notin\{i,i^*\}$.
Since $|z_{i^*}-z_i|\leq|z_j-z_i|<\tau_n$,
we also have
$z_{i^*}\in D^+(z_i,\tau_n)\setminus
(z_i+\cup_{\ell:z_i\in\Omega_\ell}B_{\ell,n})$.
Thus \eqref{1 is not 1} indeed implies
$\xi^{(n)}(D^+(z_i,\tau_n))\geq3$.

From the above, we conclude that
\begin{multline}
 |\chi^{(n)}(J)-\xi^{(n)}_{J}(\Omega)| = \bigg|\sum_{i=1}^{n-1} \big( \mathbbm{1}_{\{z_{i}\in \Omega\}}\mathbbm{1}_{\{z_{i^*}\in z_i+\cup_{\ell:z_i\in\Omega_\ell}B_{\ell,n}\}} - \mathbbm{1}_{\{z_{i}\in \Omega\}}\mathbbm{1}_{\{\xi^{(n)}(z_i+\cup_{\ell:z_i\in\Omega_\ell}B_{\ell,n}) = 1\}} \big) \bigg| \\
\leq \sum_{i=1}^{n-1} \mathbbm{1}_{\{z_{i}\in \Omega\}} \mathbbm{1}_{\{\xi^{(n)}(D^+(z_i, \tau_{n}))\geq 3\}}\leq\Xi^{(3)}(\mathcal{E}), \label{lol25}
\end{multline}
where 
\begin{align*}
\Xi^{(3)} = \sum_{z_{i_1}, z_{i_2}, z_{i_3} \text{ distinct}}\delta_{(z_{i_1},z_{i_2},z_{i_3})}
\end{align*}
and
\begin{align*}
\mathcal{E} = \{(z, x_1, x_2): z\in \Omega \text{ and } (x_1, x_2)\in D^+(z, \tau_{n})^2\}.
\end{align*}
Since convergence in $L^1$ implies convergence in distribution, by \eqref{lol25} it is enough to show that $\mathbb{E}(\Xi^{(3)}(\mathcal{E}))\to 0$ as $n\to +\infty$. Using \eqref{def of DPP} and \eqref{def of rhok}, we obtain 
\begin{equation}\label{v3.3}
\mathbb{E}(\Xi^{(3)}(\mathcal{E})) = \int_{\Omega}d^{2}z\int_{D^+(z, \tau_{n})^2}\rho_3(z, x_1, x_2)d^{2}x_1d^{2}x_2
\end{equation}
where $\rho_3$ is given by
\begin{align}
& \rho_3(z, x_1, x_2) = \det \begin{pmatrix}
K_{n}(z,z) & K_{n}(z,x_{1}) & K_{n}(z,x_{2}) \\[0.15cm]
K_{n}(x_{1},z) & K_{n}(x_{1},x_{1}) & K_{n}(x_{1},x_{2}) \\[0.25cm]
K_{n}(x_{2},z) & K_{n}(x_{2},x_{1}) & K_{n}(x_{2},x_{2})
\end{pmatrix}. \label{v3.2 rho1}
\end{align} 
Let $\epsilon_{n}:=M \sqrt{\log n}/\sqrt{n}$, where $M$ is a fixed constant satisfying $M \geq 100/\sqrt{\min_{\mathcal{S}}\Delta Q}>0$ so that the conclusion of Theorem \ref{thm:Ameur bulk-bulk correlation improved} holds. We will analyze \eqref{v3.3} by splitting $\Omega$ in the following three regions:
\begin{align*}
& \Omega_{\mathrm{b}} = \{z\in \Omega \cap \mathcal{S} : \mathrm{dist}(z,\partial \mathcal{S}) \geq 2\epsilon_{n}\}, \\
& \Omega_{\mathrm{e}} = \{z\in \Omega : \mathrm{dist}(z,\partial \mathcal{S}) < 2\epsilon_{n}\}, \\
& \Omega_{\mathrm{ext}} = \{z\in \Omega \setminus \mathcal{S} : \mathrm{dist}(z,\partial \mathcal{S}) \geq 2\epsilon_{n}\}.
\end{align*}
The sets $\Omega_{\mathrm{b}}$, $\Omega_{\mathrm{e}}$ and $\Omega_{\mathrm{ext}}$ correspond to the parts of $\Omega$ lying in the bulk, edge, and exterior region, respectively. We start with $\Omega_{\mathrm{b}}$.

\underline{The contribution to \eqref{v3.3} from $\Omega_{\mathrm{b}}$.} Using Theorem \ref{thm:Ameur bulk-bulk correlation improved} and \eqref{v3.2 rho1}, we get $\rho_3(z, x_1, x_2) = n^{3} \mathrm{det}_{3}$, where $\mathrm{det}_{3}$ satisfies
\begin{align*}
\mathrm{det}_{3} = \det \begin{pmatrix}
\tilde{b}_{0}(z,\overline{z})+\bigO(\frac{1}{n}) & \tilde{b}_{0}(z,\overline{x_{1}})+\bigO(\frac{1}{n}) & \tilde{b}_{0}(z,\overline{x_{2}})+\bigO(\frac{1}{n}) \\[0.15cm]
\tilde{b}_{0}(x_{1},\overline{z})+\bigO(\frac{1}{n}) & \tilde{b}_{0}(x_{1},\overline{x_{1}})+\bigO(\frac{1}{n}) & \tilde{b}_{0}(x_{1},\overline{x_{2}})+\bigO(\frac{1}{n}) \\[0.25cm]
\tilde{b}_{0}(x_{2},\overline{z})+\bigO(\frac{1}{n}) & \tilde{b}_{0}(x_{2},\overline{x_{1}})+\bigO(\frac{1}{n}) & \tilde{b}_{0}(x_{2},\overline{x_{2}})+\bigO(\frac{1}{n})
\end{pmatrix}, \qquad \mbox{as } n \to + \infty,
\end{align*}
uniformly for $z\in \Omega_{\mathrm{b}}$ and $x_{1},x_{2}\in D^{+}(z,\tau_{n})$, with
\begin{align}\label{def of b0t}
\tilde{b}_{0}(z,\overline{w}) = \frac{b_{0}(z,\overline{w})}{\pi}e^{n \psi(z,\overline{w}) -n \frac{Q(z)+Q(w)}{2}}.
\end{align}
Define also
\begin{align*}
& \mathrm{det}_{3}^{(1)} := \det \begin{pmatrix}
\tilde{b}_{0}(z,\overline{z}) & \tilde{b}_{0}(z,\overline{x_{1}}) & \tilde{b}_{0}(z,\overline{x_{2}}) \\[0.15cm]
\tilde{b}_{0}(x_{1},\overline{z}) & \tilde{b}_{0}(x_{1},\overline{x_{1}}) & \tilde{b}_{0}(x_{1},\overline{x_{2}}) \\[0.25cm]
\tilde{b}_{0}(x_{2},\overline{z}) & \tilde{b}_{0}(x_{2},\overline{x_{1}}) & \tilde{b}_{0}(x_{2},\overline{x_{2}})
\end{pmatrix},  \\
& \mathrm{det}_{3}^{(2)} := \mathrm{det}_{3} - \mathrm{det}_{3}^{(1)}.
\end{align*}
A direct computation, which uses \eqref{asymp psi z near w}, shows that $\mathrm{det}_{3}^{(1)} = \bigO(n^{-1})$ as $n\to + \infty$. Moreover, we clearly have $\mathrm{det}_{3}^{(2)} = \bigO(n^{-1})$ as $n\to + \infty$. Hence, as $n\to + \infty$,
\begin{align}
& \rho_3(z, x_1, x_2) = \bigO( n^{2} ), \nonumber \\
& \int_{\Omega_{\mathrm{b}}} d^{2}z \int_{D^+(z, \tau_{n})^2} \rho_3(z, x_1, x_2)d^{2}x_1d^{2}x_2 = \bigO ( n^{-1} ). \label{contr Omega1}
\end{align}
\underline{The contribution to \eqref{v3.3} from $\Omega_{\mathrm{e}}$.} By assumption, $\partial \mathcal{S}$ is a smooth Jordan curve, and therefore $\Omega_{\mathrm{e}}$ is of size $\epsilon_{n}$ as $n\to + \infty$. On the other hand, \eqref{v3.2 rho1}, \eqref{Q minus Qhat lower bounded} and Proposition \ref{prop:pointwise estimate} imply that
\begin{align}\label{estimate rho3 general}
|\rho_{3}(z,x_{1},x_{2})| \leq 6 \, C_{1}^{3} n^{3} e^{-\frac{3n}{2}\min_{z_{*}\in\{z,x_{1},x_{2}\}}(Q(z_{*})-\widehat{Q}(z_{*}))}, \qquad z,x_{1},x_{2} \in \C.
\end{align}
Since $Q-\widehat{Q}\geq 0$ by \eqref{Q minus Qhat lower bounded}, we have $\rho_{3}(z,x_{1},x_{2}) = \bigO(n^{3})$ as $n\to + \infty$ uniformly for $z,x_{1},x_{2}\in \C$. This implies
\begin{align}
\bigg|\int_{\Omega_{\mathrm{e}}} d^{2}z \int_{D^+(z, \tau_{n})^2} \rho_3(z, x_1, x_2)d^{2}x_1d^{2}x_2\bigg| \leq C  \int_{\Omega_{\mathrm{e}}}d^{2}z = \bigO(\epsilon_{n}), \qquad \mbox{as } n \to + \infty. \label{contr Omega2}
\end{align}

\underline{The contribution to \eqref{v3.3} from $\Omega_{\mathrm{ext}}$.}
Since $\tau_n=\tau n^{-3/4} = o(\epsilon_n)$, for all sufficiently large $n$,
$z\in\Omega_{\mathrm{ext}}$ and $x_1,x_2\in D^+(z,\tau_n)$ imply
\begin{align*}
\operatorname{dist}(x_j,\mathcal{S})
\geq \operatorname{dist}(z,\mathcal{S})-|x_j-z|
\geq 2\epsilon_n-\tau_n
\geq\epsilon_n,
\qquad j=1,2.
\end{align*}
Thus, by Hadamard's inequality (see also Lemma \ref{lemma1} below), Proposition
\ref{prop:pointwise estimate} and
\eqref{Q minus Qhat lower bounded},
\begin{align*}
0\leq\rho_3(z,x_1,x_2)
\leq\rho_1(z)\rho_1(x_1)\rho_1(x_2)
\leq Cn^2e^{-cn\epsilon_n^2}\rho_1(z)
\end{align*}
for some $C,c>0$, uniformly for $z\in\Omega_{\mathrm{ext}}$ and $x_1,x_2\in D^+(z,\tau_n)$.

Using $\int_{\C}\rho_1(z)d^2z=n$ and
$\tau_n=\tau n^{-3/4}$, and increasing $M>0$ if necessary, we thus find
\begin{align}\label{contr Omega3}
&\int_{\Omega_{\mathrm{ext}}}d^2z
\int_{D^+(z,\tau_n)^2}
\rho_3(z,x_1,x_2)d^2x_1d^2x_2
\nonumber\\
&\qquad\leq
C_2n^2\tau_n^4e^{-cn\epsilon_n^2}
\int_{\Omega_{\mathrm{ext}}}\rho_1(z)d^2z
\leq C_3e^{-cn\epsilon_n^2}
=\bigO(n^{-100}),
\end{align}
as $n\to+\infty$ for some $C_{2},C_{3}>0$.
The claim now directly follows from \eqref{contr Omega1},
\eqref{contr Omega2} and \eqref{contr Omega3}.
\end{proof}

Our next task is to prove that 
\begin{align}\label{conv of xi A Omega to Poisson}
\xi^{(n)}_{J}(\Omega)\oset{\mathrm{law}}{\underset{n\to \infty}{\xrightarrow{\hspace*{0.85cm}}}}\mathrm{Poisson}(\lambda_{J})
\end{align}
holds for any $m\in\N_{>0}$ and bounded Borel sets
$B_1,\ldots,B_m\subset\HH_+$ and
$\Omega_1,\ldots,\Omega_m\subset\C$
such that the rectangles $B_\ell\times\Omega_\ell$
are pairwise disjoint, with $J=\bigcup_{\ell=1}^{m}(B_\ell\times\Omega_\ell)$ and $\Omega=\bigcup_{\ell=1}^{m}\Omega_\ell$. By Lemma \ref{lemma2}, $\chi^{(n)}(J)$ and $\xi^{(n)}_{J}(\Omega)$ have the same asymptotic distribution, so this would imply \eqref{eqn2} as desired. To establish \eqref{conv of xi A Omega to Poisson}, we will follow the same strategy as in \cite{SJ2012}: we will show that the factorial moments of $\xi^{(n)}_{J}(\Omega)$ converge to that of a Poisson random variable with intensity $\lambda_{J}$, i.e. that
\begin{equation}\label{eqn3}
\lim_{n\to + \infty}\mathbb{E}\biggl[\frac{\xi^{(n)}_{J}(\Omega)!}{(\xi^{(n)}_{J}(\Omega)-k)!}\biggl] = \lambda_{J}^{k}, \qquad \mbox{ for all } k\geq 1.
\end{equation}
For a proof that \eqref{eqn3} indeed implies \eqref{conv of xi A Omega to Poisson}, see e.g. \cite[Proposition 2.8]{LM2024}. By \eqref{def of DPP 2} with $m=1$, $k_{1}=k$ and $\Lambda_{1}=\Omega$, \eqref{eqn3} can be rewritten as
\begin{equation}\label{v3.5}
\lim_{n\to + \infty}\int_{\Omega^k}\rho_k^{J}(w_1, \ldots, w_k)d^{2}w_1\ldots d^{2}w_k  = \lambda_{J}^k, \qquad \mbox{ for all } k\geq 1.
\end{equation}
To establish (\ref{v3.5}) for any $J$ and $\Omega$ as above, we will justify each of the following equalities:
\begin{align}
& \lim_{n\to + \infty}\int_{\Omega^k}\rho_k^{J}(w_1, \ldots, w_k)d^{2}w_1\ldots d^{2}w_k = \lim_{n\to + \infty}\int_{\Omega_{0}^k}\rho_k^{J}(w_1, \ldots, w_k)d^{2}w_1\ldots d^{2}w_k \nonumber \\
& = \int_{\Omega^{k}}\lim_{n\to + \infty}\mathbbm{1}_{\Omega_{0}^k}\rho_k^{J}(w_1, \ldots, w_k)d^{2}w_1\ldots d^{2}w_k \nonumber \\
& = \int_{\Omega^k}
\prod_{i=1}^k
\sum_{\ell=1}^m \pi^2\rho(w_i)^3
\mathbbm{1}_{\Omega_\ell}(w_i)
\bigg(\frac1\pi\int_{B_\ell}|u|^2d^2u\bigg)
d^2w_1\ldots d^2w_k = \lambda_{J}^k, \label{strategy of proof}
\end{align}
where we recall from \eqref{def of mu and rho} that $\rho(z)=\frac{\Delta Q(z)}{4\pi} \mathbf{1}_{\mathcal{S}}(z)$, 
\begin{align*}
\Omega_{0}^{k} := \{(w_1, \ldots, w_k)\in \Omega^k: |w_{i}-w_{j}|>2\tau_{n} \mbox{ for all } 1\leq i < j\leq k\},
\end{align*}
$\tau_{n}:=\tau n^{-3/4}$ and $\tau>0$ is fixed and such that $\cup_{\ell=1}^{m}B_\ell \subset \{u\in \C:|u|<\tau\}$. 

The last equality in \eqref{strategy of proof} directly follows from the definition \eqref{def of lambda A Omega}. Hence, it only remains to justify the first three equalities in \eqref{strategy of proof}.

We begin with the third equality, using the following lemma from \cite{SJ2012}.

\begin{lemma}\label{lemma1}\cite[Lemma 3.2]{SJ2012} 
For any $n\times n$ Hermitian positive definite matrix $M$ and set $\omega\subseteq\{1, 2, \ldots, n\}$, we have
\begin{align*}
\det(M)\leq\det(M_\omega)\det(M_{\omega^{c}}),
\end{align*}
where $M_\omega$ (resp. $M_{\omega^{c}}$) denotes the submatrix of $M$ using rows and columns indexed by $\omega$ (resp. $\{1, \ldots, n\}\setminus\omega$).
\end{lemma}
We now use Lemma \ref{lemma1} to prove the third equality in \eqref{strategy of proof}.
\begin{lemma}\label{lemma3}
(The third equality in \eqref{strategy of proof})
For any $m\in\N_{>0}$, bounded Borel sets
$B_1,\ldots,B_m\subset\HH_+$ and
$\Omega_1,\ldots,\Omega_m\subset\C$ such that the rectangles $B_\ell\times\Omega_\ell$
are pairwise disjoint, set $J=\bigcup_{\ell=1}^{m}(B_\ell\times\Omega_\ell)$ and $\Omega=\bigcup_{\ell=1}^{m}\Omega_\ell$.
Let $w_1,\ldots,w_k\in\C\setminus\partial\mathcal{S}$
be pairwise distinct. Then
\begin{align}\label{lol16}
\lim_{n\to+\infty}\rho_k^{J}(w_1,\ldots,w_k)
=
\prod_{i=1}^k
\sum_{\ell=1}^{m} \pi^2\rho(w_i)^3
\mathbbm{1}_{\Omega_\ell}(w_i)
\bigg(\frac1\pi\int_{B_\ell}|u|^2d^2u\bigg).
\end{align}
\end{lemma}

%
%

\begin{proof} 
Let $w_{1},\ldots,w_{k}\in \C$ be pairwise distinct. Then there exists $n_{0}\in \N$ such that $|w_{i}-w_{j}|>2\tau_{n}$ for all $i<j$ and all $n\geq n_{0}$.
The correlation function $\rho_{k}^{J}(w_{1},\ldots,w_{k})$ can be computed in terms of the $\rho_{j}$'s using an inclusion-exclusion argument, as mentioned in \cite[(4.5)]{SoshnikovSurvey}. 

Indeed, since the density of \eqref{general density intro} is continuous, for all $n\geq n_{0}$ we have 
\begin{align}
& \rho_{k}^{J}(w_{1},\ldots,w_{k}) = \E\bigg[ \sum_{\substack{\ell_{1},\ldots,\ell_{k}=1,\ldots,n\\ z_{\ell_{j}}\in \xi_{J}^{(n)}}} \prod_{j=1}^{k} \delta_{w_{j}}(z_{\ell_{j}}) \bigg] = \E\bigg[ \sum_{\ell_{1},\ldots,\ell_{k}=1}^{n} \prod_{j=1}^{k} \delta_{w_{j}}(z_{\ell_{j}})\mathbbm{1}_{\{\xi^{(n)}(w_{j}+\cup_{\ell:w_j\in\Omega_\ell}B_{\ell,n})=1\}} \bigg] \nonumber \\
& = \E\bigg[ \sum_{\ell_{1},\ldots,\ell_{2k}=1}^{n} \prod_{j=1}^{k} \delta_{w_{j}}(z_{\ell_{j}})\mathbbm{1}_{\{z_{\ell_{k+j}}\in w_{j}+\cup_{\ell:w_j\in\Omega_\ell}B_{\ell,n}\}} \prod_{\substack{j=1 \\j \neq \ell_{1},\ldots,\ell_{2k}}}^{n}\big( 1-\mathbbm{1}_{\{z_{j}\in \tilde{B}_{n,k}\}}\big) \bigg] \nonumber \\
& = \sum_{r=0}^{+\infty}\frac{(-1)^r}{r!}\int_{w_1+\cup_{\ell:w_1\in\Omega_\ell}B_{\ell,n}}d^{2}x_1\ldots\int_{w_k+\cup_{\ell:w_k\in\Omega_\ell}B_{\ell,n}}d^{2}x_k \nonumber \\
& \hspace{3cm} \int_{\tilde{B}_{n,k}^{r}}\rho_{2k+r}(w_1, x_1, \ldots, w_k, x_k, y_1, \ldots, y_r)d^{2}y_1\ldots d^{2}y_r, \label{eqn5}
\end{align}
where 
\begin{align*}
\tilde{B}_{n,k}:=\bigg(w_1+\bigcup_{\ell:w_1\in\Omega_\ell}B_{\ell,n}\bigg) \sqcup \ldots \sqcup\bigg(w_k+\bigcup_{\ell:w_k\in\Omega_\ell}B_{\ell,n}\bigg)
\end{align*}
(In the first sums in \eqref{eqn5}, we do not need to require $\ell_{1},\ldots,\ell_{2k}$ to be distinct because $|w_i-w_j|>2\tau_n$ for $i \neq j$ and $B_{\ell,n}\subset\{u:0\prec u,\ |u|<\tau_n\}$ for every $\ell$. Note also that for $2k+r>n$ we have $\rho_{2k+r}=0$, and therefore the above series contains finitely many terms.)

We first treat the case where $w_i\in \mathcal{S}^{\circ}$ for all $i\in \{1,\ldots,k\}$. We begin with the term corresponding to $r=0$ in \eqref{eqn5}, and then show that the terms corresponding to $r\geq 1$ are negligible. It follows from \eqref{def of rhok} that
\begin{align}
\rho_{2k}(w_1, x_1, \ldots, w_k, x_k) = \det_{1 \leq i,j \leq k} \begin{pmatrix}
K_{n}(w_{i},w_{j}) & K_{n}(w_{i},x_{j}) \\
K_{n}(x_{i},w_{j}) & K_{n}(x_{i},x_{j})
\end{pmatrix}.\label{eqn6}
\end{align}
Theorem \ref{thm:bulk damping} implies a fast decay for the off-diagonal $2\times 2$ blocks: indeed, if $i\neq j$, then there exists $c>0$ such that
\begin{align}\label{lol18}
\begin{pmatrix}
K_{n}(w_{i},w_{j}) & K_{n}(w_{i},x_{j}) \\
K_{n}(x_{i},w_{j}) & K_{n}(x_{i},x_{j})
\end{pmatrix} = \begin{pmatrix}
\bigO(e^{-c\sqrt{n}}) & \bigO(e^{-c\sqrt{n}}) \\
\bigO(e^{-c\sqrt{n}}) & \bigO(e^{-c\sqrt{n}})
\end{pmatrix}, \qquad \mbox{as } n \to + \infty
\end{align}
uniformly for $x_i\in w_i+\cup_{\ell:w_i\in\Omega_\ell}B_{\ell,n}$
and $x_j\in w_j+\cup_{\ell:w_j\in\Omega_\ell}B_{\ell,n}$.

On the other hand, for the diagonal case $i=j$, Theorem \ref{thm:Ameur bulk-bulk correlation improved} implies
\begin{align*}
\det \begin{pmatrix}
K_{n}(w_{i},w_{i}) & K_{n}(w_{i},x_{i}) \\
K_{n}(x_{i},w_{i}) & K_{n}(x_{i},x_{i})
\end{pmatrix} = n^{2} \mathrm{det}_{2}, 
\end{align*}
where $\mathrm{det}_{2}$ satisfies
\begin{align*}
\mathrm{det}_{2} = \det \begin{pmatrix}
\tilde{b}_{0}(w_{i},\overline{w_{i}})+\bigO(\frac{1}{n}) & \tilde{b}_{0}(w_{i},\overline{x_{i}})+\bigO(\frac{1}{n})  \\[0.15cm]
\tilde{b}_{0}(x_{i},\overline{w_{i}})+\bigO(\frac{1}{n}) & \tilde{b}_{0}(x_{i},\overline{x_{i}})+\bigO(\frac{1}{n}) 
\end{pmatrix}, \qquad \mbox{as } n \to + \infty
\end{align*}
uniformly for $x_i\in w_i+\cup_{\ell:w_i\in\Omega_\ell}B_{\ell,n}$, and $\tilde{b}_{0}$ is given by \eqref{def of b0t}. Let us also define
\begin{align*}
& \mathrm{det}_{2}^{(1)} := \det \begin{pmatrix}
\tilde{b}_{0}(w_{i},\overline{w_{i}}) & \tilde{b}_{0}(w_{i},\overline{x_{i}})  \\[0.15cm]
\tilde{b}_{0}(x_{i},\overline{w_{i}}) & \tilde{b}_{0}(x_{i},\overline{x_{i}}) 
\end{pmatrix}, \qquad \mathrm{det}_{2}^{(2)} := \mathrm{det}_{2} - \mathrm{det}_{2}^{(1)}.
\end{align*}
A direct computation, which uses
\begin{align*}
L_{z}(\lambda)|_{z=\lambda} = \partial_{z}\partial_{w}L_{z}(\lambda)|_{z=\lambda}=0, \qquad \partial_{z}L_{z}(\lambda)|_{z=\lambda} = -\partial_{w}L_{z}(\lambda)|_{z=\lambda},
\end{align*} 
and the parametrization $x_{i} = w_{i} + u_{i}n^{-\frac{3}{4}}$, $u_{i}\in \cup_{\ell:w_i\in\Omega_\ell}B_\ell$, gives
\begin{align*}
\mathrm{det}_{2}^{(1)} = \frac{|u_{i}|^{2}}{\pi}  b_{0}(w_{i},\overline{w_{i}})^{2} \frac{\Delta Q(w_{i})}{4\pi}\frac{1}{\sqrt{n}} + \bigO(n^{-1}), \qquad \mbox{as } n \to + \infty
\end{align*}
uniformly for $u_i\in\cup_{\ell:w_i\in\Omega_{\ell}}B_{\ell}$. We also trivially have $\mathrm{det}_{2}^{(2)} = \bigO(n^{-1})$ as $n\to + \infty$. Since $b_{0}(w_{i},\overline{w_{i}}) = \frac{\Delta Q(w_{i})}{4} = \pi\rho(w_{i})$, this gives 
\begin{align}\label{lol17}
\det \begin{pmatrix}
K_{n}(w_{i},w_{i}) & K_{n}(w_{i},x_{i}) \\
K_{n}(x_{i},w_{i}) & K_{n}(x_{i},x_{i})
\end{pmatrix} = \frac{|u_{i}|^{2}}{\pi}  \pi^{2}\rho(w_{i})^{3}n^{\frac{3}{2}} + \bigO(n), \qquad \mbox{as } n \to + \infty
\end{align}
uniformly for $u_{i} \in \cup_{\ell:w_i\in\Omega_{\ell}}B_{\ell}$. Combining \eqref{lol17} with \eqref{lol18}, we get
\begin{align*}
& \int_{w_1+\cup_{\ell:w_1\in\Omega_{\ell}}B_{\ell,n}}d^{2}x_1\ldots\int_{w_k+\cup_{\ell:w_k\in\Omega_{\ell}}B_{\ell,n}}d^{2}x_k \rho_{2k}(w_1, x_1, \ldots, w_k, x_k) \\
& = \prod_{i=1}^k\int_{w_i+\cup_{\ell:w_i\in\Omega_{\ell}}B_{\ell,n}}\rho_2(w_i, x_i)d^{2}x_i (1+\bigO(e^{-c\sqrt{n}})) \\
& = \prod_{i=1}^k\biggl(
\int_{\cup_{\ell:w_i\in\Omega_{\ell}}B_{\ell}}\frac{|u_{i}|^{2}}{\pi}  \pi^{2}\rho(w_{i})^{3} d^{2}u_i
\biggl)(1+\bigO(n^{-1/2})) \\
& =  \prod_{i=1}^k \pi^{2}\rho(w_{i})^{3} \bigg( \frac{1}{\pi} \int_{\cup_{\ell:w_i\in\Omega_{\ell}}B_{\ell}}|u|^{2}d^{2}u \bigg) (1+\bigO(n^{-1/2})).
\end{align*}
We now show that the terms corresponding to $r\geq1$ in \eqref{eqn5} are negligible. By Lemma \ref{lemma1} we have, for $r\geq1$,
\begin{align*}
\rho_{2k+r}(w_1, x_1, \ldots, w_{k}, x_k, y_1, \ldots, y_r)\leq\rho_{2k}(w_1, x_1, \ldots, w_k, x_k)\prod_{i=1}^r\rho_1(y_i).
\end{align*}
Thus the sum corresponding to $r\geq1$ in \eqref{eqn5} is bounded by
\begin{multline}\label{v3.6}
\bigg(\int_{w_1+\cup_{\ell:w_1\in\Omega_{\ell}}B_{\ell,n}}d^{2}x_1\ldots\int_{w_k+\cup_{\ell:w_k\in\Omega_{\ell}}B_{\ell,n}}d^{2}x_k\rho_{2k}(w_1, x_1, \ldots, w_k, x_k)\bigg) \\
\times \sum_{r\geq1}\frac{1}{r!}\bigg(\int_{\tilde{B}_{n,k}}\rho_1(y)d^{2}y\bigg)^r.
\end{multline}
Proposition \ref{prop:pointwise estimate} implies that $\rho_{1}(y) = \bigO(n)$ as $n\to + \infty$ uniformly for $y\in \C$. Hence
\begin{align}\label{lol19}
\sum_{r\geq1}\frac{1}{r!}\bigg(\int_{\tilde{B}_{n,k}}\rho_1(y)d^{2}y\bigg)^r = \bigO(n^{-\frac{1}{2}}), \qquad \mbox{as } n \to + \infty,
\end{align}
and therefore
\begin{align*}
\rho_k^{J}(w_1, \ldots, w_k) =  \prod_{i=1}^k \pi^{2}\rho(w_{i})^{3} \bigg( \frac{1}{\pi} \int_{\cup_{\ell:w_i\in\Omega_{\ell}}B_{\ell}}|u|^{2}d^{2}u \bigg) \big( 1+\bigO(n^{-1/2}) \big), \qquad \mbox{as } n \to + \infty.
\end{align*}
This proves the claim when $w_{1},\ldots,w_{k} \in \mathcal{S}^{\circ}$.

Let $p \geq 1$. We now consider the case where $w_{1},\ldots,w_{p} \notin \mathcal{S}$ and $w_{p+1},\ldots,w_{k} \in \mathcal{S}^{\circ}$. Then, using twice Lemma \ref{lemma1} in \eqref{eqn5}, we get
\begin{align}
& \rho_k^{J}(w_1, \ldots, w_k) \nonumber \\
& \leq \biggl(\int_{w_1+\cup_{\ell:w_1\in\Omega_{\ell}}B_{\ell,n}}d^{2}x_1\ldots\int_{w_k+\cup_{\ell:w_k\in\Omega_{\ell}}B_{\ell,n}} \hspace{-0.25cm} d^{2}x_k\rho_{2k}(w_1, x_1, \ldots, w_k, x_k)\biggl) \sum_{r\geq0}\frac{1}{r!}\biggl(\int_{\tilde{B}_{n,k}}\rho_1(y)d^{2}y\biggl)^{r} \nonumber \\
& \leq \prod_{i=1}^k\biggl(\int_{w_i+\cup_{\ell:w_i\in\Omega_{\ell}}B_{\ell,n}}\rho_2(w_i, x)d^{2}x\biggl)\sum_{r\geq0}\frac{1}{r!}\biggl(\int_{\tilde{B}_{n,k}}\rho_1(y)d^{2}y\biggl)^{r}. \label{lol20}
\end{align}
Since $\rho_{1}(y) = \bigO(n)$ as $n\to + \infty$ uniformly for $y\in \C$, the above series is bounded (recall \eqref{lol19}). Let us now analyze the product on the right-hand side of \eqref{lol20}.
By \eqref{lol17}, for $i >p$, we have
\begin{align}\label{lol21}
\int_{w_i + \cup_{\ell:w_i\in\Omega_{\ell}}B_{\ell,n}}\rho_2(w_i, x)d^{2}x = 
\frac{1}{\pi} \int_{\cup_{\ell:w_i\in\Omega_{\ell}}B_{\ell}}|u|^{2}d^{2}u \times  \pi^{2}\rho(w_{i})^{3}+\bigO(n^{-1/2}), \quad \mbox{as } n \to +\infty.
\end{align} 
On the other hand, for $i\leq p$, we have by \eqref{Q minus Qhat lower bounded} and Proposition \ref{prop:pointwise estimate} that
\begin{align}\label{estimate rho2 esterior}
|\rho_{2}(w_{i},x_{i})| = \bigO(e^{-cn}), \qquad \mbox{as } n \to + \infty
\end{align}
uniformly for $x_{i} \in w_i +\cup_{\ell:w_i\in\Omega_{\ell}}B_{\ell,n}$ for some $c>0$, and thus
\begin{align}\label{lol22}
\int_{w_i + \cup_{\ell:w_i\in\Omega_{\ell}}B_{\ell,n}}\rho_2(w_i, x)d^{2}x = \bigO(e^{-cn}), \qquad \mbox{as } n \to + \infty.
\end{align}
Substituting the above in \eqref{lol20} yields $\rho_k^{J}(w_1, \ldots, w_k) = \bigO(e^{-cn})$ as $n\to + \infty$. This proves the claim when at least one $w_{j}$ lies in $\C\setminus \mathcal{S}$.
\end{proof}
The next lemma, combined with Lebesgue's dominated convergence theorem, justifies the second equality in \eqref{strategy of proof}. 
\begin{lemma}\label{lemma4}(The second equality in \eqref{strategy of proof}) For any $m\in\N_{>0}$, bounded Borel sets
$B_1,\ldots,B_m\subset\HH_+$ and
$\Omega_1,\ldots,\Omega_m\subset\C$, set $J=\bigcup_{\ell=1}^{m}(B_\ell\times\Omega_\ell)$ and $\Omega=\bigcup_{\ell=1}^{m}\Omega_\ell$. As $n \to + \infty$, $\rho_k^{J}(w_1, \ldots, w_k) = \bigO(1)$ uniformly for $(w_1, \ldots, w_k)\in \Omega_{0}^{k}$.
\end{lemma}

%
%

\begin{proof} 
Note that the inequalities in \eqref{lol20} hold for any $(w_1, \ldots, w_k)\in \Omega_{0}^{k}$, and we already proved that the series in \eqref{lol20} is bounded. Hence, to prove the claim, it suffices to prove that 
\begin{align*}
\int_{w+\cup_{\ell:w\in\Omega_{\ell}}B_{\ell,n}}\rho_2(w, x)d^{2}x
\end{align*}
is bounded as $n\to + \infty$ uniformly for $w \in \C$. Let $\epsilon_{n}:=M \sqrt{\log n}/\sqrt{n}$, where $M$ is a fixed constant satisfying $M \geq 100/\sqrt{\min_{\mathcal{S}}\Delta Q}>0$ so that the conclusion of Theorem \ref{thm:Ameur bulk-bulk correlation improved} holds. As in the proof of Lemma \ref{lemma2}, we will consider three cases: 
\begin{enumerate}
\item $w\in \mathcal{S}$ and $\mathrm{dist}(w,\partial \mathcal{S}) \geq 2\epsilon_{n}$,
\item \vspace{-0.1cm} $\mathrm{dist}(w,\partial \mathcal{S}) \leq 2\epsilon_{n}$,
\item \vspace{-0.1cm} $w\notin \mathcal{S}$ and $\mathrm{dist}(w,\partial \mathcal{S}) \geq 2\epsilon_{n}$.
\end{enumerate}
The first and third cases can be analyzed as in Lemma \ref{lemma3}; the only difference with Lemma \ref{lemma3} is that here $w$ is allowed to approach $\partial \mathcal{S}$ as $n\to + \infty$, while in Lemma \ref{lemma3} $w$ was fixed. However, since Theorem \ref{thm:Ameur bulk-bulk correlation improved} still applies, the proof remains essentially unchanged. For the first case, we find that \eqref{lol21} still holds. In particular, $\int_{w + \cup_{\ell:w\in\Omega_{\ell}}B_{\ell,n}}\rho_2(w, x)d^{2}x=\bigO(1)$ as $n\to + \infty$ uniformly for $w\in \mathcal{S}$ such as $\mathrm{dist}(w,\partial \mathcal{S}) \geq 2\epsilon_{n}$. For the third case, in a similar way as \eqref{lol22}, using \eqref{Q minus Qhat lower bounded} and Proposition \ref{prop:pointwise estimate} we find (increasing $M$ if necessary)
\begin{align*}
\int_{w + \cup_{\ell:w\in\Omega_{\ell}}B_{\ell,n}}\rho_2(w, x)d^{2}x \leq C e^{-cn \epsilon_{n}^{2}} = \bigO(n^{-100}), \qquad \mbox{as } n \to + \infty
\end{align*}
uniformly for $w\notin\mathcal{S}$ such that
$\mathrm{dist}(w,\partial\mathcal{S})\geq2\epsilon_n$, for some constants $C,c>0$.
It remains to analyze the second case, i.e. when $\mathrm{dist}(w,\partial \mathcal{S}) \leq 2\epsilon_{n}$. For this, we use the parametrizations
\begin{align*}
w = z_{0} + \mathrm{n} \frac{2\sqrt{2}\xi}{\sqrt{\Delta Q(z_{0})n}}, \qquad x = z_{0} + \mathrm{n} \frac{2\sqrt{2}\eta}{\sqrt{\Delta Q(z_{0})n}}, \qquad |\xi|,|\eta| \leq \sqrt{\mathrm{max}_{\mathcal{S}}\Delta Q} \; M \sqrt{\log n},
\end{align*}
where $z_{0}\in \partial \mathcal{S}$ and $\mathrm{n}$ is the outward pointing unit normal to $\partial \mathcal{S}$ at $z_{0}$. The condition $x\in w+\cup_{\ell:w\in\Omega_{\ell}}B_{\ell,n}$ implies that $\xi,\eta$ satisfy $|\eta-\xi|=\bigO(n^{-1/4})$ as $n\to + \infty$.
Using Theorem \ref{thm:Charlier}, we obtain (using that the unimodular functions $c_{n}$ cancel out when considering $\rho_{2}$)
\begin{align}\label{lol51}
\rho_{2}(w,x) = \rho(z_{0})^{2}n^{2} \mathrm{det}_{2,\mathrm{e}},
\end{align}
where $\mathrm{det}_{2,\mathrm{e}}$ satisfies
\begin{align*}
& \mathrm{det}_{2,\mathrm{e}} = \det \begin{pmatrix}
\mathrm{k}(\xi,\xi)+\frac{\mathrm{k}_{2}(\xi,\xi;z_{0})}{\rho(z_{0})\sqrt{n}} + \bigO(\frac{(\log n)^{3}}{n}) & \mathrm{k}(\xi,\eta)+\frac{\mathrm{k}_{2}(\xi,\eta;z_{0})}{\rho(z_{0})\sqrt{n}} + \bigO(\frac{(\log n)^{3}}{n})  \\[0.15cm]
\mathrm{k}(\eta,\xi)+\frac{\mathrm{k}_{2}(\eta,\xi;z_{0})}{\rho(z_{0})\sqrt{n}} + \bigO(\frac{(\log n)^{3}}{n}) & \mathrm{k}(\eta,\eta)+\frac{\mathrm{k}_{2}(\eta,\eta;z_{0})}{\rho(z_{0})\sqrt{n}} + \bigO(\frac{(\log n)^{3}}{n}) 
\end{pmatrix}, \quad \mbox{as } n \to + \infty
\end{align*}
uniformly for $|\xi|,|\eta| \leq \sqrt{\max_{\mathcal S}\Delta Q}\,M \sqrt{\log n}$ and $z_{0}\in \partial \mathcal{S}$. Let us also define
\begin{align*}
& \mathrm{det}_{2,\mathrm{e}}^{(1)} := \det \begin{pmatrix}
\mathrm{k}(\xi,\xi) & \mathrm{k}(\xi,\eta)  \\[0.15cm]
\mathrm{k}(\eta,\xi) & \mathrm{k}(\eta,\eta) 
\end{pmatrix}, \qquad \mathrm{det}_{2,\mathrm{e}}^{(2)}:= \det \begin{pmatrix}
\mathrm{k}(\xi,\xi)+\frac{\mathrm{k}_{2}(\xi,\xi;z_{0})}{\rho(z_{0})\sqrt{n}} & \mathrm{k}(\xi,\eta)+\frac{\mathrm{k}_{2}(\xi,\eta;z_{0})}{\rho(z_{0})\sqrt{n}}   \\[0.15cm]
\mathrm{k}(\eta,\xi)+\frac{\mathrm{k}_{2}(\eta,\xi;z_{0})}{\rho(z_{0})\sqrt{n}} & \mathrm{k}(\eta,\eta)+\frac{\mathrm{k}_{2}(\eta,\eta;z_{0})}{\rho(z_{0})\sqrt{n}} \end{pmatrix}, 
\end{align*}
and $\mathrm{det}_{2,\mathrm{e}}^{(3)}:= \mathrm{det}_{2,\mathrm{e}} - \mathrm{det}_{2,\mathrm{e}}^{(2)}$. Using the definition of $\mathrm{k}$ and the fact that $\mathrm{k}(\xi,\eta)$ is bounded for $\xi,\eta \in \C$ (this is proved in Appendix \ref{appendix:k bounded}), we infer that for each $\alpha_{1},\alpha_{2},\alpha_{3},\alpha_{4}\in \N$, there exists $C_{\alpha_{1},\alpha_{2},\alpha_{3},\alpha_{4}}>0$ such that
\begin{align}\label{lol46}
|\partial_{\re \xi}^{\alpha_{1}}\partial_{\im \xi}^{\alpha_{2}}\partial_{\re \eta}^{\alpha_{3}}\partial_{\im \eta}^{\alpha_{4}}\mathrm{k}(\xi,\eta)| \leq C_{\alpha_{1},\alpha_{2},\alpha_{3},\alpha_{4}} (1+|\xi|+|\eta|)^{\alpha_{1}+\alpha_{2}+\alpha_{3}+\alpha_{4}}, \qquad \mbox{for all } \xi,\eta \in \C.
\end{align} 
Writing $\eta = \xi + \frac{u}{n^{1/4}}$, we obtain
\begin{align}\label{lol45}
& \mathrm{det}_{2,\mathrm{e}}^{(1)} = \frac{|u|^{2}}{2\pi n^{1/2}}F(\xi) + \bigO\big((1+|\xi|^{3})n^{-3/4}\big), \qquad \mbox{as } n \to + \infty,
\end{align}
uniformly for $u$ in compact subsets of $\C$, for $|\xi|\leq \sqrt{\max_{\mathcal S}\Delta Q}\,M\sqrt{\log n}$ and for $z_{0}\in \partial \mathcal{S}$, with
\begin{align*}
F(\xi) := -2e^{-8(\re \xi)^{2}} + 4 \sqrt{\pi} (\re \xi) e^{-4 (\re \xi)^{2}} \mathrm{erfc}(2 \, \re \xi) + \pi  \mathrm{erfc}(2 \, \re \xi)^{2}.
\end{align*}
For the error term in \eqref{lol45}, we have used Taylor's theorem and \eqref{lol46} with $\alpha_{1}+\alpha_{2}+\alpha_{3}+\alpha_{4}\leq 3$. Since $\erfc(x)$ is bounded for $x\in \R$, $F(\xi)$ is bounded for $\xi \in \C$. Thus 
\begin{align}\label{lol48}
\mathrm{det}_{2,\mathrm{e}}^{(1)}=\bigO(n^{-1/2}), \qquad \mbox{as } n \to + \infty
\end{align}
uniformly for $z_{0}\in \partial \mathcal{S}$ and for $|\xi|,|\eta| \leq \sqrt{\max_{\mathcal S}\Delta Q}\,M \sqrt{\log n}$ such that $|\eta-\xi|=\bigO(n^{-1/4})$. In a similar way as \eqref{lol46}, using the definition of $\mathrm{k}_{2}$ we infer that for each $\alpha_{1},\ldots,\alpha_{4}\in \N$, there exists $\tilde{C}_{\alpha_{1},\alpha_{2},\alpha_{3},\alpha_{4}}>0$ such that
\begin{multline}\label{lol47}
|\partial_{\re \xi}^{\alpha_{1}}\partial_{\im \xi}^{\alpha_{2}}\partial_{\re \eta}^{\alpha_{3}}\partial_{\im \eta}^{\alpha_{4}}\mathrm{k}_{2}(\xi,\eta;z_0)| \leq \tilde{C}_{\alpha_{1},\alpha_{2},\alpha_{3},\alpha_{4}} (1+|\xi|+|\eta|)^{3+\alpha_{1}+\alpha_{2}+\alpha_{3}+\alpha_{4}}, \\ \mbox{for all } \xi,\eta \in \C \mbox{ and }z_0\in\partial\mathcal S.
\end{multline}
Using Taylor's theorem and \eqref{lol46} and \eqref{lol47} with $\alpha_{1}+\alpha_{2}+\alpha_{3}+\alpha_{4} \leq 2$, we get
\begin{align}\label{lol49}
\mathrm{det}_{2,\mathrm{e}}^{(2)}-\mathrm{det}_{2,\mathrm{e}}^{(1)} = \bigO\bigg( \frac{(1+|\xi|+|\eta|)^{5}}{n} \bigg), \qquad \mbox{as } n \to + \infty
\end{align}
uniformly for $z_{0}\in \partial \mathcal{S}$ and for $|\xi|,|\eta| \leq \sqrt{\max_{\mathcal S}\Delta Q}\,M \sqrt{\log n}$ such that $|\eta-\xi|=\bigO(n^{-1/4})$. Finally, we also trivially have
\begin{align}\label{lol50}
\mathrm{det}_{2,\mathrm{e}}^{(3)} = \bigO\bigg( \frac{(\log n)^{3}}{n} \bigg), \qquad \mbox{as } n \to + \infty
\end{align}
uniformly for $z_{0}\in \partial \mathcal{S}$ and for $|\xi|,|\eta| \leq \sqrt{\max_{\mathcal S}\Delta Q}\,M \sqrt{\log n}$ such that $|\eta-\xi|=\bigO(n^{-1/4})$. Combining \eqref{lol48}, \eqref{lol49} and \eqref{lol50} yields 
\begin{align*}
\mathrm{det}_{2,\mathrm{e}} =\bigO(n^{-1/2}), \qquad \mbox{as } n \to + \infty
\end{align*}
uniformly for $z_{0}\in \partial \mathcal{S}$ and for $|\xi|,|\eta| \leq \sqrt{\max_{\mathcal S}\Delta Q}\,M \sqrt{\log n}$ such that $|\eta-\xi|=\bigO(n^{-1/4})$. By \eqref{lol51}, it follows that
\begin{align*}
\int_{w + \cup_{\ell:w\in\Omega_{\ell}}B_{\ell,n}}\rho_2(w, x)d^{2}x = \bigO(1), \qquad \mbox{as } n \to + \infty
\end{align*}
uniformly for $\mathrm{dist}(w,\partial \mathcal{S}) \leq 2\epsilon_{n}$, which finishes the proof.
\end{proof}
Let us define $\Omega_{1}^{k} = \Omega^k\setminus\Omega_{0}^{k}$. The next lemma shows that the integral of $\rho_{k}^{J}$ over $\Omega_{1}^{k}$ vanishes as $n\to +\infty$, thereby establishing the first equality in \eqref{strategy of proof}.
\begin{lemma}\label{lemma5}
(The first equality in \eqref{strategy of proof}) For any $m\in\N_{>0}$, bounded Borel sets
$B_1,\ldots,B_m\subset\HH_+$ and
$\Omega_1,\ldots,\Omega_m\subset\C$, set $J=\bigcup_{\ell=1}^{m}(B_\ell\times\Omega_\ell)$ and $\Omega=\bigcup_{\ell=1}^{m}\Omega_\ell$.
\begin{align*}
\lim_{n\to + \infty}\int_{\Omega_{1}^{k}}\rho_k^{J}(w_1, \ldots, w_k)d^{2}w_1\ldots d^{2}w_k = 0.
\end{align*}
\end{lemma}
\begin{proof}
The proof is a minor adaptation of \cite[Lemma 3.6]{SJ2012}. 

Let $(w_{1},\ldots,w_{k})\in \Omega_{1}^{k}$. Since the subset of $\Omega_{1}^{k}$ where at least two $w_{j}$'s are equal has measure zero, it is enough to consider the case where $w_{i}\neq w_{j}$ for $i\neq j$. 

We introduce an equivalence relation $\sim$ on the set $\{w_1, \ldots, w_k\}$ as follows. We first write $w_i\sim w_j$ if $|w_i-w_j| \leq 2\tau_{n}$. We then extend this relation transitively: if there exist $w_{i_1}, \ldots, w_{i_t}$ such that $w_i\sim w_{i_1},w_{i_{1}}\sim w_{i_2},\ldots, w_{i_{t-1}}\sim w_{i_{t}}$ and $w_{i_{t}} \sim w_{j}$, then we also write $w_i\sim w_j$. 

The definition of $\Omega_{1}^{k}$ implies that there are $p<k$ equivalence classes. For each $j\in \{1,\ldots,p\}$, let $w_{i_{j}}$ denote the maximal element in the $j$-th class with respect to $\succ$. 

Since $B_{\ell,n}\subset\{u\in\C:|u|<\tau_n\}$ for all $\ell$,
the sets $w_{i_j}+\cup_{\ell:w_{i_{j}}\in\Omega_{\ell}}B_{\ell,n}$, $j=1,\ldots,p$, are pairwise disjoint; moreover, each of these sets contains none of $w_1,\ldots,w_k$.

Since $w_{i}\neq w_{j}$ for $i\neq j$ and the density in \eqref{general density intro} is continuous, we have
\begin{align*}
& \rho_{k}^{J}(w_{1},\ldots,w_{k}) = \E\bigg[ \sum_{\ell_{1},\ldots,\ell_{k}} \prod_{j=1}^{k} \delta_{w_{j}}(z_{\ell_{j}})\mathbbm{1}_{\{\xi^{(n)}(w_{j}+\cup_{\ell:w_{j}\in\Omega_{\ell}}B_{\ell,n})=1\}} \bigg] \\
& \leq \E\bigg[ \sum_{\ell_{1},\ldots,\ell_{k}} \prod_{j=1}^{k} \delta_{w_{j}}(z_{\ell_{j}}) \prod_{j=1}^{p}\mathbbm{1}_{\{\xi^{(n)}(w_{i_{j}}+\cup_{\ell:w_{i_{j}}\in\Omega_{\ell}}B_{\ell,n})=1\}} \bigg] \\
& \leq \E\bigg[ \sum_{\ell_{1},\ldots,\ell_{k}} \prod_{j=1}^{k} \delta_{w_{j}}(z_{\ell_{j}}) \prod_{j=1}^{p}\xi^{(n)}(w_{i_{j}}+\cup_{\ell:w_{i_{j}}\in\Omega_{\ell}}B_{\ell,n}) \bigg] \\
& = \E\bigg[ \sum_{\ell_{1},\ldots,\ell_{k}} \prod_{j=1}^{k} \delta_{w_{j}}(z_{\ell_{j}}) \sum_{\ell_{k+1},\ldots,\ell_{k+p}} \prod_{j=1}^{p}\mathbbm{1}_{\{z_{\ell_{k+j}}\in w_{i_{j}}+\cup_{\ell:w_{i_{j}}\in\Omega_{\ell}} B_{\ell,n} \}} \bigg] \\
& = \int_{w_{i_1}+\cup_{\ell:w_{i_{1}}\in\Omega_{\ell}}B_{\ell,n}} d^{2}x_1\ldots\int_{w_{i_p}+\cup_{\ell:w_{i_{p}}\in\Omega_{\ell}}B_{\ell,n}}d^{2}x_p\rho_{k+p}(w_1, \ldots, w_k, x_1, \ldots, x_p).
\end{align*}
Using Lemma \ref{lemma1}, we then get
\begin{align*}
\rho_{k}^{J}(w_{1},\ldots,w_{k}) \leq \prod_{\substack{j=1,\ldots,k \\ j\notin \{ i_1, \ldots, i_p\}}}\rho_1(w_j) \; \prod_{j=1}^p\int_{w_{i_j}+\cup_{\ell:w_{i_{j}}\in\Omega_{\ell}}B_{\ell,n}}\rho_2(w_{i_j}, x_j)d^{2}x_j.
\end{align*}
Proposition \ref{prop:pointwise estimate} implies that $\rho_{1}(w)=\bigO(n)$ as $n\to + \infty$ uniformly for $w\in \C$, and from the proof of Lemma \ref{lemma4}, we also have that $\int_{w+\cup_{\ell:w\in\Omega_{\ell}}B_{\ell,n}}\rho_2(w, x)d^{2}x=\bigO(1)$ as $n\to + \infty$ uniformly for $w\in \C$. Hence
\begin{align}\label{lol26}
\rho_{k}^{J}(w_{1},\ldots,w_{k}) = \bigO(n^{k-p}), \qquad \mbox{as } n \to + \infty
\end{align}
uniformly for $(w_{1},\ldots,w_{k})\in \Omega_{1,p}^{k}$, where 
\begin{align*}
\Omega_{1,p}^{k} = \{(w_{1},\ldots,w_{k})\in \Omega_{1}^{k}: w_{i}\neq w_{j} \mbox{ if } i \neq j \mbox{ and there are $p$ equivalence classes} \}.
\end{align*}
Since $\Omega$ is bounded, for each $p\in \{1,\ldots,k-1\}$ the region $\Omega_{1,p}^{k}$ has Lebesgue measure $\mathcal{O}(n^{-\frac{3}{2}(k-p)})$. Combined with the estimate \eqref{lol26}, this gives
\begin{align*}
\int_{\Omega_{1,p}^{k}}\rho_k^{J}(w_1, \ldots, w_k)d^{2}w_1\ldots d^{2}w_k = \bigO(n^{-\frac{k-p}{2}}), \qquad \mbox{as } n \to + \infty.
\end{align*}
Since $\Omega_{1}^{k}\setminus \cup_{p=1}^{k-1}\Omega_{1,p}^{k}$ has measure zero, the proof is complete.
\end{proof}

%
%

\begin{proof}[Proof of Theorem \ref{thm1}]
Lemmas \ref{lemma3}, \ref{lemma4} and \ref{lemma5} together imply that \eqref{v3.5} (and hence also \eqref{eqn3}) holds for any $m\in\N_{>0}$ and bounded Borel sets $B_1,\ldots,B_m\subset\HH_+$ and $\Omega_1,\ldots,\Omega_m\subset\C$ such that the rectangles $B_\ell\times\Omega_\ell$
are pairwise disjoint. That is, all factorial moments of $\xi^{(n)}_{J}(\Omega)$ converge as $n\to + \infty$ to those of a $\mathrm{Poisson}(\lambda_{J})$ random variable, where $J=\bigcup_{\ell=1}^{m}(B_\ell\times\Omega_\ell)$ and $\Omega=\bigcup_{\ell=1}^{m}\Omega_\ell$. It then follows from \cite[Proposition 2.8]{LM2024} that $\xi^{(n)}_{J}(\Omega) \oset{\mathrm{law}}{\underset{n\to \infty}{\xrightarrow{\hspace*{0.85cm}}}} \mathrm{Poisson}(\lambda_{J})$. By Lemma \ref{lemma2}, this implies
\begin{align}\label{lol27}
\chi^{(n)}(J) \oset{\mathrm{law}}{\underset{n\to \infty}{\xrightarrow{\hspace*{0.85cm}}}} \mathrm{Poisson}(\lambda_{J}).
\end{align}
Since \eqref{lol27} holds for any finite union $J$
of bounded rectangles, by Proposition \ref{prop1}, the sequence of point processes $\chi^{(n)}$ converges in distribution to a Poisson point process $\chi$ with intensity \eqref{intensity main thm}, as desired.
\end{proof}

\begin{proof}[Proof of Corollary \ref{col1}.]
Let $\widetilde{t}^{(n)}_{\ell}$ be the $\ell$-th smallest element of the set $\{n^{3/4}|z_i-z_{i^*}|: 1\leq i \leq n-1\}$, and let $t^{(n)}_\ell$ be the $\ell$-th smallest element of $\{n^{3/4}|z_i-z_j|: 1\leq i< j \leq n\}$. For each $\ell=1,\ldots,k$, let $p_\ell<q_\ell$ be the (almost surely unique) indices such that
$t_\ell^{(n)}=n^{3/4}|z_{q_\ell}-z_{p_\ell}|$, let $i_\ell$ be the (almost surely unique) index such that
$\widetilde{t}_\ell^{(n)}
=n^{3/4}|z_{i_\ell^*}-z_{i_\ell}|$,
and define $\theta_\ell^{(n)},\widetilde{\theta}_\ell^{(n)}\in[0,\pi)$ by
\begin{align*}
t_\ell^{(n)}e^{i\theta_\ell^{(n)}}
=n^{3/4}(z_{q_\ell}-z_{p_\ell}), \qquad \widetilde{t}_\ell^{(n)}e^{i\widetilde{\theta}_\ell^{(n)}}
=n^{3/4}(z_{i_\ell^*}-z_{i_\ell}).
\end{align*}
We first compute the limiting joint distribution of
$\{\widetilde{t}_\ell^{(n)},
\widetilde{\theta}_\ell^{(n)}\}_{\ell=1}^{k}$.
We will then show that
$\{t_\ell^{(n)},\theta_\ell^{(n)}\}_{\ell=1}^{k}$
has the same limiting distribution.

For $0\leq a<b$ and $0\leq\alpha<\beta\leq\pi$, let
\begin{align*}
B_{a,b}:=\{u\in\HH_{+}:a<|u|<b\}, \qquad B_{a,b}^{\alpha,\beta}
:=\{re^{i\theta}:a<r<b,\ \alpha<\theta<\beta\}.
\end{align*} 
The event
$\{x_\ell<\widetilde{t}_\ell^{(n)}<y_\ell,\;
\alpha_\ell<\widetilde{\theta}_\ell^{(n)}<\beta_\ell
\mbox{ for all }1\leq\ell\leq k\}$
is (almost surely) equivalent to
\begin{align*}
&\chi^{(n)}(B_{x_k,y_k},\C)\geq1,
\qquad
\alpha_k<\widetilde{\theta}_k^{(n)}<\beta_k,\\
&\chi^{(n)}(B_{x_\ell,y_\ell}^{\alpha_\ell,\beta_\ell},\C) =\chi^{(n)}(B_{x_\ell,y_\ell},\C)=1,
& & \mbox{for all }1\leq\ell\leq k-1,\\
&\chi^{(n)}(B_{y_{\ell-1},x_\ell},\C)=0,
& & \mbox{for all }1\leq\ell\leq k,
\end{align*}
where we set $y_0=0$.
By \eqref{intensity main thm} and monotone convergence,
\begin{align}
\E\big[\chi(B_{a,b}^{\alpha,\beta},\C)\big]
&= \lim_{m\to + \infty} \E\big[\chi(B_{a,b}^{\alpha,\beta},\{z\in \C: |z|\leq m\})\big] \nonumber \\
& =\frac{\mathcal{J}}{\pi}
\int_\alpha^\beta\int_a^b r^3dr\,d\theta=\frac{\mathcal{J}}{4}(b^4-a^4)
\frac{\beta-\alpha}{\pi}, \label{lol55}
\end{align}
where we recall that $\mathcal{J}=\pi^{2}\int_{\mathcal{S}} \rho(z)^{3}d^{2}z$.

Conditionally on the event $\chi(B_{a,b},\C)\geq1$, let
$(te^{i\theta},v)$, $t\geq 0$, $\theta\in [0,\pi)$, be the point of $\chi$ with smallest $t$ among those satisfying
$a<t<b$. We now show that $\theta$ is uniformly distributed
on $[0,\pi)$, and is independent of the moduli $|u|$ for all $(u,w)\in \chi$ satisfying $a<|u|<b$. Condition first on $\chi(B_{a,b},\C)=m\geq1$, and denote by $\{R_j,\Theta_j\}_{j=1}^{m}$ the polar coordinates of the points of $\chi$ lying in $B_{a,b}$, ordered so that
$R_1<\ldots<R_m$. In particular, $R_1=t$ and $\Theta_1=\theta$. For $a<r_1<\ldots<r_m<b$ and
$\varphi_1,\ldots,\varphi_m\in[0,\pi)$, we have
\begin{align*}
&\PP\big(
R_j\in(r_j,r_j+d\epsilon),\;
\Theta_j\in(\varphi_j,\varphi_j+d\varphi_j)
\mbox{ for all }1\leq j\leq m
\mid \chi(B_{a,b},\C)=m
\big)\\
&\qquad=
\frac{
e^{-\frac{\mathcal{J}}{4}(b^4-a^4)}
\prod_{j=1}^{m}
\frac{\mathcal{J}}{\pi}r_j^3\,d\epsilon\,d\varphi_j
}{
e^{-\frac{\mathcal{J}}{4}(b^4-a^4)}(\frac{\mathcal{J}}{4}(b^4-a^4))^m/m!
} =
m!\prod_{j=1}^{m}
\frac{4r_j^3\,d\epsilon}{b^4-a^4}
\frac{d\varphi_j}{\pi}.
\end{align*}
Thus
\begin{align}\label{lol56}
\PP(\alpha<\theta<\beta
\mid \chi(B_{a,b},\C)=m,R_1,\ldots,R_m)
=\frac{\beta-\alpha}{\pi}.
\end{align}
Since the right-hand side does not depend on $m,R_1,\ldots,R_m$, the angle
$\theta$ is uniformly distributed on $[0,\pi)$ and is
independent of the $R_{j}$ (conditionally on $\chi(B_{a,b},\C)\geq1$).

Fix a compact set $K$ such that $\mathcal{S}\subset K^\circ$.
By Proposition \ref{prop:pointwise estimate},
\eqref{Q minus Qhat lower bounded} and \eqref{growth of $Q$},
\begin{align}\label{lol54}
\PP\big(\xi^{(n)}(\C\setminus K)\geq1\big)
&\leq \E\big[ \xi^{(n)}(\C\setminus K) \big] = \int_{\C\setminus K}\rho_1(v)d^2v \leq C_1n\int_{\C\setminus K}
e^{-n(Q(v)-\widehat{Q}(v))}d^2v \to 0
\end{align}
as $n\to + \infty$, which implies
\begin{align*}
& \lim_{n\to+\infty}
\mathbb{P}\big(
x_\ell<\widetilde{t}_\ell^{(n)}<y_\ell,\;
\alpha_\ell<\widetilde{\theta}_\ell^{(n)}<\beta_\ell
\mbox{ for all }1\leq\ell\leq k
\big)\\
&=\lim_{n\to+\infty}
\PP\big(x_\ell<\widetilde{t}_\ell^{(n)}<y_\ell,\;
\alpha_\ell<\widetilde{\theta}_\ell^{(n)}<\beta_\ell
\mbox{ for all }1\leq\ell\leq k \; , \;
\xi^{(n)}(\C\setminus K)=0
\big).
\end{align*}
 Since $\chi$ is a Poisson point process, $\chi(\mathcal{N}_{1})$ and $\chi(\mathcal{N}_{2})$ are independent whenever $\mathcal{N}_{1}$ and $\mathcal{N}_{2}$ are disjoint, and therefore
\begin{align*}
& \lim_{n\to+\infty}
\mathbb{P}\big(
x_\ell<\widetilde{t}_\ell^{(n)}<y_\ell,\;
\alpha_\ell<\widetilde{\theta}_\ell^{(n)}<\beta_\ell
\mbox{ for all }1\leq\ell\leq k
\big)\\
& =\lim_{n\to+\infty}
\mathbb{P}\big(
x_\ell<\widetilde{t}_\ell^{(n)}<y_\ell,\;
\alpha_\ell<\widetilde{\theta}_\ell^{(n)}<\beta_\ell
\mbox{ for all }1\leq\ell\leq k,\;
\xi^{(n)}(\C\setminus K)=0
\big) \\
& = \mathbb{P}\big(\chi(B_{x_k,y_k},\C)\geq1\big)
\frac{\beta_k-\alpha_k}{\pi}
\prod_{\ell=1}^{k-1}
\mathbb{P}\big(
\chi(B_{x_\ell,y_\ell}^{\alpha_\ell,\beta_\ell},\mathbb{C})
=\chi(B_{x_\ell,y_\ell},\mathbb{C})=1
\big)
\prod_{\ell=1}^{k}
\mathbb{P}\big(\chi(B_{y_{\ell-1},x_\ell},\mathbb{C})=0\big)\\
& = \Big(1-e^{-\frac{\mathcal{J}}{4}(y_k^4-x_k^4)}\Big)
\prod_{\ell=1}^{k-1}
\frac{\mathcal{J}}{4}(y_\ell^4-x_\ell^4)
e^{-\frac{\mathcal{J}}{4}(y_\ell^4-x_\ell^4)}
\prod_{\ell=1}^{k}
e^{-\frac{\mathcal{J}}{4}(x_\ell^4-y_{\ell-1}^4)}
\prod_{\ell=1}^{k}\frac{\beta_\ell-\alpha_\ell}{\pi}\\
& = \Big(e^{-\frac{\mathcal{J}}{4}x_k^4}
-e^{-\frac{\mathcal{J}}{4}y_k^4}\Big)
\prod_{\ell=1}^{k-1}
\frac{\mathcal{J}}{4}(y_\ell^4-x_\ell^4)
\prod_{\ell=1}^{k}\frac{\beta_\ell-\alpha_\ell}{\pi}.
\end{align*}
where we used Theorem \ref{thm1} and \eqref{lol56} for the second equality.
Therefore, $(\widetilde{t}_1^{(n)}, \ldots, \widetilde{t}_k^{(n)},\widetilde{\theta}_1^{(n)}, \ldots, \widetilde{\theta}_k^{(n)})$ converges in distribution as $n\to + \infty$ to some random variables $(\widetilde{t}_1, \ldots, \widetilde{t}_k,\widetilde{\theta}_1, \ldots, \widetilde{\theta}_k)$ whose joint probability measure is given by taking $y_{\ell}\to x_{\ell}$ and $\beta_{\ell}\to \alpha_{\ell}$ for all $1\leq \ell \leq k$ in the above expression; more precisely, the joint probability density of $(\widetilde{t}_1, \ldots, \widetilde{t}_k,\widetilde{\theta}_1, \ldots, \widetilde{\theta}_k)$ is given for $0 < x_{1} < \ldots < x_{k}$ and $\alpha_{1},\ldots,\alpha_{k}\in [0,\pi)$ by 
\begin{align*}
\mathbb{P}\big( \widetilde{t}_{\ell} \in (x_{\ell},x_{\ell}+dx_{\ell}) \mbox{ and } \widetilde{\theta}_{\ell}\in (\alpha_{\ell},\alpha_{\ell}+d\alpha_{\ell}) \mbox{ for all } 1 \leq \ell \leq k \big) = \mathcal{J}^{k}e^{-\frac{\mathcal{J}}{4}x_{k}^{4}}\prod_{\ell=1}^{k}x_{\ell}^{3}dx_{\ell} \frac{d\alpha_\ell}{\pi},
\end{align*}
and integrating the above on the region 
\begin{align*}
\{(x_{1},\ldots,x_{k-1},\alpha_{1},\ldots,\alpha_{k}):0<x_{1}<\ldots<x_{k-1}<x_{k}, \; 0 \leq \alpha_{j} < \pi, \; j=1,\ldots,k\}
\end{align*}
yields
\begin{align*}
\mathbb{P}\big( \widetilde{t}_{k} \in (x_{k},x_{k}+dx_{k}) \big) = C_{k}^{-1}x_{k}^{4k-1}e^{-\frac{\mathcal{J}}{4}x_{k}^{4}}dx_{k},
\end{align*}
where $C_{k}=\int_{0}^{+\infty}x_{k}^{4k-1}e^{-\frac{\mathcal{J}}{4}x_{k}^{4}}dx_{k}=(k-1)!4^{k-1}/\mathcal{J}^{k}$.

\bigskip

It remains to prove that $\tilde{G}^{(n)}:=(\widetilde{t}^{(n)}_{\ell},\widetilde{\theta}^{(n)}_{\ell})_{\ell=1}^{k}$ and $G^{(n)}:=(t^{(n)}_{\ell},\theta^{(n)}_{\ell})_{\ell=1}^{k}$ have the same limiting distribution. For any Borel set $\mathcal{R}\subset \{(x_{\ell},\alpha_{\ell})_{\ell=1}^{k}\in(\R^{+}\times[0,\pi))^{k}:x_{1}<\ldots<x_{k}\}$, we have
\begin{multline*}
|\PP(G^{(n)}\in \mathcal{R})-\PP(\tilde{G}^{(n)}\in \mathcal{R})| = |\E(\mathbbm{1}_{\{G^{(n)}\in \mathcal{R}\}}-\mathbbm{1}_{\{\tilde{G}^{(n)}\in \mathcal{R}\}})| \leq \E(|\mathbbm{1}_{\{G^{(n)}\in \mathcal{R}\}}-\mathbbm{1}_{\{\tilde{G}^{(n)}\in \mathcal{R}\}}|) \\
\leq \E(\mathbbm{1}_{\{G^{(n)} \neq \tilde{G}^{(n)}\}}) = \PP(G^{(n)} \neq \tilde{G}^{(n)}),
\end{multline*}
and therefore it is enough to prove that
\begin{align}\label{lol23}
\lim_{n\to + \infty} \PP(G^{(n)} \neq \tilde{G}^{(n)}) = 0.
\end{align}
Let $p_{\ell}<q_{\ell}$ be the indices such that $n^{3/4}|z_{p_\ell} - z_{q_{\ell}}| = t^{(n)}_\ell$. We will prove that the $2k$ points $\{z_{p_\ell}, z_{q_\ell}: \ell = 1, \ldots, k\}$ are distinct with probability tending to $1$ as $n\to + \infty$, which in turn implies \eqref{lol23}.

If these $2k$ points are not distinct, then there must be three points $z_{i_1}, z_{i_2}, z_{i_3}$ such that $|z_{i_2} - z_{i_1}|, |z_{i_3} - z_{i_1}| \leq \widetilde{t}^{(n)}_kn^{-3/4}$, which we expect to be a rare event. We are thus led to consider the point process
\begin{align*}
\Xi^{(3)} = \sum_{z_{i_1}, z_{i_2}, z_{i_3}\text{ distinct}}\delta_{(z_{i_1}, z_{i_2}, z_{i_3})}
\end{align*}
as well as the set
\begin{align*}
\mathcal{B} = \{(z, x_1, x_2): z \in \C, |x_1-z|\leq Mn^{{-3/4}}, |x_2-z|\leq Mn^{{-3/4}}\}
\end{align*}
where $M>0$ is a constant. We then have the inequality
\begin{align}\label{lol24}
\PP(G^{(n)} \neq \tilde{G}^{(n)}) \leq \PP(\Xi^{(3)}(\mathcal{B})\neq 0) + \PP(\widetilde{t}^{(n)}_k> M).
\end{align}
It follows from the proof of Lemma \ref{lemma2} that
\begin{align*}
\PP(\Xi^{(3)}(\mathcal{B})\neq 0) \leq \E(\Xi^{(3)}(\mathcal{B})) \to 0, \qquad \mbox{as } n \to + \infty.
\end{align*}
Hence, taking the limit $n\to + \infty$ in \eqref{lol24}, we find
\begin{align*}
\lim_{n\to + \infty}\PP(G^{(n)} \neq \tilde{G}^{(n)}) \leq \PP(\widetilde{t}_k> M).
\end{align*}
Since the above is valid for any $M>0$, taking $M\to + \infty$ yields \eqref{lol23}, as desired.

\end{proof}

\appendix

\section{Proof of Theorem \ref{thm:Charlier}}\label{appendix:Proof of edge}
In this appendix, $K>M>0$ are fixed constants, and $\epsilon_{n}:=K \sqrt{\log n}/\sqrt{n}$. We also let $z_{0}\in \partial \mathcal{S}$ and define 
\begin{align*}
z=z_{0}+ \mathrm{n}\frac{c_{0}\xi}{\sqrt{n}}, \qquad w = z_{0}+ \mathrm{n}\frac{c_{0}\eta}{\sqrt{n}}, \qquad c_{0}=\frac{\sqrt{2}}{\sqrt{\Delta Q(z_{0})/4}},
\end{align*}
where $\xi,\eta\in \C$ satisfy $|\xi|,|\eta| \leq M \sqrt{\log n}$ and $\mathrm{n}\in \{z\in \C:|z|=1\}$ denotes the outward pointing unit normal to $\partial \mathcal{S}$ at $z_{0}$. Recall from \eqref{def of Kn} that
\begin{align}\label{def of Kn 2}
K_n(z,w) =  e^{-\frac{1}{2}nQ(z)-\frac{1}{2}nQ(w)}\sum_{j = 0}^{n-1}p_{j,n}(z)\overline{p_{j,n}(w)},
\end{align}
where $p_{j,n}$ is defined via \eqref{ortho condition}. Let $\tau_{1} = 1-\epsilon_{n}$. By \cite[Proposition 2.3 with $\mathcal{K}=\emptyset$, $u=p_{j,n}$, $\tau = \tau_{1}$]{HW2021} (see also \cite[Lemma 3.7]{AmeurCronvall}), we have
\begin{align*}
e^{-\frac{1}{2}nQ(z)}|p_{j,n}(z)| \leq D_{1} n^{1/2} e^{-\frac{n}{2}(Q-\widehat{Q}_{\tau_{1}})(z)}, \qquad \mbox{for all } j \leq n-\lfloor \epsilon_{n}n \rfloor - 1, \; |\xi| \leq M\sqrt{\log n},
\end{align*}
for some constant $D_{1}>0$ independent of $z_{0}$, and where $\widehat{Q}_{\tau}$ is the solution to the obstacle problem
\begin{align*}
\widehat{Q}_{\tau}(z) := \sup\{q(z)| \; & q:\C\to \R\cup\{-\infty\} \mbox{ is subharmonic on }\C, \\
& q(z) \leq 2\tau \log |z|+\bigO(1) \mbox{ as } |z|\to\infty, \\
& q \leq Q \mbox{ on } \C\}.
\end{align*}
Let $\delta>0$ be a small but fixed constant. For $\tau \in [1-\delta,1+\delta]$, let $\mathcal{S}_{\tau}$ denote the support of the equilibrium measure associated with $\tau^{-1}Q$. By \cite[Lemma 2.5]{HW2021}, we have $\mathcal{S}_{\tau} \subset \mathcal{S}$ for all $\tau\in [1-\delta,1)$, and $\mathcal{S}_{1}=\mathcal{S}$. Since $Q$ is as in Theorem \ref{thm1}, we can choose $\delta>0$ sufficiently small so that the curves $\partial \mathcal{S}_{\tau}$ are real-analytically smooth for $\tau\in [1-\delta,1+\delta]$ (see
\cite[page 318]{HW2021}, \cite{S1991} or \cite[Theorem 6.2]{HS2002}). By \cite[Lemma 2.5]{HW2021}, there exists a constant $c>0$ independent of $z_0$ such that
\begin{align*}
\mathrm{dist}(z_0,\partial\mathcal{S}_{\tau_1})
\geq c(1-\tau_1)=c\epsilon_n,
\qquad z_0\in\partial\mathcal{S},
\end{align*}
for all sufficiently large $n$. Choose $K$ large enough so that $cK>c_0M$ for all $z_0 \in \partial\mathcal{S}$. Then $z\in\mathcal{S}_{\tau_1}^{c}$, and
\begin{align*}
\mathrm{dist}(z,\partial\mathcal{S}_{\tau_1})
&\geq
\mathrm{dist}(z_0,\partial\mathcal{S}_{\tau_1})
-|z-z_0|\geq c\epsilon_n-|z-z_0|
\geq (cK-c_0M)\sqrt{\frac{\log n}{n}}>0.
\end{align*}
Hence, using \cite[(5.5)]{HW2021}, we find
\begin{align*}
(Q-\widehat{Q}_{\tau_{1}})(z) \geq D_{2} \, \mathrm{dist}(z,\partial \mathcal{S}_{\tau_{1}})^{2} \geq D_{3}(c \, \epsilon_{n}-|z-z_{0}|)^{2} \geq D_{3}(cK-c_{0}M)^{2}\frac{\log n}{n},
\end{align*}
for all sufficiently large $n$, for all $|\xi| \leq M \sqrt{\log n}$, and for some constants $D_{2},D_{3}>0$ independent of $M$, $K$, and $z_{0}$. Hence, increasing $K$ if necessary, we get
\begin{align*}
e^{-\frac{1}{2}nQ(z)}|p_{j,n}(z)| \leq n^{-51}, \qquad \mbox{for all $n$ large enough, } j \leq n-\lfloor \epsilon_{n}n \rfloor - 1 \mbox{, and } z_{0}\in \partial \mathcal{S}.
\end{align*}
Substituting the above estimate in \eqref{def of Kn 2} directly yields
\begin{align}\label{lol31}
K_n(z,w) =  \bigO(n^{-100})+e^{-\frac{1}{2}nQ(z)-\frac{1}{2}nQ(w)}\sum_{j = n-\lfloor\epsilon_{n}n\rfloor}^{n-1}p_{j,n}(z)\overline{p_{j,n}(w)}, \qquad \mbox{as } n \to + \infty
\end{align}
uniformly for $|\xi|,|\eta|\leq M\sqrt{\log n}$ and $z_{0}\in \partial \mathcal{S}$. The remaining sum can be estimated using \cite[Theorem 1.3]{HW2021}. Let $\phi_{\tau}$ be the conformal map from $\mathcal{S}_{\tau}^{c}$ to $\{|z|>1\}$ normalized by $\phi_{\tau}(\infty)=\infty$ and $\phi_{\tau}'(\infty)>0$, and let $\mathcal{Q}_{\tau}$ denote  the bounded holomorphic function on $\mathcal{S}_{\tau}^{c}$ whose real part equals $Q$ on $\partial \mathcal{S}_{\tau}$, and whose imaginary part vanishes at $\infty$. Since $Q$ is real analytic in a neighborhood of $\mathcal{S}$ and $\partial\mathcal{S}_{\tau}$ is real-analytically smooth for $\tau \in [1-\delta,1]$, $\phi_{\tau}$ and $\mathcal{Q}_{\tau}$ extend holomorphically on an open neighborhood of $\overline{\mathcal{S}^{c}}$ for all $\tau \in [1-\delta,1]$, with $\phi_{\tau}$ staying conformal there (shrinking $\delta>0$ if necessary), see \cite[Proposition 3.1]{Pommerenke1992}. By \cite[Theorem 1.3]{HW2021}, there exists a bounded holomorphic function $\mathcal{B}_{\tau}$ defined in a fixed neighborhood of $\overline{\mathcal{S}^{c}}$ such that 
\begin{align}
e^{-\frac{1}{2}nQ(z)}p_{n-j,n}(z) & = n^{1/4}\phi_{\tau_{n-j}}'(z)^{1/2}\phi_{\tau_{n-j}}(z)^{\tau_{n-j} n} e^{\frac{n}{2}(\mathcal{Q}_{\tau_{n-j}}-Q)(z)} \Big( \mathcal{B}_{\tau_{n-j}}(z)+\bigO(n^{-1}) \Big) \label{lol28}
\end{align}
as $n\to + \infty$ uniformly for $|\xi|\leq M\sqrt{\log n}$, $j \leq \lfloor\epsilon_{n}n\rfloor$ and $z_{0}\in \partial \mathcal{S}$, and where $\tau_{j}:=j/n$. In \eqref{lol28}, the branch is chosen so that $\phi_{\tau_{n-j}}'(\infty)^{1/2}>0$. Moreover, \cite[Theorem 1.5]{HW2021} gives the following explicit expression for $\mathcal{B}_{\tau}$:
\begin{align}\label{lol35}
\mathcal{B}_{\tau} = 2^{-\frac{1}{4}}\pi^{-\frac{3}{4}}e^{H_{\tau}},
\end{align}
where $H_{\tau}$ is bounded and holomorphic in $\mathcal{S}_{\tau}^{c}$ and satisfies $\im H_{\tau}(\infty)=0$, as well as $\re H_{\tau} = \frac{1}{4}\log \frac{\Delta Q}{4}$ on $\partial \mathcal{S}_{\tau}$. (In \eqref{lol35}, the extra factor $2^{-\frac{1}{4}}$ comes from the fact that $2Q$ in \cite{HW2021} corresponds to $Q$ here, and the factor $\pi^{-\frac{3}{4}}$ stems from the fact that the orthonormal polynomials in \cite{HW2021} are normalized with respect to $dA(z)=\frac{d^{2}z}{\pi}$.) The following useful lemma will be proved at the end of this section.

\begin{lemma}\label{lemma:joint regularity}
Suppose $Q$ is as in Theorem \ref{thm1}.
There exist $\delta>0$ and an open neighborhood $U$
of $\partial\mathcal{S}$ such that, for
$\tau\in(1-\delta,1+\delta)$, the functions
$\phi_\tau$, $\mathcal{Q}_\tau$, $H_\tau$ and
$\mathcal{B}_\tau$ admit holomorphic extensions to $U$.
These extensions are jointly real-analytic in
$(\re z,\im z,\tau)$ on $U\times(1-\delta,1+\delta)$.
\end{lemma}

By Lemma \ref{lemma:joint regularity}, the functions
$\phi_\tau(z)$, $\mathcal{Q}_\tau(z)$ and
$\mathcal{B}_\tau(z)$ are jointly smooth in
$(\re z,\im z,\tau)$ in a fixed neighborhood of
$\partial\mathcal{S}\times\{1\}$. Hence, by Taylor's formula, 
\begin{align}
& \phi_{\tau_{n-j}}'(z)^{1/2} \mathcal{B}_{\tau_{n-j}}(z)  = b_{0} + \frac{b_{1}\re \xi + b_{2}\im \xi}{\sqrt{n}}  +b_{3} \frac{j}{n} + \bigO \bigg( \frac{1+|\xi|^{2}}{n} + \frac{(1+|\xi|)j}{n^{3/2}} + \frac{j^{2}}{n^{2}} \bigg), \label{lol30} 
\end{align}
as $n\to + \infty$ uniformly for $|\xi|\leq M\sqrt{\log n}$, $j \leq \lfloor\epsilon_{n}n\rfloor$ and $z_{0}\in \partial \mathcal{S}$, where
\begin{align*}
& b_{0}=\phi_{1}'(z_{0})^{1/2}\mathcal{B}_{1}(z_{0}), & & b_{1}= c_{0}\partial_{\mathrm{n}}(\phi_{1}'^{1/2}\mathcal{B}_{1})(z_{0}), \\
& b_{2}= c_{0}\partial_{\mathrm{t}}(\phi_{1}'^{1/2}\mathcal{B}_{1})(z_{0}), & & b_{3} = - \partial_{\tau}(\phi_{\tau}'^{1/2}\mathcal{B}_{\tau})|_{\tau=1}(z_{0}),
\end{align*}
and where $\partial_{\mathrm{n}}$ and $\partial_{\mathrm{t}}$ denote the normal and tangential derivatives, respectively. By \eqref{lol35}, $|b_{0}|=|\phi_{1}'(z_{0})^{1/2}| \big(\frac{\Delta Q(z_{0})}{4\pi}\big)^{1/4}\pi^{-\frac{1}{2}}2^{-\frac{1}{4}}$. Using the notation $\breve{Q}_{\tau} = 2\tau \log |\phi_{\tau}| + \re \mathcal{Q}_{\tau}$, we write
\begin{align*}
\phi_{\tau}(z)^{\tau n} e^{\frac{n}{2}(\mathcal{Q}_{\tau}-Q)(z)} = e^{-\frac{n}{2}(Q-\breve{Q}_{\tau})(z)} e^{\frac{i n}{2}(\im \mathcal{Q}_{\tau}(z) + 2\tau \arg \phi_{\tau}(z))}.
\end{align*}
Since $e^{\frac{i n}{2}(\im \mathcal{Q}_{\tau}(s) + 2\tau \arg \phi_{\tau}(s))}$ is smooth for $\tau$ close to $1$ and $s$ close to $\partial \mathcal{S}$, in a similar way as \eqref{lol30} we infer that there exists $\{a_{j}\}_{j=1}^{20}\subset \R$ (which are smooth functions of $z_{0}\in \partial \mathcal{S}$) such that
\begin{align}
& e^{\frac{i n}{2}(\im \mathcal{Q}_{\tau}(z) + 2\tau \arg \phi_{\tau}(z))} = e^{a_{1} in  +  i (a_{2} \re \xi + a_{3} \im \xi) \sqrt{n} + i [a_{4}j + a_{5} (\re \xi)^{2} + a_{6}(\re \xi)(\im \xi) + a_{7} (\im \xi)^{2} ]} \nonumber \\
& \times e^{i(a_{8} \re \xi + a_{9} \im \xi) \frac{j}{\sqrt{n}}+ia_{10} \frac{j^{2}}{n}} e^{i \frac{a_{11} (\re \xi)^{3} + a_{12} (\re \xi)^{2}(\im \xi) + a_{13} (\re \xi)(\im \xi)^{2} + a_{14} (\im \xi)^{3}}{\sqrt{n}}} \nonumber \\
& \times e^{ i \big(a_{15}(\re \xi)^{2} + a_{16} (\re \xi)(\im \xi) + a_{17}(\im \xi)^{2}\big) \frac{j}{n} + i(a_{18} \re \xi + a_{19} \im \xi) \frac{j^{2}}{n^{3/2}} + ia_{20} \frac{j^{3}}{n^{2}}}  \nonumber   \\
& \times \bigg\{ 1 + \bigO \bigg( \frac{1+|\xi|^{4}}{n} + \frac{(1+|\xi|^{3})j}{n^{3/2}} + \frac{(1+|\xi|^{2})j^{2}}{n^{2}} + \frac{(1+|\xi|)j^{3}}{n^{5/2}} + \frac{j^{4}}{n^{3}} \bigg) \bigg\} \label{lol32}
\end{align}
as $n\to + \infty$ uniformly for $|\xi| \leq M \sqrt{\log n}$, $j \leq \lfloor\epsilon_{n}n\rfloor$ and $z_{0}\in \partial \mathcal{S}$. Note that the coefficients $a_{1},a_{4},a_{10},a_{20}$ will not play any role in the asymptotics of $K_{n}(z,w)$, as the corresponding terms will get canceled when considering the product
\begin{align*}
e^{\frac{i n}{2}(\im \mathcal{Q}_{\tau}(z) + 2\tau \arg \phi_{\tau}(z))}e^{-\frac{i n}{2}(\im \mathcal{Q}_{\tau}(w) + 2\tau \arg \phi_{\tau}(w))}.
\end{align*} 
The coefficients $a_{2},a_{3},a_{5},a_{6},a_{7},a_{11},a_{12},a_{13},a_{14}$ are also of no importance, as the corresponding terms are independent of $j$ and can therefore be suppressed by premultiplying $K_{n}(z,w)$ by an appropriate unimodular factor. So the only important coefficients are $a_{8},a_{9},a_{15},a_{16},a_{17},a_{18},a_{19}$.
Let $\mathrm{z}_{\tau}$ be the nearest point to $z_{0}$ such that $\mathrm{z}_{\tau} \in (z_{0}+\mathrm{n}\R)\cap \partial \mathcal{S}_{\tau}$, and let $\mathtt{n}_{\tau}$ be the outer unit normal to $\partial \mathcal{S}_{\tau}$ at $\mathrm{z}_{\tau}$ (thus $\mathrm{z}_{1}=z_{0}$ and $\mathtt{n}_{1}=\mathrm{n}$). The parametrization $\mathrm{z}_{\tau}=z_{0}+\lambda_{\tau}\mathrm{n}$, $\lambda_{\tau}\in \R$, together with the formulas
\begin{align}\label{lol38}
|\phi_{\tau}(\mathrm{z}_{\tau})|^{2}=1, \qquad \mathtt{n}_{\tau} = \frac{\frac{d}{dx}\phi_{\tau}^{-1}(x\phi_{\tau}(\mathrm{z}_{\tau}))|_{x=1}}{\big|\frac{d}{dx}\phi_{\tau}^{-1}(x\phi_{\tau}(\mathrm{z}_{\tau}))|_{x=1}\big|},
\end{align}
and the fact that $\phi_{\tau}$ is conformal in an open neighborhood of $\partial \mathcal{S}_{\tau}$, implies by Lemma \ref{lemma:joint regularity} and the inverse function theorem that $\lambda_{\tau}, \mathrm{z}_{\tau}, \mathtt{n}_{\tau}$ are jointly
real-analytic in $(z_0,\tau)$ for $z_{0}\in \partial \mathcal{S}$ and $\tau$ in a fixed neighborhood of $\tau=1$. (Here ``$f(z_{0},\tau)$ is jointly real-analytic in $(z_0,\tau)$" means that $f(z_{0}(t),\tau)$ is jointly real-analytic in $(t,\tau)$, where $t\mapsto z_0(t)$ is any real-analytic parametrization of $\partial\mathcal{S}$.) Define $z_{j}=\mathrm{z}_{\tau_{n-j}}$ and $\mathrm{n}_{j}=\mathtt{n}_{\tau_{n-j}}$.  From the above discussion, there exist coefficients $z_{0}^{(1)},z_{0}^{(2)},\mathrm{n}^{(1)} \in \C$ (which are smooth functions of $z_{0}\in \partial \mathcal{S}$) such that
\begin{align}\label{lol29}
z_{j} = z_{0} + z_{0}^{(1)} \frac{j}{n} + z_{0}^{(2)} \frac{j^{2}}{n^{2}} + \bigO\bigg( \frac{j^{3}}{n^{3}} \bigg), \qquad \mathrm{n}_{j} = \mathrm{n} + \mathrm{n}^{(1)} \frac{j}{n} + \bigO\bigg( \frac{j^{2}}{n^{2}} \bigg),
\end{align}
as $n\to + \infty$ uniformly for $j \leq \epsilon_{n}n$ and $z_{0}\in \partial \mathcal{S}$ (since $\partial \mathcal{S}$ is compact). 

By \cite[page 366]{HW2021}, $z_{0}^{(1)}=-\mathrm{n} \frac{2|\phi_{1}'(z_{0})|}{\Delta Q(z_{0})}$. 
Hence, using \eqref{lol29},
\begin{align*}
z = z_{0} + \mathrm{n}\frac{c_{0}\xi}{\sqrt{n}} = \;  &  z_{j} + \mathrm{n}_{j} \bigg( c_{0} \frac{\xi}{\sqrt{n}} - \frac{z_{0}^{(1)}}{\mathrm{n}} \frac{j}{n} + \bigg[ -c_{0} \frac{\mathrm{n}^{(1)}}{\mathrm{n}}\xi \frac{j}{n^{3/2}} +  \frac{\mathrm{n}^{(1)}z_{0}^{(1)}-\mathrm{n} z_{0}^{(2)}}{\mathrm{n}^{2}} \frac{j^{2}}{n^{2}} \bigg] \bigg) \\
& + \bigO\bigg( \frac{j^{2}(1+|\xi|)}{n^{5/2}} + \frac{j^{3}}{n^{3}} \bigg)
\end{align*}
as $n\to + \infty$ uniformly for $|\xi| \leq M \sqrt{\log n}$, $j \leq \epsilon_{n}n$ and $z_{0}\in \partial \mathcal{S}$. Let $\mathcal{L}_{\tau}:=Q-\breve{Q}_{\tau}$. Along the boundary $\partial \mathcal{S}_{\tau}$, by \cite[page 341]{HW2021}, we have $\mathcal{L}_{\tau}=0$ and $\nabla \mathcal{L}_{\tau}=0$. Fix $\tilde{z}_{0} \in \partial \mathcal{S}_{\tau}$. A Taylor expansion of $0=\mathcal{L}_{\tau}(\tilde{z})=\partial_{\mathrm{n}}\mathcal{L}_{\tau}(\tilde{z})$ as $\tilde{z}\to \tilde{z}_{0}, \tilde{z} \in \partial \mathcal{S}_{\tau}$ shows that $\partial_{\mathrm{n}}\partial_{\mathrm{t}}\mathcal{L}(\tilde{z}_{0}) = \partial_{\mathrm{t}}^{2} \mathcal{L}(\tilde{z}_{0}) = \partial_{\mathrm{t}}^{3} \mathcal{L}(\tilde{z}_{0}) = 0$. Since $\tilde{z}_{0}$ was arbitrary, we conclude that $\partial_{\mathrm{n}}\partial_{\mathrm{t}}\mathcal{L} = \partial_{\mathrm{t}}^{2} \mathcal{L} = \partial_{\mathrm{t}}^{3} \mathcal{L} = 0$ along $\partial \mathcal{S}_{\tau}$. In particular,
\begin{align*}
\partial_{\mathrm{n}}^{2}\mathcal{L}_{\tau} = (\partial_{\mathrm{n}}^{2}+\partial_{\mathrm{t}}^{2})\mathcal{L}_{\tau} = \Delta Q.
\end{align*}
Therefore, using also \eqref{lol29} and the fact that all derivatives $\{\partial_{\mathrm{n}}^{\alpha_{1}}\partial_{\mathrm{t}}^{\alpha_{2}}\mathcal{L}_{\tau}(z_{j})\}_{\alpha_{1},\alpha_{2}\in \N}$ are smooth for $\tau \in [1-\delta,1]$ and $z_{0}\in \partial \mathcal{S}$ (which follows from Lemma \ref{lemma:joint regularity} and the joint real-analyticity of $\mathrm{z}_{\tau},\mathtt{n}_{\tau}$ in $(z_{0},\tau)$), we obtain
\begin{align*}
\mathcal{L}_{\tau_{n-j}}(z_{j}+\mathrm{n}_{j}\lambda) & = \frac{\Delta Q(z_{j})}{2}(\re \lambda)^{2} + \frac{1}{6} \big\{ \partial_{\mathrm{n}}^{3}\mathcal{L}_{\tau_{n-j}}(z_{j}) (\re \lambda)^{3} + 3 \partial_{\mathrm{n}}^{2}\partial_{\mathrm{t}} \mathcal{L}_{\tau_{n-j}}(z_{j}) (\re \lambda)^{2}(\im \lambda) \\
& + 3 \partial_{\mathrm{n}}\partial_{\mathrm{t}}^{2} \mathcal{L}_{\tau_{n-j}}(z_{j}) (\re \lambda)(\im \lambda)^{2} \big\} + \bigO(|\lambda|^{4}) \\
& \hspace{-1.1cm} = \bigg(\frac{\Delta Q(z_{0})}{2} + \frac{z_{0}^{(1)}}{\mathrm{n}} \frac{\partial_{\mathrm{n}}\Delta Q(z_{0})}{2} \frac{j}{n}\bigg)(\re \lambda)^{2} + \frac{1}{6} \big\{ \partial_{\mathrm{n}}^{3}\mathcal{L}_{1}(z_{0}) (\re \lambda)^{3} + 3 \partial_{\mathrm{n}}^{2}\partial_{\mathrm{t}} \mathcal{L}_{1}(z_{0}) (\re \lambda)^{2}(\im \lambda) \\
& + 3 \partial_{\mathrm{n}}\partial_{\mathrm{t}}^{2} \mathcal{L}_{1}(z_{0}) (\re \lambda)(\im \lambda)^{2} \big\} + \bigO\bigg(|\lambda|^{2} \frac{j^{2}}{n^{2}} + |\lambda|^{4}\bigg)
\end{align*}
as $\lambda \to 0$ uniformly for $n$ large, $j \leq \lfloor\epsilon_{n}n\rfloor$ and $z_{0}\in \partial \mathcal{S}$. Using the above expansion with
\begin{align*}
\lambda = c_{0} \frac{\xi}{\sqrt{n}} - \frac{z_{0}^{(1)}}{\mathrm{n}} \frac{j}{n} + \bigg[ -c_{0} \frac{\mathrm{n}^{(1)}}{\mathrm{n}}\xi \frac{j}{n^{3/2}} +  \frac{\mathrm{n}^{(1)}z_{0}^{(1)}-\mathrm{n} z_{0}^{(2)}}{\mathrm{n}^{2}} \frac{j^{2}}{n^{2}} \bigg] + \bigO\bigg( \frac{j^{2}(1+|\xi|)}{n^{5/2}} + \frac{j^{3}}{n^{3}} \bigg)
\end{align*}
yields
\begin{align}
& e^{-\frac{n}{2}\mathcal{L}_{\tau_{n-j}}(z)}  = \exp \bigg( -\Big(c_{0} \re \xi - \frac{j}{\sqrt{n}}\frac{z_{0}^{(1)}}{\mathrm{n}}\Big)^{2} \frac{\Delta Q(z_{0})}{4} \bigg) \bigg\{ 1 + \big[ \alpha_{1} (\re \xi)^{3} + \alpha_{2} (\re \xi)^{2}(\im \xi) \nonumber \\
&  + \alpha_{3} (\im \xi)^{2}(\re \xi) \big] \frac{1}{\sqrt{n}} + \big[ \alpha_{4} (\re \xi)^{2} + \alpha_{5} (\re \xi)(\im \xi) + \alpha_{6} (\im \xi)^{2} \big] \frac{j}{n} + \big[ \alpha_{7} \re \xi + \alpha_{8} \im \xi \big] \frac{j^{2}}{n^{3/2}} \nonumber \\
& + \alpha_{9} \frac{j^{3}}{n^{2}} + \bigO \bigg( \sum_{\ell=0}^{6} \frac{(1+|\xi|^{6-\ell})(j/\sqrt{n})^{\ell}}{n} \bigg) \bigg\} \label{lol33}
\end{align}
as $n\to + \infty$ uniformly for $|\xi| \leq M \sqrt{\log n}$, $j \leq \lfloor\epsilon_{n}n\rfloor$ and $z_{0}\in \partial \mathcal{S}$, for some $\alpha_{1},\ldots,\alpha_{9}\in \R$ that depend smoothly on $z_{0}$. Let $\tilde{c}_{n}(\xi)=\tilde{c}_{n}(\xi;z_{0})$ be the unimodular factor given by
\begin{align}\label{def of cn unimod}
\tilde{c}_{n}(\xi) & = e^{-i (a_{2} \re \xi + a_{3} \im \xi) \sqrt{n} - i [a_{5} (\re \xi)^{2} + a_{6}(\re \xi)(\im \xi) + a_{7} (\im \xi)^{2} ]} \nonumber \\
& \times e^{-i \frac{a_{11} (\re \xi)^{3} + a_{12} (\re \xi)^{2}(\im \xi) + a_{13} (\re \xi)(\im \xi)^{2} + a_{14} (\im \xi)^{3}}{\sqrt{n}}} \times e^{-2i \, \re \xi \, \im \xi}.
\end{align}
By combining \eqref{lol31}, \eqref{lol28}, \eqref{lol30}, \eqref{lol32} and \eqref{lol33}, we obtain
\begin{align}
& \tilde{c}_{n}(\xi)\overline{\tilde{c}_{n}(\eta)}K_{n}(z,w) = \bigO(n^{-100})+\tilde{c}_{n}(\xi)\overline{\tilde{c}_{n}(\eta)}e^{-\frac{1}{2}nQ(z)-\frac{1}{2}nQ(w)}\sum_{j = 1}^{\lfloor\epsilon_{n}n\rfloor}p_{n-j,n}(z)\overline{p_{n-j,n}(w)} \nonumber \\
& = \bigO(n^{-100})+ \sum_{j = 1}^{\lfloor\epsilon_{n}n\rfloor} \sqrt{n}f_{0}(\tfrac{j}{\sqrt{n}}) \bigg\{ 1 + \frac{1}{\sqrt{n}}f_{1}(\tfrac{j}{\sqrt{n}}) +  \bigO \bigg( \sum_{\ell=0}^{6} \frac{(1+|\xi|^{6-\ell}+|\eta|^{6-\ell})(j/\sqrt{n})^{\ell}}{n} \bigg) \bigg\} \label{lol41}
\end{align}
as $n\to + \infty$ uniformly for $|\xi|,|\eta|\leq M\sqrt{\log n}$ and $z_{0}\in \partial \mathcal{S}$, with
\begin{align*}
& f_{0}(x) = \frac{|\phi_{1}'(z_{0})|}{\pi \sqrt{2}}\bigg( \frac{\Delta Q(z_{0})}{4\pi} \bigg)^{\frac{1}{2}} e^{2i \, (\im \eta \, \re \eta-\im \xi \, \re \xi)} \exp \bigg\{ (1+\tilde{a}_{8}^{2})\big(2 \, \re \xi \, \re \eta-(\re \xi)^{2}-(\re \eta)^{2}\big)  \\
& + \tilde{a}_{9}^{2}\big( 2 \, \im \xi \, \im \eta - (\im \xi)^{2} - (\im \eta)^{2} \big) - 2 \tilde{a}_{8}\tilde{a}_{9} \, \re(\xi-\eta) \, \im(\xi-\eta) - 2i \bigg[ \tilde{a}_{8} \, \re(\xi+\eta) \re(\xi-\eta) \\
& + \tilde{a}_{9} \, \re(\xi + \eta) \im(\xi-\eta) \bigg] \bigg\} \exp \bigg\{ -\bigg( \frac{\sqrt{2}|\phi_{1}'(z_{0})|x}{\sqrt{\Delta Q(z_{0})}} + \re(\xi+\eta) - i \tilde{a}_{8} \re(\xi-\eta) - i \tilde{a}_{9} \im(\xi-\eta) \bigg)^{2} \bigg\} \\
& f_{1}(x) = \alpha_{1}(\re \xi)^{3}+\alpha_{1}(\re \eta)^{3} + \alpha_{2}(\re \xi)^{2}\im \xi + \alpha_{2}(\re \eta)^{2}\im\eta + \alpha_{3} (\re \xi)(\im \xi)^{2} + \alpha_{3} (\re \eta) (\im \eta)^{2} \\
& + \hat{\beta}_{4}x(\re \xi)^{2}+ \overline{\hat{\beta}_{4}}x(\re \eta)^{2} + \hat{\beta}_{5}x(\re \xi)(\im \xi) + \overline{\hat{\beta}_{5}}x(\re \eta)(\im \eta) + \hat{\beta}_{6}x(\im \xi)^{2}+ \overline{\hat{\beta}_{6}}x(\im \eta)^{2} \\
& + (\hat{\beta}_{7}+x^{2}\hat{\beta}_{8})\re \xi + (\overline{\hat{\beta}_{7}}+x^{2}\overline{\hat{\beta}_{8}})\re \eta  + (\hat{\beta}_{9}+x^{2}\hat{\beta}_{10})\im \xi + (\overline{\hat{\beta}_{9}}+x^{2}\overline{\hat{\beta}_{10}})\im \eta + \hat{\gamma}_{1} x + \hat{\gamma}_{2} x^{3} 
\end{align*}
for some $\hat{\beta}_{4},\ldots,\hat{\beta}_{10}\in \C$ and $\hat{\gamma}_{1},\hat{\gamma}_{2}\in \R$ that depend smoothly on $z_{0}$, and where 
\begin{align*}
\tilde{a}_{j} = a_{j} \frac{\sqrt{\Delta Q(z_{0})}}{2\sqrt{2}|\phi_{1}'(z_{0})|}, \qquad j=8,9.
\end{align*}
By the Euler–Maclaurin formula and the fact that $f_{0}(x),f_{0}'(x),f_{0}''(x)$ are $\bigO(e^{-cx^{2}})$ as $x\to + \infty$ for some $c>0$, we can choose $K$ sufficiently large such that
\begin{align}\label{lol34}
\sum_{j = 1}^{\lfloor\epsilon_{n}n\rfloor} \sqrt{n}f_{0}(\tfrac{j}{\sqrt{n}}) = n \int_{0}^{+\infty}f_{0}(x)dx - \frac{1}{2} f_{0}(0) \sqrt{n} + \bigO(1+|f_{0}'(0)|), \qquad \mbox{as } n \to + \infty
\end{align}
uniformly for $|\xi|,|\eta|\leq M\sqrt{\log n}$ and $z_{0}\in \partial \mathcal{S}$. On the other hand, we know from \cite[Corollary 1.7]{HW2021} that the leading term of $K_{n}(z,w)$ is equal to $\frac{\Delta Q(z_{0})n}{4\pi}\mathrm{k}(\xi,\eta)$ (up to unimodular factors). To reconcile \cite[Corollary 1.7]{HW2021} with \eqref{lol35} and \eqref{lol34}, it is easy to check that we must have $\tilde{a}_{8}=0$ and $\tilde{a}_{9}=-1$. Indeed, $\tilde{a}_{8},\tilde{a}_{9}\in \R$, and comparing $\partial_{\re\xi}\partial_{\im\eta} \log \mathrm{k}(\xi,\eta)|_{\xi=\eta=0} = -i\big( 2-\frac{4}{\pi} \big)$ with
\begin{align*}
\partial_{\re\xi}\partial_{\im\eta} \log \int_{0}^{+\infty}f_{0}(x)dx\Big|_{\xi=\eta=0} = \tilde{a}_9(\tilde{a}_8+i)\bigg( 2-\frac{4}{\pi} \bigg)
\end{align*}
gives $\tilde{a}_9(\tilde{a}_8+i)=-i$. Therefore $f_{0}$ simplifies to
\begin{align*}
& f_{0}(x) = \frac{|\phi_{1}'(z_{0})|}{\pi \sqrt{2}}\bigg( \frac{\Delta Q(z_{0})}{4\pi} \bigg)^{\frac{1}{2}} e^{2\xi \overline{\eta}-(|\xi|^{2}+|\eta|^{2})} \exp \bigg\{ -\bigg( \frac{\sqrt{2}|\phi_{1}'(z_{0})|x}{\sqrt{\Delta Q(z_{0})}} + \xi + \overline{\eta} \bigg)^{2}\bigg\}.
\end{align*}
In particular, 
\begin{align}\label{lol42}
-\frac{1}{2}f_{0}(0) = \hat{\gamma}_{3} \frac{e^{-2(\re \xi)^{2}-2(\re \eta)^{2}}}{e^{2i(\re \xi \, \im \xi - \re \eta \, \im \eta)}}, \qquad |f_{0}'(0)| \leq C |\xi + \overline{\eta}|e^{-2(\re \xi)^{2}-2(\re \eta)^{2}},
\end{align}
for some $\hat{\gamma}_{3}\in \R$ that depends smoothly on $z_{0}$ and some $C>0$ that is independent of $z_{0}$. Using the change of variables $y=\frac{\sqrt{2}|\phi_{1}'(z_{0})|x}{\sqrt{\Delta Q(z_{0})}}$, we obtain
\begin{align}\label{change of var x to y}
f_{0}(x)dx = \frac{\Delta Q(z_{0})}{4\pi} e^{2\xi \overline{\eta}-(|\xi|^{2}+|\eta|^{2})} \frac{1}{\sqrt{\pi}} e^{-(y+\xi+\overline{\eta})^{2}}dy,
\end{align}
and therefore, 
\begin{align}
& \sum_{j = 1}^{\lfloor\epsilon_{n}n\rfloor} f_{0}(\tfrac{j}{\sqrt{n}})f_{1}(\tfrac{j}{\sqrt{n}}) = \sqrt{n} \int_{0}^{+\infty}f_{0}(x)f_{1}(x)dx + \bigO(1+|\xi|^{4}+|\eta|^{4}) \nonumber \\
& = \bigg[\Big( \tilde{\beta}_{1} (\re \xi)^{3} + \tilde{\beta}_{1} (\re \eta)^{3} + \tilde{\beta}_{2} (\re \xi)^{2} \im \xi + \tilde{\beta}_{2} (\re \eta)^{2} \im \eta + \tilde{\beta}_{3} \re \xi (\im \xi)^{2}+ \tilde{\beta}_{3} \re \eta (\im \eta)^{2} \nonumber \\
& + \tilde{\beta}_{4} \re \xi + \overline{\tilde{\beta}_{4}} \re \eta + \tilde{\beta}_{5} \im \xi + \overline{\tilde{\beta}_{5}} \im \eta \Big) \mathrm{k}(\xi,\eta) + \Big( \tilde{\beta}_{6} (\re \xi)^{2} + \overline{\tilde{\beta}_{6}} (\re \eta)^{2} + \tilde{\beta}_{7} (\re \xi)(\im \xi) \nonumber \\
&  + \overline{\tilde{\beta}_{7}} (\re \eta)(\im \eta) + \tilde{\beta}_{8} (\im \xi)^{2} + \overline{\tilde{\beta}_{8}} (\im \eta)^{2} + \tilde{\gamma}_{1} \Big) \Big( (\xi + \overline{\eta})\mathrm{k}(\xi,\eta) - \frac{e^{-2 (\re \xi)^{2}-2 (\re \eta)^{2}}}{2\sqrt{\pi}e^{2i(\re \xi \, \im \xi - \re \eta \, \im \eta)}} \Big) \nonumber \\
& + \tilde{\gamma}_{2} (\xi + \overline{\eta})\mathrm{k}(\xi,\eta) + \Big( \tilde{\beta}_{9} \re \xi + \overline{\tilde{\beta}_{9}} \re \eta + \tilde{\beta}_{10} \im \xi + \overline{\tilde{\beta}_{10}} \im \eta \Big) \bigg\{ (\xi + \overline{\eta})^{2}\mathrm{k}(\xi,\eta) \nonumber \\
& - \frac{(\xi + \overline{\eta})e^{-2 (\re \xi)^{2}-2 (\re \eta)^{2}}}{2\sqrt{\pi}e^{2i(\re \xi \, \im \xi - \re \eta \, \im \eta)}} \bigg\} + \tilde{\gamma}_{3} \bigg\{ (\xi + \overline{\eta})^{3}\mathrm{k}(\xi,\eta) - \frac{(\xi + \overline{\eta})^{2}e^{-2 (\re \xi)^{2}-2 (\re \eta)^{2}}}{2\sqrt{\pi}e^{2i( \re \xi \, \im \xi - \re \eta \, \im \eta)}} \bigg\}\bigg]\sqrt{n} \nonumber \\
& + \bigO\big( 1+|\xi|^{4}+|\eta|^{4} \big)  \label{lol39}
\end{align}
as $n\to + \infty$ uniformly for $|\xi|,|\eta|\leq M \sqrt{\log n}$ and $z_{0}\in \partial \mathcal{S}$, and where the coefficients $\tilde{\beta}_{4},\ldots,\tilde{\beta}_{10}\in \C$ and $\tilde{\beta}_{1},\tilde{\beta}_{2},\tilde{\beta}_{3},\tilde{\gamma}_{1},\tilde{\gamma}_{2},\tilde{\gamma}_{3} \in \R$ are smooth functions of $z_{0}$ that can be explicitly expressed in terms of $\alpha_{1},\alpha_{2},\alpha_{3},\hat{\beta}_{4},\ldots,\hat{\beta}_{10}, \hat{\gamma}_{1},\hat{\gamma}_{2}$. The identity \eqref{change of var x to y} implies that $\int_{0}^{+\infty}f_{0}(x)dx = \frac{\Delta Q(z_{0})}{4\pi}\mathrm{k}(\xi,\eta)$. Using that
\begin{align}\label{lol40}
\sum_{j = 1}^{\lfloor\epsilon_{n}n\rfloor} \sqrt{n}|f_{0}(\tfrac{j}{\sqrt{n}})| \bigg( \sum_{\ell=0}^{6} \frac{(1+|\xi|^{6-\ell})(j/\sqrt{n})^{\ell}}{n} \bigg)  = \bigO\big( 1+|\xi|^{6}+|\eta|^{6} \big), \qquad \mbox{as } n \to + \infty
\end{align}
uniformly for $|\xi|,|\eta| \leq M \sqrt{\log n}$ and $z_{0}\in \partial \mathcal{S}$, by combining \eqref{lol41}, \eqref{lol34}, \eqref{lol42}, \eqref{lol39} and \eqref{lol40}, we obtain
\begin{multline}
\tilde{c}_{n}(\xi)\overline{\tilde{c}_{n}(\eta)}K_{n}(z,w) = n \frac{\Delta Q(z_{0})}{4\pi}\mathrm{k}(\xi,\eta) \\ + \bigg(\int_{0}^{+\infty}f_{0}(x)f_{1}(x)dx-\frac{1}{2}f_{0}(0)\bigg) \sqrt{n} + \bigO\big( 1+|\xi|^{6}+|\eta|^{6} \big) \label{lol43}
\end{multline}
as $n\to + \infty$ uniformly for $|\xi|,|\eta| \leq M \sqrt{\log n}$ and $z_{0}\in \partial \mathcal{S}$. Let $\hat{c}_{n}(\xi)=\hat{c}_{n}(\xi;z_{0})$ be the unimodular function defined by
\begin{align*}
\hat{c}_{n}(\xi)=e^{-\frac{i}{\sqrt{n}}(\frac{\Delta Q(z_{0})}{4\pi})^{-1}\big( \im \tilde{\beta}_{4} \, \re \xi + \im \tilde{\beta}_{5} \, \im \xi \big)},
\end{align*}
and set $c_{n}(\xi):=\tilde{c}_{n}(\xi)\hat{c}_{n}(\xi)$. By multiplying both sides of \eqref{lol43} by $\hat{c}_{n}(\xi)\overline{\hat{c}_{n}(\eta)}$, we find that the term of order $\sqrt{n}$ in the large $n$ asymptotics of $c_{n}(\xi)\overline{c_{n}(\eta)}K_{n}(z,w)$ is independent of $\im \tilde{\beta}_{4},\im \tilde{\beta}_{5}$, and can therefore be rewritten in the form $\mathrm{k}_{2}(\xi,\eta;z_{0}) \sqrt{n}$, where $\mathrm{k}_{2}(\xi,\eta;z_{0})$ is as in Theorem \ref{thm:Charlier}.

\begin{proof}[Proof of Lemma \ref{lemma:joint regularity}]
The proof is divided into three steps and relies on several results from \cite{HS2002}. 

\underline{Step 1: properties of $\chi_s$ and establishing \eqref{eq:local conformal flow}.}
Since \cite[Theorem 6.2]{HS2002} is formulated for conformal maps from the unit disk to a bounded, simply connected domain, we fix $p\in\mathcal{S}^{\circ}$ and choose $\delta>0$ sufficiently small so that $p\in\mathcal{S}_{1-s}^{\circ}$ for $s\in(-\delta,\delta)$. For $s\in(-\delta,\delta)$, we define
\begin{align}\label{def of chi s}
& \chi_s(z):=\frac{1}{\phi_{1-s}^{-1}(1/z)-p} \;\;\;  \mbox{ for } |z|<1,
\qquad \chi_s(0):=0, \\
& \mathcal{D}(s) = \chi_{s}(\{z :|z|<1\}) = \bigg\{ \frac{1}{\phi_{1-s}^{-1}(z)-p} : |z|>1 \bigg\} \cup \{0\} = \bigg\{ \frac{1}{z-p} : z \in \mathcal{S}_{1-s}^{c} \bigg\} \cup \{0\}.
\end{align}
The domain $\mathcal{D}(s)$ is bounded and simply connected.

 By \cite[page 318]{HW2021}, $\partial\mathcal{S}_{1-s}$
is real-analytic for each $s \in (-\delta,\delta)$ (decreasing $\delta>0$ if necessary). Hence, by
\cite[Proposition 3.1]{Pommerenke1992}, for each $s \in (-\delta,\delta)$, $\phi_{1-s}^{-1}$ extends conformally in a neighborhood of the unit circle,
and thus $\chi_s$ is conformal on
$\{z\in\C:|z|<1+\epsilon_{s}\}$ for some $\epsilon_{s}>0$.
Note also that
$\chi_s'(0)=\phi_{1-s}'(\infty)>0$.

By \cite[(2.5)]{HW2021}, we have
\begin{align}\label{lol57}
\partial_\tau\bigg(
\mathbbm{1}_{\mathcal{S}_\tau}(z)
\frac{\Delta Q(z)}{4\pi}\,d^2z
\bigg)
=
d\omega_{\tau,\infty}(z)
=
\frac{|\phi_\tau'(z)|}{2\pi}\,
|dz|,
\end{align}
where $\omega_{\tau,\infty}$ is the harmonic measure
for $\mathcal{S}_\tau^c$ viewed from infinity,
and $|dz|$ is the arclength
measure on $\partial\mathcal{S}_\tau$. Under the change of variables $z=p+1/w$, we have
\begin{align}\label{lol59}
\frac{\Delta Q(z)}{4\pi}\,d^2z
=
\varpi(w)\,\frac{d^2w}{\pi},
\qquad
\varpi(w):=\frac{\Delta Q(p+1/w)}{4|w|^4},
\qquad w\in\C\setminus\{0\}.
\end{align}
Since $\Delta Q$ is positive and real-analytic in a neighborhood of $\partial\mathcal{S}$, $\varpi$ is positive and real-analytic in a neighborhood of
$\partial\mathcal{D}(0)=\chi_0(\{|z|=1\})$.
Decreasing $\delta>0$ if necessary, this neighborhood contains $\mathcal{D}(s_{1})\setminus\mathcal{D}(s_{0})$ for every $-\delta<s_0<s_1<\delta$. 

\begin{remark}\label{remark:difference between me and HS}
Note that $\varpi$ in \eqref{lol59} may be singular at $0$, whereas \cite[Theorem 6.2]{HS2002} assumes that $\varpi$ is positive and real-analytic in a neighborhood of $\overline{\mathcal{D}(0)}$. Nevertheless, as we show below, the arguments at the beginning of \cite[Section 6]{HS2002} can be adapted to obtain \cite[(6.7)]{HS2002} in our setting. The proof of \cite[Theorem 6.2]{HS2002}
can then easily be adapted: since $\varpi_s=\varpi$ is independent of $s$, the term containing $\partial_s\varpi_s$ in \cite[(6.14)]{HS2002} vanishes, and the rest of the proof only requires $\varpi$ to be positive and real-analytic near $\partial\mathcal{D}(0)$.
\end{remark}

For $-\delta < s_{0}< s_{1}<\delta$, if $h$ is harmonic in a neighborhood
of $\overline{\mathcal{D}(s_{1})}$, then the change of variables
$z=p+1/w$ and \eqref{lol57} give
\begin{align*}
& \int_{\mathcal{D}(s_{1})\setminus\mathcal{D}(s_{0})}
h(w)\varpi(w)\,\frac{d^2w}{\pi} =
\int_{\mathcal{S}_{1-s_{0}}\setminus\mathcal{S}_{1-s_{1}}}
h\bigg(\frac{1}{z-p}\bigg)
\frac{\Delta Q(z)}{4\pi}\,d^2z \\
&=
\int_{1-s_{1}}^{1-s_{0}}
\bigg[
\int_{\partial\mathcal{S}_\tau}
h\bigg(\frac{1}{z-p}\bigg)
\,d\omega_{\tau,\infty}(z)
\bigg]d\tau =
\int_{1-s_{1}}^{1-s_{0}}h(0)\,d\tau
=(s_{1}-s_{0})\,h(0).
\end{align*}
The above is an analogue of \cite[(6.2) with $t(s_{1})-t(s_{0})=s_{1}-s_{0}$ and $\varpi_s=\varpi$]{HS2002}.

We now follow \cite[(6.3)--(6.7)]{HS2002}.
Fix $s_0\in(-\delta,\delta)$. As in \cite[equation before (6.7)]{HS2002},
for $s$ sufficiently close to $s_0$, set
\begin{align*}
\widetilde{\mathcal{D}}(s) :=\chi_{s_0}^{-1}(\mathcal{D}(s)), \quad \widetilde{\chi}_s :=\chi_{s_0}^{-1}\circ\chi_s, \quad \widetilde{\varpi}(z) :=\varpi(\chi_{s_0}(z))|\chi_{s_0}'(z)|^2,
\qquad |z|-1\in (-\epsilon,\epsilon),
\end{align*}
where $\epsilon>0$ is small enough (and may depend on $s_{0}$). Then
\begin{align*}
\widetilde{\mathcal{D}}(s_0)=\{z:|z|<1\},
\qquad
\widetilde{\chi}_{s_0}(z)=z.
\end{align*}
Note that $\widetilde{\varpi}$ is positive, real-analytic near the unit circle, and is independent of $s$.

The function $g(z):=\chi_{s_0}^{-1}(\frac{1}{z-p})
=\frac{1}{\phi_{1-s_0}(z)}$ maps $\mathcal{S}_{1-s_0}^{c}$ to the unit disk; in particular, it maps the inward normal to $\partial\mathcal{S}_{1-s_0}$
to the outward normal to the unit circle, and $|g'(z)| = |\phi_{1-s_0}'(z)|$ for $z\in \partial \mathcal{S}_{1-s_0}$. Thus, by \cite[Lemma 2.5]{HW2021}, we have
$\partial\widetilde{\mathcal{D}}(s) = \{r_{s}(\theta)e^{i\theta}:\theta \in [0,2\pi], \; r_{s}(\theta)>0\}$, where 
\begin{align}\label{lol58}
r_s(\theta)
=
1+(s-s_0)
\frac{2|\phi_{1-s_0}'(\zeta)|^2}{\Delta Q(\zeta)}
+O((s-s_0)^2), \qquad
\zeta=\phi_{1-s_0}^{-1}(e^{-i\theta})=p+\frac{1}{\chi_{s_{0}}(e^{i\theta})},
\end{align}
as $s\to s_{0}$ uniformly for $\theta\in[0,2\pi]$.
Thus \cite[(6.3)]{HS2002} is satisfied with
\begin{align*}
\rho(e^{i\theta})
=
\frac{2|\phi_{1-s_0}'(\zeta)|^2}{\Delta Q(\zeta)} 
= \frac{2|\chi_{s_0}(e^{i\theta})|^4}
{\Delta Q(\zeta)|\chi_{s_0}'(e^{i\theta})|^2}
=
\frac{1}{2\widetilde{\varpi}(e^{i\theta})}.
\end{align*}
(Here the middle equality follows by differentiating $
\phi_{1-s_0}\big(p+\frac{1}{\chi_{s_0}(z)}\big)
=\frac{1}{z}$ at $z=e^{i\theta}$.)

By \cite[(6.6)]{HS2002} and \eqref{lol58}, for $|z|<1$ we have
\begin{align*}
\left.\partial_s\widetilde{\chi}_s(z)\right|_{s=s_0}
&=
\frac{z}{2\pi}\int_0^{2\pi}
\frac{e^{i\theta}+z}{e^{i\theta}-z}
\rho(e^{i\theta})\,d\theta =
\frac{z}{2\pi}\int_0^{2\pi}
\frac{e^{i\theta}+z}{e^{i\theta}-z}
\frac{d\theta}{2\widetilde{\varpi}(e^{i\theta})}.
\end{align*}
Since $\chi_s=\chi_{s_0}\circ\widetilde{\chi}_s$
and $\widetilde{\chi}_{s_0}(z)=z$, we have
\begin{align*}
\left.\partial_s\chi_s(z)\right|_{s=s_0}
&=
\chi_{s_0}'(z)
\left.\partial_s\widetilde{\chi}_s(z)\right|_{s=s_0} =
\frac{z\chi_{s_0}'(z)}{2\pi}
\int_0^{2\pi}
\frac{e^{i\theta}+z}{e^{i\theta}-z}
\frac{d\theta}{
2\varpi(\chi_{s_0}(e^{i\theta}))
|\chi_{s_0}'(e^{i\theta})|^2}.
\end{align*}
As $s_0\in(-\delta,\delta)$ was arbitrary, this proves
\begin{align}\label{eq:local conformal flow}
\partial_s\chi_s(z)
=
\frac{z\chi_s'(z)}{2\pi}
\int_0^{2\pi}
\frac{e^{i\theta}+z}{e^{i\theta}-z}
\frac{d\theta}{
2\varpi(\chi_s(e^{i\theta}))
|\chi_s'(e^{i\theta})|^2},
\qquad |z|<1.
\end{align}
This is \cite[(6.7)]{HS2002} with $t'(s)=1$
and $\partial_s\varpi_s=0$.

\underline{Step 2: properties of $\widehat{\chi}_s$.} We now apply \cite[Theorem 6.2]{HS2002} with $s_0=0$, $\chi_0$ as in \eqref{def of chi s}, $\varpi_s=\varpi$
and $t(s)=1+s$ (as mentioned in Remark \ref{remark:difference between me and HS}, our $\varpi$ does not meet the assumptions of \cite[Theorem 6.2]{HS2002}, but the proof carries over). Thus, decreasing $\delta>0$ if necessary, there exist $\epsilon'>0$
and conformal maps $\{\widehat{\chi}_s\}_{s\in (-\delta,\delta)}$
on the fixed disk $\{z:|z|<1+\epsilon'\}$ satisfying
\eqref{eq:local conformal flow} with $\chi_s$ replaced
by $\widehat{\chi}_s$, and
\begin{align*}
\widehat{\chi}_0=\chi_0,\qquad
\widehat{\chi}_s(0)=0,\qquad
\widehat{\chi}_s'(0)>0.
\end{align*}
Moreover, the last paragraph of the proof of
\cite[Theorem 6.2]{HS2002} (together with the definition
of the spaces $B_r$ immediately after \cite[(6.15)]{HS2002}) gives, after decreasing $\epsilon'$ and $\delta$ if necessary, real-analytic dependence of $s \mapsto \widehat{\chi}_s$ 
with respect to the norm $\|f\|:=\sup\{|f(z)|:|z|\leq1+\epsilon'\}$. More precisely, there are functions $c_m$,
holomorphic in a neighborhood of
$\{z:|z|\leq1+\epsilon'\}$, such that
\begin{align*}
\widehat{\chi}_s(z)
=\sum_{m=0}^{+\infty}c_m(z)s^m,
\qquad
\sum_{m=0}^{+\infty}
\sup_{|z|\leq1+\epsilon'}|c_m(z)|\,\delta^m
<+\infty.
\end{align*}
Hence $\widehat{\chi}_s(z)$ is jointly real-analytic
in $(\re z,\im z,s)$ for $|z|<1+\epsilon'$ and
$|s|<\delta$.

It remains to show that $\widehat{\chi}_s=\chi_s$.

\underline{Step 3: $\widehat{\chi}_s=\chi_s$.} 
For $0\leq s<\delta$, $\widehat{\chi}_s=\chi_s$ follows from \cite[Theorem 6.8 and Remark 6.9(b) with $K_*=\mathcal{S}$, $t_*=1$ and $t=1-s$]{HM2013}.
Here, we will also prove that $\widehat{\chi}_s=\chi_s$
holds for $-\delta<s\leq0$. Below, we explicitly use that $Q$ is regular, i.e. that \eqref{EL inequality} holds strictly (which is not assumed in \cite[Theorem 6.8]{HM2013}). Since this result may be of independent interest, we have decided to keep our proof, although the case $s<0$ is not needed for the proof of Theorem \ref{thm:Charlier}.

Set
\begin{align}\label{def of phi hat}
\widehat{\phi}_\tau^{-1}(z)
:=p+\frac{1}{\widehat{\chi}_{1-\tau}(1/z)}, \qquad |z|>\frac{1}{1+\epsilon'}, \; |1-\tau| < \delta,
\end{align}
and define $\widehat{\mathcal{S}}_\tau:= \C\setminus \widehat{\phi}_\tau^{-1}(\{|z|> 1\})$.
Since $\widehat{\chi}_0=\chi_0$, we have $\widehat{\mathcal{S}}_1=\mathcal{S}$.
The map $\widehat{\phi}_\tau^{-1}(z)$ is jointly real-analytic in $(\re z,\im z,\tau)$ for $|z|>(1+\epsilon')^{-1}$ and
$|1-\tau|<\delta$. By the inverse function theorem, $\widehat{\phi}_\tau(z)$ is also jointly real-analytic in $(\re z,\im z,\tau)$ for $z\in \widehat{\phi}_\tau^{-1}(\{|z|>(1+\epsilon')^{-1}\})$ and
$|1-\tau|<\delta$. 

Decreasing $\delta>0$ if necessary, there exists an open neighborhood $U$ of
$\overline{\mathcal{S}^{c}}$ such that
\begin{align*}
U\subset
\widehat{\phi}_\tau^{-1}
\big(\{|z|>(1+\epsilon')^{-1}\}\big)
\qquad \mbox{for all }  |1-\tau|<\delta.
\end{align*}
Thus $\widehat{\phi}_\tau(z)$ is jointly real-analytic
in $(\re z,\im z,\tau)$ on $U\times(1-\delta,1+\delta)$.

Suppose $g$ is a real-valued and real-analytic function in a neighborhood of $\partial\mathcal{S}$, and let $\mathcal{H}_\tau[g]$ be
the bounded holomorphic function in $\widehat{\mathcal{S}}_\tau^c$ satisfying $\im \mathcal{H}_\tau[g](\infty)=0$ and $\re \mathcal{H}_\tau[g]=g$ on $\partial\widehat{\mathcal{S}}_\tau$.
By the Herglotz formula (see \cite[Section 2.5]{HW2021}),
\begin{align}\label{def of mathcalH tau}
\mathcal{H}_\tau[g](\widehat{\phi}_\tau^{-1}(z))
=
\frac{1}{2\pi}\int_0^{2\pi}
\frac{z+e^{it}}{z-e^{it}}\,
g(\widehat{\phi}_\tau^{-1}(e^{it}))\,dt,
\qquad |z|>1.
\end{align}
The function $g(\widehat{\phi}_\tau^{-1}(e^{it}))$ is jointly
real-analytic in $(t,\tau)$ and $2\pi$-periodic in $t$.
Thus, decreasing $\delta>0$ if necessary, it admits
a holomorphic extension $\tilde{g}_\tau(t)$ to
$\{t\in\C:|\im t|<\epsilon''\}$ for some $\epsilon''>0$
independent of $\tau\in (1-\delta,1+\delta)$.
By the identity theorem, $\tilde{g}_\tau(t+2\pi)=\tilde{g}_\tau(t)$ for all $|\im t|<\epsilon''$. Moreover, $\tilde{g}_\tau(t)$ is jointly real-analytic
in $(\re t,\im t,\tau)$ for $|\im t|<\epsilon''$
and $|1-\tau|<\delta$. Define
\begin{align*}
\hat{g}_\tau(\eta):=\tilde{g}_\tau(-i\log\eta),
\qquad e^{-\epsilon''}<|\eta|<e^{\epsilon''},
\quad \eta\notin(-\infty,0],
\end{align*}
where the principal branch is used for $\log$. Since $\tilde{g}_\tau$ is $2\pi$-periodic, $\hat{g}_\tau$
admits a holomorphic extension to the whole
annulus $\{e^{-\epsilon''}<|\eta|<e^{\epsilon''}\}$.
This extension is jointly real-analytic in
$(\re\eta,\im\eta,\tau)$. Choosing $\rho=e^{-\epsilon''/2}$
deforming the contour gives
\begin{align*}
\mathcal{H}_\tau[g](\widehat{\phi}_\tau^{-1}(z))
=
\frac{1}{2\pi i}\int_{|\eta|=\rho}
\frac{z+\eta}{z-\eta}\,
\hat{g}_\tau(\eta)\frac{d\eta}{\eta}, \qquad |z|>1.
\end{align*}
The right-hand side is clearly holomorphic in the larger domain $|z|>\rho$ and is also jointly real-analytic in
$(\re z,\im z,\tau)$. Thus $\mathcal{H}_\tau[g](z)$ is holomorphic in $z$ in a neighborhood of $\overline{\mathcal{S}^{c}}$ and jointly real-analytic in $(\re z,\im z,\tau)$.

Note that $\widehat{\mathcal{S}}_\tau=\mathcal{S}_\tau$
implies that $\widehat{\phi}_\tau=\phi_\tau$ (and hence
$\widehat{\chi}_s=\chi_s$). This follows from the uniqueness part of the conformal mapping theorem, since they are both normalized in the same way (i.e.
$\phi_\tau(\infty)=\widehat{\phi}_\tau(\infty)=\infty$
and $\phi_\tau'(\infty),\widehat{\phi}_\tau'(\infty)>0$).

We will prove that $\widehat{\mathcal{S}}_\tau=\mathcal{S}_\tau$ using the uniqueness of equilibrium measures \cite[Theorem I.1.3]{SaTo}. 
To this end, we introduce
\begin{align}\label{lol62}
d\widehat{\mu}_\tau(z) :=\frac{\Delta Q(z)}{4\pi}
\mathbbm{1}_{\widehat{\mathcal{S}}_\tau}(z)\,d^2z, \qquad \widehat{P}_\tau(z) :=Q(p)+2\int_{\widehat{\mathcal{S}}_\tau} \log\frac{|z-w|}{|p-w|}\,d\widehat{\mu}_\tau(w).
\end{align}
Using Green's formula as in \cite[(6.9)]{HS2002} (using that $\frac{\Delta Q}{4\pi}$ is
independent of $\tau$), we obtain for every $f\in C_c^\infty(\C)$ that
\begin{align}\label{lol60}
\frac{d}{d\tau}\int_\C f(z)\,d\widehat{\mu}_\tau(z)
=
\int_{\partial\widehat{\mathcal{S}}_\tau}
f(z)\frac{\Delta Q(z)}{4\pi}V_\tau(z)\,|dz|,
\end{align}
where $V_\tau$ is the outward normal velocity of $\partial\widehat{\mathcal{S}}_\tau$, defined for $\widehat{z}_{0} \in \partial\widehat{\mathcal{S}}_\tau$ by
\begin{align*}
V_\tau(\widehat{z}_0) =
\left.
\re\bigg(
\partial_\tau\widehat{\phi}_\tau^{-1}(\zeta)\,
\overline{\widehat{\mathrm{n}}_\tau}
\bigg)\right|_{\zeta=\widehat{\phi}_\tau(\widehat{z}_0)}.
\end{align*}
Here, in a similar way as in \eqref{lol38}, $\widehat{\mathrm{n}}_\tau$ is the outward unit normal to
$\partial\widehat{\mathcal{S}}_\tau$ at $\widehat{z}_0$ and is given by
\begin{align*}
\widehat{\mathrm{n}}_\tau
=
\frac{\widehat{\phi}_\tau(\widehat{z}_0)\,
(\widehat{\phi}_\tau^{-1})'(\widehat{\phi}_\tau(\widehat{z}_0))}
{|(\widehat{\phi}_\tau^{-1})'(\widehat{\phi}_\tau(\widehat{z}_0))|}.
\end{align*}
We now compute $V_\tau$. For
$\widehat{z}_0\in\partial\widehat{\mathcal{S}}_\tau$, set
$s=1-\tau$ and $\zeta=\widehat{\phi}_\tau(\widehat{z}_0)$,
so that $|\zeta|=1$. Then
\begin{align}
V_\tau(\widehat{z}_0)
& =
\re\bigg(
\partial_\tau\widehat{\phi}_\tau^{-1}(\zeta)\,
\frac{\overline{\zeta (\widehat{\phi}_\tau^{-1})'(\zeta)}}
{|(\widehat{\phi}_\tau^{-1})'(\zeta)|}
\bigg)  =
|(\widehat{\phi}_\tau^{-1})'(\zeta)|
\re\frac{\partial_\tau\widehat{\phi}_\tau^{-1}(\zeta)}
{\zeta(\widehat{\phi}_\tau^{-1})'(\zeta)} =
|(\widehat{\phi}_\tau^{-1})'(\zeta)|
\re\frac{\partial_s\widehat{\chi}_s(\zeta^{-1})}
{\zeta^{-1}\widehat{\chi}_s'(\zeta^{-1})}
\nonumber\\
&=
\frac{|(\widehat{\phi}_\tau^{-1})'(\zeta)|}{4\pi}
\lim_{r\to 1, r<1}
\int_0^{2\pi}
\re\bigg(
\frac{e^{i\theta}+r\zeta^{-1}}
{e^{i\theta}-r\zeta^{-1}}
\bigg)
\frac{d\theta}{
\varpi(\widehat{\chi}_s(e^{i\theta}))
|\widehat{\chi}_s'(e^{i\theta})|^2} =
\frac{|(\widehat{\phi}_\tau^{-1})'(\zeta)|}{
2\varpi(\widehat{\chi}_s(\zeta^{-1}))
|\widehat{\chi}_s'(\zeta^{-1})|^2}
\nonumber\\
& =
\frac{|(\widehat{\phi}_\tau^{-1})'(\zeta)|}{
\displaystyle
\frac{\Delta Q(\widehat{z}_0)}
{2|\widehat{\chi}_s(\zeta^{-1})|^4}
|\widehat{\chi}_s(\zeta^{-1})|^4
|(\widehat{\phi}_\tau^{-1})'(\zeta)|^2} =
\frac{2}{
\Delta Q(\widehat{z}_0)
|(\widehat{\phi}_\tau^{-1})'(\zeta)|}
=
\frac{2|\widehat{\phi}_\tau'(\widehat{z}_0)|}
{\Delta Q(\widehat{z}_0)}. \label{lol61}
\end{align}
Above, we first used \eqref{def of phi hat}, then
\eqref{eq:local conformal flow}, and then \eqref{def of phi hat} and \eqref{lol59}.

Substituting \eqref{lol61} into \eqref{lol60} gives
\begin{align*}
\frac{d}{d\tau}\int_\C f(z)\,d\widehat{\mu}_\tau(z) =
\int_{\partial\widehat{\mathcal{S}}_\tau}
f(w)\frac{|\widehat{\phi}_\tau'(w)|}{2\pi}\,|dw| =
\int_{\partial\widehat{\mathcal{S}}_\tau} f(w)\,d\widehat{\omega}_{\tau,\infty}(w),
\end{align*}
where $\widehat{\omega}_{\tau,\infty}$ is the harmonic measure for $\widehat{\mathcal{S}}_\tau^{c}$ viewed from infinity. Thus we have $\partial_\tau\widehat{\mu}_\tau=d\widehat{\omega}_{\tau,\infty}$ in distribution.

Since $\widehat{\mathcal{S}}_1=\mathcal{S}_1$ and $\widehat{\omega}_{\tau,\infty}(\partial\widehat{\mathcal{S}}_\tau)=1$, we have that $\widehat{\mu}_1$ is the equilibrium measure for $Q$, and thus $\widehat{\mu}_\tau(\C)=\tau$. Moreover, since $p\in\mathcal{S}^{\circ}$,
\eqref{EL equality} and \eqref{def of Qhat} give
\begin{align}\label{lol63}
\widehat{P}_1(z) =
Q(p)+2\int_\C
\log\frac{|z-w|}{|p-w|}\,d\mu(w) =
\ell_Q+2\int_\C\log|z-w|\,d\mu(w)
=\widehat{Q}(z), \qquad z \in \C.
\end{align}

Let $z$ be such that $z\in \widehat{\mathcal{S}}_\tau^\circ$ for all $\tau \in (1-\delta,1+\delta)$. We have
\begin{align*}
\partial_\tau\widehat{P}_\tau(z)
=
2\int_{\partial\widehat{\mathcal{S}}_\tau}
\log\frac{|z-w|}{|p-w|}
\,d\widehat{\omega}_{\tau,\infty}(w)=0.
\end{align*}
(The last equality holds because the integrand is
harmonic in $\widehat{\mathcal{S}}_\tau^{c}$
and vanishes at infinity.)

Since $\widehat{P}_1=Q$ on $\widehat{\mathcal{S}}_{1}=\mathcal{S}$, it follows that
$\widehat{P}_\tau=Q$ in a small disk around $p$. Moreover, it directly follows from \eqref{lol62} that $\Delta(Q-\widehat{P}_\tau)(z)=0$ for all $z\in \widehat{\mathcal{S}}_\tau^\circ$ and that $\widehat{P}_\tau$ is $C^1$ on $\C$. Thus, by harmonic continuation,
\begin{align}\label{eq:constructed potential equality}
\widehat{P}_\tau(z)=Q(z),
\qquad z\in \widehat{\mathcal{S}}_\tau, \; \tau \in (1-\delta,1+\delta).
\end{align}
For $z\in \widehat{\mathcal{S}}_\tau^{c}$, $\widehat{P}_\tau-2\tau\log|\widehat{\phi}_\tau|$ is harmonic, equal to $Q$ on $\partial\widehat{\mathcal{S}}_\tau$, and is bounded at $\infty$. Thus
\begin{align*}
\widehat{P}_\tau(z)
=
2\tau\log|\widehat{\phi}_\tau(z)|
+\re\mathcal{H}_\tau[Q](z),
\qquad z\in \widehat{\mathcal{S}}_\tau^{c}, \; \tau \in (1-\delta,1+\delta).
\end{align*}
Hence, for every
$\widehat{z}_0\in\partial\widehat{\mathcal{S}}_\tau$,
\begin{align*}
& (Q-\widehat{P}_\tau)(\widehat{z}_0)=0,
\qquad
\nabla(Q-\widehat{P}_\tau)(\widehat{z}_0)=0,\\
& (Q-\widehat{P}_\tau)
\big(\widehat{z}_0
+\hat{\epsilon}\,\widehat{\mathrm{n}}_\tau(\widehat{z}_0)\big) = \frac{\Delta Q(\widehat{z}_0)}{2}\hat{\epsilon}^2
+O(\hat{\epsilon}^3),
\qquad \mbox{as } \hat{\epsilon}\to 0, \hat{\epsilon}> 0,
\end{align*}
uniformly for
$\widehat{z}_0\in\partial\widehat{\mathcal{S}}_\tau$
and $\tau \in (1-\delta,1+\delta)$, and where the dependence of 
$\widehat{\mathrm{n}}_\tau$ in $\widehat{z}_0$ has been made explicit. Since $\Delta Q>0$ on $\mathcal{S}$, decreasing $\delta>0$ if necessary, we infer that there exists $d_0>0$ independent of
$\tau\in (1-\delta,1+\delta)$, such that
\begin{align*}
\widehat{P}_\tau(z)<Q(z)
\qquad\mbox{if }\;
0<\operatorname{dist}(z,\widehat{\mathcal{S}}_\tau)<d_0.
\end{align*}
Decreasing again $\delta>0$ if necessary, $\operatorname{dist}(z,\widehat{\mathcal{S}}_\tau)\geq d_0$ implies that $\operatorname{dist}(z,\mathcal{S})\geq d_0/2$ for all $\tau \in (1-\delta,1+\delta)$. Let $R>0$ be large but fixed. By \eqref{Q minus Qhat lower bounded} (which uses the fact that $Q$ is regular) and \eqref{lol63},
decreasing $\delta>0$ further if necessary, we obtain
\begin{align*}
Q(z)-\widehat{P}_\tau(z)
&\geq
C_0\min\{1,d_0^2/4\}
-\sup_{|w|\leq R}
|\widehat{P}_\tau(w)-\widehat Q(w)|
>0,
\end{align*}
for $|z|\leq R$,
$\operatorname{dist}(z,\widehat{\mathcal{S}}_\tau)\geq d_0$ and $\tau \in (1-\delta,1+\delta)$.

The growth condition \eqref{growth of $Q$}
and the expansion
\begin{align*}
\widehat{P}_\tau(z)=2\tau\log|z|+O(1),
\qquad |z|\to+\infty,
\end{align*}
give $\widehat{P}_\tau(z)<Q(z)$ for all $|z|>R$ and $\tau \in (1-\delta,1+\delta)$ (after increasing $R>0$ and decreasing again $\delta$ if necessary). Thus
\begin{align}\label{lol64}
\widehat{P}_\tau\leq Q\quad\mbox{on }\C,
\qquad
\widehat{P}_\tau=Q\quad\mbox{on }\widehat{\mathcal{S}}_\tau.
\end{align}
Since $\widehat{\mu}_\tau(\C)=\tau$, \eqref{lol64}
implies by \cite[Theorem I.3.3]{SaTo} that
$\widehat{\mu}_\tau/\tau$ is the equilibrium measure
associated with $Q/\tau$. Therefore
\begin{align*}
\widehat{\mathcal{S}}_\tau=\mathcal{S}_\tau,
\qquad
\widehat{\phi}_\tau=\phi_\tau,
\qquad
\widehat{\chi}_s=\chi_s.
\end{align*}
Recall also from the text surrounding \eqref{lol31}--\eqref{lol35} that
\begin{align*}
& \mathcal{Q}_\tau =\mathcal{H}_\tau[Q], \qquad H_\tau=\mathcal{H}_\tau
\bigg[\frac14\log\frac{\Delta Q}{4}\bigg], \qquad \mathcal{B}_\tau=2^{-1/4}\pi^{-3/4}e^{H_\tau},
\end{align*}
where the operator $\mathcal{H}_\tau$ is defined in \eqref{def of mathcalH tau}. Hence $\phi_\tau$, $\mathcal{Q}_\tau$, $H_\tau$ and $\mathcal{B}_\tau$ are holomorphic on $U$ and jointly real-analytic in
$(\re z,\im z,\tau)$ on $U\times(1-\delta,1+\delta)$ for some open neighborhood $U$ of $\partial \mathcal{S}$.
\end{proof}

\section{Consistency checks for Theorem \ref{thm:Charlier}}\label{appendix:consistency checks}

In this section, we provide two consistency checks for Theorem \ref{thm:Charlier}.

\subsection{Elliptic Ginibre kernel}

The elliptic Ginibre kernel $K_{n}(z,w)$ is given by \eqref{def of Kn} with $Q(z)=|z|^2-\tau \re z^{2}$ for some $\tau \in [0,1)$. In this setting, the subleading term in the large $n$ asymptotics of $K_{n}$ near the edge is explicitly known: the case $\xi=\eta$ was first established in \cite{LR2016}, then extended in \cite[Proposition 3.5]{ByunEbke} to distinct but bounded $\xi,\eta$, and further generalized in \cite[Proposition 5.1]{Molag2023} to the regime of mildly growing $\xi$ and $\eta$. These results allow for a direct verification of Theorem \ref{thm:Charlier} in the elliptic Ginibre case (note that in this subsection we use a different normalization for $Q$ than in the introduction, to be consistent with \cite{Molag2023}). Let $c_{n}$ be the unimodular factor from \cite[Proposition 5.1]{Molag2023}, and let $\tilde{c}_{n}$ be given by
\begin{align*}
\tilde{c}_{n}(\xi) = \tilde{c}_{n}(\xi;z_{0})  = e^{\frac{2\pi i \tilde{\kappa}}{\sqrt{n}}\big((\im \xi)^{3}-3 \, \im \xi \, (\re \xi)^{2}\big)}, \qquad \tilde{\kappa} := \frac{\sqrt{2}}{3\pi}\kappa,
\end{align*}
where $\kappa=\kappa(z_{0})$ is the curvature of $\partial \mathcal{S}$ at $z_{0}$. Using \cite[Proposition 5.1 with $u=\sqrt{2}\xi$ and $v=\sqrt{2}\eta$]{Molag2023}, and multiplying both sides by $\tilde{c}_{n}(\xi)\overline{\tilde{c}_{n}(\eta)}$, yields 
\begin{multline}
c_{n}(\xi)\overline{c_{n}(\eta)}\tilde{c}_{n}(\xi)\overline{\tilde{c}_{n}(\eta)}K_{n}\bigg(z_{0}+\mathrm{n} \frac{\sqrt{2}\xi}{\sqrt{\Delta Q(z_{0})n/4}},  z_{0}+\mathrm{n} \frac{\sqrt{2}\eta}{\sqrt{\Delta Q(z_{0})n/4}}\bigg) \\
= \frac{n}{\pi}\mathrm{k}(\xi,\eta) + \tilde{\kappa}\sqrt{n}\bigg\{ \frac{e^{-2 (\re \xi)^{2}-2 (\re \eta)^{2}}}{2\sqrt{\pi}e^{2i(\re \xi \, \im \xi - \re \eta \, \im \eta)}}\big( 2\xi^{2}+2\overline{\eta}^{2} - 2\xi \overline{\eta} -1 \big) \\ + 2i \mathrm{k}(\xi,\eta)\big( (\im \xi)^{3} - (\im \eta)^{3} + 3 (\re \eta)^{2} \im \eta - 3(\re \xi)^{2}\im \xi \big) \bigg\} + o(n^{\epsilon}) \label{lol36}
\end{multline}
as $n\to + \infty$ uniformly for $|\xi|,|\eta| \leq M\sqrt{\log n}$, and where $\epsilon>0$ and $M>0$ are arbitrary but fixed. It is now straightforward to check that the subleading term in \eqref{lol36} can be written as $\mathrm{k}_{2}(\xi,\eta;z_{0})\sqrt{n}$, where $\mathrm{k}_{2}(\xi,\eta;z_{0})$ is as in Theorem \ref{thm:Charlier} with $\gamma_{2} = \gamma_{4} = \gamma_{5} = \gamma_{6} = \beta_{4} = \beta_{5} = 0$ and
\begin{align*}
& \gamma_{1} = 2\tilde{\kappa}, \quad \gamma_{3} = - 6 \tilde{\kappa}, \quad \gamma_{8} = - \gamma_{7} = \tilde{\kappa}, \quad \beta_{3} = -\beta_{1} = 3\tilde{\kappa}, \quad \beta_{2} = -6i\tilde{\kappa}.
\end{align*}

\subsection{Kernels of rotation invariant potentials}

Let $K_{n}(z,w)$ be given by \eqref{def of Kn}, where $Q$ is as in Theorem \ref{thm1} and such that $Q(z) = Q(|z|)$ for all $z\in \C$. The result \cite[Theorem 1.11 with $N=0$, $b_{N}=z_{0}$, and with $t$ replaced by $2t$]{ACC2022} yields the following asymptotics for the kernel along the diagonal $\xi = \eta = t \in \R$:
\begin{multline}
K_{n}\bigg(z_{0}+\mathrm{n} \frac{\sqrt{2}t}{\sqrt{\Delta Q(z_{0})n/4}},  z_{0}+\mathrm{n} \frac{\sqrt{2}t}{\sqrt{\Delta Q(z_{0})n/4}}\bigg) = \frac{n \Delta Q(z_{0})}{4\pi} \mathrm{k}(t,t) - \frac{\sqrt{n \Delta Q(z_{0})}}{12\pi z_{0} \sqrt{2}} \\
\times \bigg\{ \frac{e^{-4 t^{2}}}{2\sqrt{\pi}} \bigg( \frac{\partial_{\mathrm{n}}\Delta Q(z_{0})}{\Delta Q(z_{0})} z_{0}(5+8t^{2}) + 4-8t^{2} \bigg) - 12 z_{0} \frac{\partial_{\mathrm{n}}\Delta Q(z_{0})}{\Delta Q(z_{0})} t \, \mathrm{k}(t,t) \bigg\} + \bigO\big((\log n)^{4}\big) \label{lol37}
\end{multline}
as $n\to + \infty$ uniformly for $|t| \leq \log n$. It is easy to check that the subleading term in \eqref{lol37} can be written as $\mathrm{k}_{2}(\xi,\eta;z_{0})\sqrt{n}$, where $\mathrm{k}_{2}(\xi,\eta;z_{0})$ is as in Theorem \ref{thm:Charlier} with $\beta_{1}=\beta_{2}=\beta_{3}=\beta_{4}=\beta_{5} = \gamma_{2} = \gamma_{3} = \gamma_{4} = \gamma_{5} = 0$ and
\begin{align*}
& \gamma_{6} = \frac{\partial_{\mathrm{n}}\Delta Q(z_{0})}{2\pi \sqrt{2 \Delta Q(z_{0})}}, & & \gamma_{7} = -\frac{4 \Delta Q(z_{0}) + 5 z_{0}\partial_{\mathrm{n}}\Delta Q(z_{0})}{12\pi z_{0}\sqrt{2\Delta Q(z_{0})}}, & & \gamma_{8} = -\frac{\gamma_{1}}{4} = \frac{z_{0}\partial_{\mathrm{n}}\Delta Q(z_{0}) - \Delta Q(z_{0})}{6\pi z_{0}\sqrt{2\Delta Q(z_{0})}}.
\end{align*}
\section{The function $\mathrm{k}(\xi,\eta)$ is bounded for $\xi,\eta\in \C$}\label{appendix:k bounded}

Recall from Theorem \ref{thm:Charlier} that $\mathrm{k}$ is given by
\begin{align}\label{lol52}
\mathrm{k}(\xi,\eta) = e^{2\xi \overline{\eta}-(|\xi|^{2}+|\eta|^{2})}\frac{\mathrm{erfc}(\xi+\overline{\eta})}{2}.
\end{align}
By \cite[7.12.1]{NIST}, as $\xi+\overline{\eta} \to \infty$ with $|\arg (\xi+\overline{\eta})| \leq \frac{2\pi}{3}$, we have
\begin{align*}
& \erfc (\xi+\overline{\eta}) = e^{-(\re \xi + \re \eta)^{2}}e^{(\im \xi-\im \eta)^{2}}\frac{e^{-2i (\im\xi - \im \eta)(\re \xi + \re \eta)}}{\sqrt{\pi}(\xi+\overline{\eta})}\bigg(1+\bigO\bigg(\frac{1}{(\xi+\overline{\eta})^{2}}\bigg)\bigg),  
\end{align*}
while if $\xi+\overline{\eta} \to \infty$ with $\frac{\pi}{3} \leq |\arg (\xi+\overline{\eta})| \leq \pi$, then
\begin{align*}
& \erfc (\xi+\overline{\eta}) = 2+ e^{-(\re \xi + \re \eta)^{2}}e^{(\im \xi-\im \eta)^{2}}\frac{e^{-2i (\im\xi - \im \eta)(\re \xi + \re \eta)}}{\sqrt{\pi}(\xi+\overline{\eta})}\bigg(1+\bigO\bigg(\frac{1}{(\xi+\overline{\eta})^{2}}\bigg)\bigg).
\end{align*}
Since $|e^{2\xi \overline{\eta}-(|\xi|^{2}+|\eta|^{2})}|=e^{-(\re \xi-\re \eta)^{2}-(\im \xi-\im \eta)^{2}}$ and
\begin{align*}
|e^{2\xi \overline{\eta}-(|\xi|^{2}+|\eta|^{2})}|e^{-(\re \xi + \re \eta)^{2}}e^{(\im \xi-\im \eta)^{2}} = e^{-2(\re \xi)^{2}-2(\re \eta)^{2}},
\end{align*}
we thus have
\begin{align}\label{lol53}
\mathrm{k}(\xi,\eta) = \bigO \bigg( e^{-(\re \xi-\re \eta)^{2}-(\im \xi-\im \eta)^{2}} + \frac{e^{-2(\re \xi)^{2}-2(\re \eta)^{2}}}{|\xi+\overline{\eta}|} \bigg), \qquad \mbox{as } \xi + \overline{\eta} \to \infty
\end{align}
in any sector. Let $M_{1},M_{2}>0$ be some constants. It follows from \eqref{lol53} that $M_{1}$ can be chosen large enough so that $\mathrm{k}(\xi,\eta)$ remains bounded for $\{(\xi,\eta)\in \C^{2}: |\xi+\overline{\eta}|\geq M_{1}\}$. On the other hand, if $\xi, \eta \to \infty$ in such a way that $\xi + \overline{\eta}$ remains bounded, then, by \eqref{lol52} and $|e^{2\xi \overline{\eta}-(|\xi|^{2}+|\eta|^{2})}|=e^{-(\re \xi-\re \eta)^{2}-(\im \xi-\im \eta)^{2}}$, we have
\begin{align*}
\mathrm{k}(\xi,\eta) = \bigO(e^{-(\re \xi-\re \eta)^{2}-(\im \xi-\im \eta)^{2}}).
\end{align*}
Thus we can choose $M_{2}>0$ large enough so that $\mathrm{k}(\xi,\eta)$ remains bounded on $\{(\xi,\eta)\in \C^{2}: |\xi|\geq M_{2}, |\xi+\overline{\eta}|\leq M_{1}\}$. Finally, since $\mathrm{k}(\xi,\eta)$ is smooth on the compact $\{(\xi,\eta)\in \C^{2}: |\xi|\leq M_{2}, |\xi+\overline{\eta}|\leq M_{1}\}$, it is also bounded there. Hence $\mathrm{k}(\xi,\eta)$ is bounded for $\xi,\eta\in \C$.

 
\paragraph{Acknowledgements.} 
The author is grateful to Aron Wennman for numerous discussions regarding Appendix \ref{appendix:Proof of edge} and some results from \cite{HM2013, HS2002, HW2021}, as well as to the two anonymous referees for many excellent remarks, in particular for suggesting that Corollary \ref{col1} should also include the angles of the smallest gaps, for pointing out that \eqref{eqn2} should be established for finite unions of rectangles, and for requesting clarifications in Appendix
\ref{appendix:Proof of edge} that led to the proof of
Lemma \ref{lemma:joint regularity}. The author also thanks Yacin Ameur, Paul Bourgade, Sung-Soo Byun, Leslie Molag and Dinghao Zhu for useful remarks. The author is a Research Associate of the Fonds de la Recherche Scientifique - FNRS, and also acknowledges support from the European Research Council (ERC), Grant Agreement No. 101115687. 

\paragraph{AI use disclosure.} A first version of this paper was made publicly available on \texttt{arXiv} in August 2025. All proofs and the introduction were written without using AI tools. After the paper was formally accepted for publication, \texttt{GPT-6 Astra} was used to check the entire manuscript. This check led only to minor corrections.

\end{document}